\documentclass[10pt]{article}
\usepackage[portrait,margin=2.54cm]{geometry}

\usepackage{hyperref}
\usepackage{amssymb,amsthm}
\usepackage[textsize=small]{todonotes}
\usepackage{graphicx}
\usepackage{upref,setspace}
\usepackage{xcolor,colortbl}
\usepackage{enumerate}
\usepackage{bbm}
\usepackage{float}
\usepackage{bm}
\usepackage[mathscr]{eucal}
\usepackage{graphicx}
\usepackage{sidecap}
\usepackage{latexsym}
\usepackage{psfrag}
\usepackage{epsfig}
\usepackage{enumerate}
\usepackage{amsmath}
\usepackage{amsthm}
\usepackage{amssymb}
\usepackage{multirow}
\usepackage{comment}
\usepackage{amsfonts}
\usepackage{color}
\usepackage[textsize=small]{todonotes}
\usepackage{hyperref}
\usepackage{mathtools}
\mathtoolsset{showonlyrefs}
\usepackage{enumitem}

\numberwithin{equation}{section}
\numberwithin{figure}{section}
\numberwithin{table}{section}
\newtheorem{theorem}{Theorem}[section]
\newtheorem{corollary}[theorem]{Corollary}
\newtheorem{lemma}[theorem]{Lemma}
\newtheorem{construction}[theorem]{Construction}
\newtheorem{property}[theorem]{Property}
\newtheorem{defn}[theorem]{Definition}
\newtheorem{assumption}[theorem]{Assumption}
\newtheorem{proposition}[theorem]{Proposition}

\theoremstyle{remark}
\newtheorem{remark}[theorem]{Remark}
\newtheorem{example}[theorem]{Example}

\numberwithin{equation}{section}

\newcommand{\veps}{\varepsilon}

\newcommand{\PP}{\mathbb{P}}
\newcommand{\EE}{\mathbb{E}}
\newcommand{\RR}{\mathbb{R}}

\newcommand{\NN}{\mathbb{N}}

\newcommand{\clv}{\mathcal{V}}
\newcommand{\clu}{\mathcal{U}}
\newcommand{\clm}{\mathcal{M}}

\newcommand{\clr}{\mathcal{R}}
\newcommand{\cla}{\mathcal{A}}
\newcommand{\clb}{\mathcal{B}}

\newcommand{\cld}{\mathcal{D}}
\newcommand{\clg}{\mathcal{G}}
\newcommand{\clf}{\mathcal{F}}
\newcommand{\clk}{\mathcal{K}}

\newcommand{\clj}{\mathcal{J}}
\newcommand{\cle}{\mathcal{E}}

\newcommand{\clp}{\mathcal{P}}
\newcommand{\cli}{\mathcal{I}}

\newcommand{\bfL}{\mathbf{L}}

\newcommand{\om}{\omega}
\newcommand{\Om}{\Omega}
\newcommand{\tplus}{A_+}
\newcommand{\barLP}{\bar{L}_{\mbox{\tiny{pair}}}}

\newcommand{\MP}{M_{\mbox{\tiny{pair}}}}
\newcommand{\cbLP}{\check{\bfL}_{\mbox{\tiny{pair}}}}
\newcommand{\newgamma}{\gamma}
\newcommand{\bdelta}{\bm{\delta}}

\definecolor{expcol}{rgb}{0.5,0.1,1.0}
\definecolor{ecol}{rgb}{0.0, 0.0, 1.0}

\newcommand{\AB}[1]{{\color{black} #1}}
\newcommand{\PZ}[1]{{\color{black} #1}}

\begin{document}

\title{Large Deviations  for Empirical Measures of  Self-Interacting Markov Chains}
\author{Amarjit Budhiraja, Adam Waterbury, Pavlos Zoubouloglou}
\maketitle

\begin{abstract}
Let $\Delta^o$ be a finite set and, for each probability measure $m$ on $\Delta^o$, let $G(m)$ be a transition kernel on $\Delta^o$. Consider the sequence $\{X_n\}$ of $\Delta^o$-valued random variables such that,  given $X_0,\ldots,X_n$, the conditional distribution of $X_{n+1}$ is $G(L^{n+1})(X_n,\cdot)$, where $L^{n+1}=\frac{1}{n+1}\sum_{i=0}^{n}\bdelta_{X_i}$. Under conditions on $G$ we establish a large deviation principle for the sequence $\{L^n\}$. As one application of this result we obtain large deviation asymptotics for the Aldous-Flannery-Palacios (1988) approximation scheme for quasi-stationary distributions of finite state Markov chains. The conditions on $G$ cover other models as well, including certain models with edge or vertex reinforcement.\\

 \end{abstract}

\section{Introduction}

In this work we are interested in the large deviations behavior of certain types of self-interacting Markov chains. The terminology `Markov chain' is in fact a misnomer as these processes are very far from being Markovian, {since} the conditional law of the state at the next time instant, given the past, depends on the whole history of the process through its empirical distribution. {Despite the fact that such processes are non-Markovian, the term self-interacting Markov chain is widely used to describe them (cf. \cite{delmorarnjasand, delmordou2010annals}).  Other terminology, such as stochastic chains with occupational self-interactions \cite{delmormic2003} has also  been used to describe these processes.  
} The general setting {considered in this work} is as follows. Consider a finite set $\Delta^o \doteq \{1, \ldots , d\}$ and let $G$ be a map from $\clp(\Delta^o)$ (the space of probability measures on $\Delta^o$) to the space
$\clk(\Delta^o)$ of transition probability kernels on $\Delta^o$. Fix $x_0 \in \Delta^o$ and let
$\{X_n, \; n \in \NN_0\}$ be a sequence of $\Delta^o$-valued random variables defined recursively as follows: $X_0 = x_0$, and given $X_0, \ldots , X_n$, the conditional law of $X_{n+1}$  is $G(L^{n+1})(X_n, \cdot)$, where
$L^{n+1} \doteq \frac{1}{n+1} \sum_{i=0}^n \bdelta_{X_i}$ is the empirical measure at time instant $n$.
Many types of reinforced  stochastic dynamical  systems fall within this framework and such processes arise in several different contexts, e.g., Monte-Carlo methods for quasi-stationary distributions \cite{AFP88, benclo, blanchet2016}, population growth models in mathematical ecology \cite{sch01, sinliv}, self organization in dynamical models of social networks, models for random monopolies in economics, models for neuron growth, bandit problems in sequential analysis, generalized P\'{o}lya's urn models, and many others; see the excellent survey by Pemantle \cite{Pemantle05} for discussion of these diverse applications.  Using techniques from stochastic approximation theory and branching processes,  under suitable conditions on $G$, law of large numbers and central limit results for the empirical measure sequence $\{L^n, \; n \in \NN\}$ have been studied in various works \cite{AFP88, benclo,BenClo18,Pemantle05}.

The goal of the current work is to establish a large deviation principle (LDP) for the sequence $\{L^n, \; n \in \NN\}$ under a broad set of conditions on the map $G$. Our main result, Theorem \ref{thm:ldp1}, provides a large deviation  upper bound only requiring that the map $G$ is Lipschitz (Assumption \ref{assu:lip}). Furthermore, this theorem shows that under a stronger condition (Assumption \ref{ass1}) the  matching large deviation lower bound holds as well, thus establishing a LDP for $\{L^n, \; n \in \NN\}$. Assumption \ref{ass1} imposes four main conditions on the model: {the} first condition says that $G$ is an affine map; the second condition imposes a natural communicability structure on the transition probability matrix $G(m)$ for $m \in \clp(\Delta^o)$; the third condition requires that the fixed point equation $\pi^* G(\pi^*) = \pi^*$  admits a  strictly positive solution in  $\clp(\Delta^o)$; and, finally, the fourth condition says that the empirical measure $L^n$ eventually charges all points in $\Delta^o$, a.s.
As discussed in Example \ref{ex:operatorexamples}, Remark \ref{rem:conditions1}, and Section \ref{sec:examp}, these conditions are satisfied for many interesting settings.

{One such setting} is the reinforced Markov chain Monte-Carlo scheme for approximating quasi-stationary distributions (QSD) of finite state Markov chains that was introduced in the work of Aldous {\em et al.}\cite{AFP88}. For an overview of QSD, see \cite{qsdbook} and see Example \ref{ex:introduceqsdalg1} for a precise definition of a QSD. Let $P$ be the transition kernel of a $\Delta = \Delta^o \cup \{0\}$-valued Markov chain  that is absorbed at $0$, and consider the substochastic kernel $P^o$ obtained by restricting $P$ to $\Delta^o$. Suppose that $P^o$ is irreducible. Then, there is a unique QSD of $P$ which is characterized as 
 the normalized Perron-Frobenius eigenvector of $P^o$. The QSD captures the long-term pre-absorption behavior of the Markov chain with transition kernel $P$,  consequently,  QSD are widely used to understand metastability behavior of stochastic systems in ecology and biology \cite{bucpol, gos2, gylsil}, chemical kinetics \cite{parpol, pol4}, epidemiology \cite{andbri, artecolop1, artecolop2}, and other fields. In particular  numerical approximation of QSD is of significant interest. Various numerical methods have been proposed to approximate QSD,  and one important family of methods are described in terms of self-interacting chains \cite{AFP88, benclo, blanchet2016, budfrawat2}. 
The precise description of this approximation scheme is recalled in Example \ref{ex:operatorexamples}; here we merely note that the scheme corresponds to simulating a self-interacting Markov chain for which the function $G$ is given as
$G(m)(x,y) \doteq P(x,y) + P(x,0)m(y)$, for $x,y \in \Delta^o$,  and 
$m \in \clp(\Delta^o)$.
The law of large numbers (LLN) for the empirical measure sequence $\{L^n, \; n \in \NN\}$ associated with this Monte-Carlo method giving a.s. convergence to the QSD has been  established in \cite{AFP88, benclo}. Under exactly the conditions  for the LLN, the current work establishes a LDP for this sequence. Beyond this example, as discussed in Section \ref{sec:examp}, the Assumptions of Theorem \ref{thm:ldp1} cover many other types of self-interacting Markov chains as well, including certain variants of edge reinforced and vertex reinforced random walks, a type of personalized PageRank algorithm, and certain generalized P\'{o}lya urn schemes.

{We  now discuss some related literature on large deviations, specifically focusing on self-interacting processes and related urn models.} The model that we consider can be formulated as a type of urn model. Large deviations for a family of urn models (that are very different from the one considered here) have been studied in \cite{dupuis2004large}. The paper \cite{bryc2009large}  studies large deviations associated with a
preferential attachment random graph by viewing it as a  special type of an urn model. In the case $d=2$, large deviations for certain generalized P\'{o}lya urns have been studied in \cite{Franchini17}. For a special choice of the
 `urn function' in \cite{Franchini17},  this model reduces to a model of the form considered in the current work with $d=2$ and $G(m)(x,y) = (mP^o)(y)$, $x,y \in \Delta^o$, $m \in \clp(\Delta^o)$ where $P^o$ is a $2\times 2$ 
 transition probability matrix with strictly positive entries. Large deviation principle for a similar model but with a general $d$ was recently studied in \cite{BudWat3} under the condition that $P^o(x,y)> 0$ for all $x,y \in \Delta^o$. The  proofs  in the latter paper  also use   stochastic control representations as in the current work, however the arguments there are significantly simpler due to fact that $G(m)(x,y)$ does not depend on $x$; in particular the main technical challenge of  time varying  equilibria does not arise in \cite{BudWat3}. In fact, as a corollary of the current work we obtain a substantial extension of the result in \cite{BudWat3} where the condition $P^o(x,y)> 0$ for all $x,y \in \Delta^o$ is relaxed to the requirement that $P^o$ is an irreducible   transition probability matrix (see Section \ref{sec:examp}, Example \ref{ex:oldmodel}).
Our results also cover certain types of edge reinforced random walks (see Section \ref{sec:examp}, Example \ref{ex:edgereinforced}). Some results on large deviations for specific  kinds 
of  edge reinforced random walks (once-reinforced random walks) can be found in
 \cite{HuaLiuKai} and \cite{Zhang14}. \AB{Finally, the very recent article \cite{GuiNecWu24}, which appeared after the current paper was submitted for publication, establishes a large deviation principle for the occupation measure, for a continuous time strong Feller process, conditioned not to exit a given subregion.
 The results are applied to
several stochastic processes such as the solutions of elliptic stochastic
diﬀerential equations driven by a rotationally invariant $\alpha$-stable process, the kinetic
Langevin process, and the overdamped Langevin process driven by a Brownian motion. Techniques and motivations for this work are different from the current work.
}

 In the special case where $G(m) = P^o$ is independent of $m$, the LDP in the current work reduces to the classical empirical measure LDP for finite state Markov chains \cite{donvar1,  donvar3}. 
As is well known, in this case the rate function takes the following simple form 
	\begin{equation} \label{eq:donvarrf} \tilde I(m) = \inf_{ \lambda \in \cli(m)} R(\lambda \| m \otimes P^o), \; m \in \clp(\Delta^o),\end{equation}
	where $\cli(m) \doteq \{ \lambda \in \clp(\Delta^o\times \Delta^o): \lambda(\Delta^o \times \cdot) = \lambda( \cdot\times \Delta^o )=m(\cdot)\}$,
 $m \otimes P^o \in \clp(\Delta^o\times \Delta^o)$ is defined as $m \otimes P^o(x,y)
 = m(x)P^o(x,y)$, $x,y \in \Delta^o$, and $R$ is the relative entropy function.
{This rate function can be interpreted as follows. Fix $m \in \clp(\Delta^o)$ and  consider $\lambda \in \cli(m)$ achieving the infimum in \eqref{eq:donvarrf}. We can disintegrate $\lambda$ as $\lambda(x,y ) = m \otimes {Q}(x,y),$
where $m(\cdot) \doteq \lambda_{(1)}(\cdot)$ is the first marginal of $\lambda$
and $Q(x,y)$ gives the conditional probability that the second coordinate is $y$ given that the first coordinate is $x$.
Since $\lambda \in \cli(m)$, it follows that  $m(y) = \sum\nolimits_{x \in \Delta^o} m(x) {Q}(x,y)$. That is,  $m$ is the stationary distribution of a Markov chain with transition kernel ${Q}$. The  rate function in \eqref{eq:donvarrf} can now be interpreted as saying that the most likely way for the empirical measure of $\{X_n, \; n \in \NN_0\}$ to (asymptotically) be close to the measure $m$ is for $\{X_n,\; n \in \NN_0\}$ to behave like a Markov chain with transition kernel ${Q}$.}
 Indeed, this  insight and an appropriate use of the ergodic theorem are the key ingredients in the proof of the large deviation lower bound in this classical setting.
 
 In contrast {to the  Markovian setting}, for the self-interacting  chains considered in the current work, the atypical behaviors for which the empirical measure sequence is asymptotically close to a   given $m \in \clp(\Delta^o)$ are significantly more complex.  Roughly speaking, after a long period of time, the suitably interpolated path constructed from the empirical measure sequence $\{L^n, \; n \in \NN\}$ behaves like a trajectory, with a linear velocity,  that converges to $m$ and whose evolution is governed by certain dynamic local equilibria associated with time-dependent transition probability kernels on $\Delta^o\times \Delta^o$ (cf. \eqref{P2}, \eqref{P3}). The instantaneous local averaging that is manifested in the form of the rate function is somewhat akin to the forms of rate functions for large deviations from stochastic averaging principles for multiscale stochastic dynamical systems \cite{buddupgan2, dupspi, ver3}.
	This atypical behavior that produces a given $m \in \clp(\Delta^o)$ can be seen from the definition of the rate function ${I_A}$ in \eqref{def:ratefunction1}, which is described in terms of time-reversal of such linear paths, so that the convergence to $m$ at $\infty$ is replaced with  the initial condition on the path {to start at $m$}. The  variational formula for the rate function involves an exponential discount factor which arises due to the natural time interpolation that is associated with the discrete evolution of $L^n$ with steps of sizes $1/(n+1)$ (see \eqref{eq:barlnk}). Such time interpolation is quite standard in the asymptotic analysis of stochastic approximation schemes \cite{benSA, BMP12, bor09, kus03} and indeed a discounted cost has been previously observed in a rate function for certain large deviation problems arising from some stochastic approximation schemes with Gaussian noises \cite{kushner1984asymptotic}. We remark that in the special case {discussed in the previous paragraph, namely when} $G(m) = P^o$, the rate function in \eqref{def:ratefunction1} is easily seen to reduce to the classical formula in \eqref{eq:donvarrf} (see Example \ref{ex:sec8example1} in Section \ref{sec:examp}). We also note that the natural analogue of $\tilde I$ in the general self-interacting setting, defined as 
 \begin{equation}\label{eq:1230nn}
 \tilde I(m) = \inf_{ \lambda \in \cli(m)} R(\lambda \| m \otimes G(m)), \; m \in \clp(\Delta^o),
 \end{equation}
 satisfies the inequality ${I_A}(m) \le \tilde I(m)$, $m \in \clp(\Delta^o)$; see Remark 
 \ref{rem:ivstili}.

We now make some comments on proof techniques. The basic idea is to use stochastic control representations for  Laplace functionals of the form in \eqref{eq:varrep} \cite{buddupbook, dupell}.
Using this variational formula, the proof of the upper bound proceeds via natural tightness and weak convergence arguments.
The main challenges and novelty are in the proof of the large deviation lower bound and so we limit our remarks to this inequality. The basic idea is to choose a near-optimal control $\eta$, and the corresponding trajectory $M$ given through \eqref{P3},  in the variational formula for the rate function ${I_A(m)}$ in \eqref{def:ratefunction1}, and then construct controlled empirical measures as in \eqref{eq:barlnk} that suitably approximate $M$ and for which the associated cost, as given by the second term on the right side of \eqref{eq:varrep}, appropriately approximates the cost associated with $\eta$ in \eqref{def:ratefunction1}. However, such a construction for an arbitrary near-optimal control $\eta$ appears quite daunting, mainly due to the local equilibrium property \eqref{P2} that the constructed stochastic controls are required to achieve asymptotically. In order to handle this, we proceed by a series of approximations that lead to a `well behaved' simple form near-optimal control that is more tractable for a suitable construction of controlled empirical measures. This is the main content of Section \ref{sec:lowbd}.  Next, in Section \ref{sec:llb} we proceed to the construction of controlled empirical measures that are designed to suitably approximate the simple form near-optimal path constructed in Section \ref{sec:lowbd}. This construction and proof of convergence are technically the most involved part of the proof. Detailed discussion of the construction can be found at the start of Section \ref{sec:llb}, but at a high level the idea is to employ the ergodic theorem in a dynamic fashion to successively approximate all the local equilibria that make up the simple form control $\eta$ using suitably controlled empirical measures in such a manner that the associated costs also have the correct asymptotic behavior.

\subsection{Notation} \label{sec:not} In this section we introduce some notation that is used throughout this work.
Fix $d \in \NN$, and let   $\Delta^o = \{1, \ldots , d\}$. For a metric space $S$, $\clb(S)$ denotes the corresponding Borel $\sigma$-field and $\clp(S)$ denotes the space of probability measures on $(S, \clb(S))$ equipped with the topology of weak convergence. When $S$ is a finite set, we let $\clp_{+}(S) \doteq \left\{ m \in \clp(S ): \min_{x \in S} m(x) > 0  \right \}$.
Recall that a function $I: S \to [0, \infty]$ is called a rate function if it has compact sublevel sets, namely
$S_k \doteq \{ x \in S: I(x) \le k\}$ is compact for every $k \in [0, \infty)$. For $x \in S$, $\bdelta_x \in \clp(S)$ denotes the Dirac probability measure concentrated at  $x$. For a probability measure $\eta$ on $S_1 \times S_2 \times S_3$, $\eta_{(i)}$ denotes the marginal distribution of $\eta$ on $S_i$, $i = 1,2, 3$, and for $i < j$,  $\eta_{(i,j)}$ denotes the marginal distribution of $\eta$ on $S_i \times S_j$. Similar notation is used for probability measures on other product spaces. \AB{Given probability measures $\mu_1, \mu_2$ on $S_1$ and $S_2$, respectively, $\mu_1\otimes \mu_2$ denotes the probability measure on $S_1 \times S_2$ characterized by $\mu_1 \otimes \mu_2 (A \times B) = \mu_1(A)\mu_2(B)$, where $A \in \clb(S_1)$ and $B \in \clb(S_2)$.}  \AB{Given a probability measure $\theta$ on $S$, and a transition probability kernel $\gamma: S \times \clp(S) \to [0,1]$, $\theta \otimes \gamma$ denotes a probability measure on $S\times S$ characterized as 
 $\theta\otimes \gamma (A\times B) = \int_A \gamma(x,B) \theta(dx)$, $A, B \in \clb(S)$.} \AB{If $\mu$ is a probability measure on a finite measurable space $S$, then, for $x \in S$, we often write $\mu(x)$ in place of $\mu(\{x\})$.}
For $\mu, \nu \in \clp(S)$, we denote the relative entropy of $\nu$ with respect to $\mu$
as $R(\nu \|\mu)$, which is the extended real number defined as
\[
R(\nu \|\mu) \doteq \int_S \left(\log \frac{d\nu}{d\mu}\right) d\nu,
\]
if $\nu$ is absolutely continuous with respect to $\mu$, and $+\infty$ otherwise.
Let {$\clv^d \doteq \{ \bdelta_x : x \in \Delta^o\}$ and note that $\clv^d \subset \clp(\Delta^o)$.} For a locally compact space $S$, let $\clm(S)$ denote the space of locally finite measures on $S$ equipped with the vague topology.
 We denote by $C_b(\clp(\Delta^o))$  the space of bounded continuous functions from $\clp(\Delta^o)$ to $\RR$. For $m, \tilde m \in \clp (\Delta^o)$, we write  $\| m - \tilde m\| \doteq \sum_{x\in\Delta^o} |m(x) - \tilde m(x)|$, and we use the same notation for the $L^1$-norm of a vector in $\RR^d$. For a Polish space $S$, $C([0, \infty):S)$
 will denote the space of continuous functions from 
 $[0,\infty)$ to $S$, equipped with the topology of local uniform convergence. 
 As a convention $\int_a^b f(s) ds$ is taken to be $0$ if $a\ge b$ and $\sum_{i=k}^{j}a_i$ is taken to be $0$ if $k > j$. For $v \in \RR^d$ we use  $v_x$ and $v(x)$ interchangeably to denote the $x$-th coordinate of $v$.
 A transition kernel $K$ on a finite set $S$ is a map $K: S\times S \to [0,1]$ such that $\sum_{y \in S}K(x,y)=1$ for all $x \in S$.
  For such a kernel and $x,y \in S$, we use $K_{x,y}$ and $K(x,y)$ interchangeably, and we write $\clk(S)$ to denote the set of transition kernels on $S$. 
  {Note that a $K \in \clk(S)$ can be identified with a $|S|\times |S|$ stochastic matrix, where $|S|$ denotes the cardinality of the set $S$, and with this identification $K_{x,y} = K(x,y)$ is the $(x,y)$-th entry of this matrix.}
  We write $I_d$ to denote the $d\times d$ identity matrix. For a matrix $A$,  {$A_{x,y}$ denotes its $(x,y)$-th entry} and we write $A>0$ to denote that all of its entries are strictly positive. Finally, we write $\RR_+$ to denote $[0,\infty)$.

  \subsection{Outline}

This paper is organized as follows.  In Section \ref{sec:mainresults} we introduce the model of interest, state our main large deviation result (Theorem \ref{thm:ldp1}), and provide one basic example that motivates this study. In Section \ref{sec:controlproc} we present the stochastic control representation that is used  in the proof of Theorem \ref{thm:ldp1}; both in proving the large deviation upper bound and lower bound. The large deviation upper bound is  proved in Section \ref{sec:lub}. In Section \ref{sec:lowbd}, through appropriate perturbation, mollification, and discretization, we construct  simple form near-optimal trajectories and controls that are tractable for constructing suitable controlled empirical measures for the proof of the large deviation lower bound. In Section \ref{sec:llb} we proceed with 
this construction and provide the proof of the convergence of the controlled processes and costs, which finishes the proof of the large deviation lower bound. In Section \ref{sec:levelset} we show  that the function ${I_A}$ introduced in \eqref{def:ratefunction1} is a rate function, namely it  has compact sublevel sets. Finally in Section \ref{sec:examp} we present several examples for which the assumptions of Theorem \ref{thm:ldp1} are satisfied.

\section{Setting and Main Result}
\label{sec:mainresults}
\subsection{Description of the Model}\label{sec:modeldescription}
Consider a map $G: \clp(\Delta^o) \to \clk(\Delta^o)$ and
fix $x_0 \in \Delta^o$.  We consider a collection $\{X_n, \; n \in \NN_0\}$ of $\Delta^o$-valued random variables, a collection $\{L^n,\;  n \in \NN\}$ of $\mathcal{P}(\Delta^o)$-valued random measures, and a filtration $\{\mathcal{F}_n, \; n \in \NN_0\}$ on some probability space $(\Omega,\mathcal{F}, P)$, defined recursively as follows. Let $X_0  \doteq x_0$, $\mathcal{F}_0 \doteq \{\emptyset, \Omega\}$,  and $L^1 \doteq \bdelta_{x_0}$. Having defined 
$\{X_i, L^{i+1},\;  0\le i \le n\}$
and $\sigma$-fields $\{\clf_i, \; i \le n\}$ for some $n \in \NN_0$, define 
\begin{equation}\label{eq:rmcdynamic1}
P(X_{n+1} = y \mid \mathcal{F}_n) \doteq G(L^{n+1})_{X_n,y},\; \; y \in \Delta^o,
\end{equation}
 $\mathcal{F}_{n+1} \doteq \sigma\{X_k, \; k \leq n+1\}$,
and
\begin{equation}\label{eq:occmzr}
L^{n+2} \doteq \frac{1}{n+2} \sum\limits_{i=0}^{n+1} \bdelta_{X_i}. 
\end{equation}
{Thus, $L^{n+1}$ is the empirical measure of the first $n+1$ states of the chain $\{X_k, \; k \in \NN_0\}$ and the first display above says that the conditional law of $X_{n+1}$ given the whole history $\clf_n$ is determined by the most recent state $X_n$ and the empirical measure of the first $n+1$ states of the chain. The dependence of the conditional law on the empirical measure $L^{n+1}$ says in particular that $\{X_n, \; n\in \NN_0\}$ is not a Markov chain in general. Furthermore, noting that 
$$L^{n+2} = \frac{n+1}{n+2} L^{n+1} + \frac{1}{n+2}\bdelta_{X_{n+1}}$$
and the fact that the conditional law of $X_{n+1}$ given $\clf_n$ is a function of not only $L^{n+1}$ but also $X_n$ says that the chain $\{L^n, \; n\in \NN\}$ of $\clp(\Delta^o)$-valued random variables is not Markovian either. Nonetheless,
it is easy to see that the sequence $\{(X_n, L^{n+1}), \; n \in \NN_0\}$ is indeed a Markov chain with values in $\Delta^o\times \clp(\Delta^o)$, whose transition kernel is inhomogeneous in time (cf \eqref{eq:defL}). This sequence can be given a pathwise representation using certain $\clv^d$-valued random fields indexed by $(x,m) \in \Delta^o\times \clp(\Delta^o)$. This representation, which is given at the start of Section \ref{sec:controlproc}, will be used throughout in our proofs. }
{Note that for each $m \in \clp(\Delta^o)$, $G(m)$ is a transition kernel on $\Delta^o$, 
 which can be identified with a $d\times d$ stochastic matrix. Recall that $G(m)_{x,y}$ denotes the $(x,y)$-th entry of this stochastic matrix which is occasionally also written as $G(m)(x,y)$.
}

{We now give one basic example of a self-interacting Markov chain to which our results apply.

\begin{example}[{\bf Quasi-Stationary Distributions}]\label{ex:introduceqsdalg1}
Let $\Delta = \Delta^o \cup \{0\}$, $P \in \clk(\Delta)$, and let $\{Y_n, \; n \in \NN_0\}$ be a Markov chain with transition probability kernel $P$. Then $\pi^* \in \clp(\Delta^o)$ is called a quasi-stationary distribution for the chain $\{Y_n, \; n \in \NN_0\}$ if
    \[
    P_{\pi^*}(Y_n = x \mid Y_n \in \Delta^o) = \pi^*_x, \; \; x \in \Delta^o, n \in \NN_0,
    \]
	where $P_{\pi^*}$ denotes the probability measure under which $Y_0$ is distributed as $\pi^*$.
 Suppose that the substochastic matrix $P^o$ defined by
\[
P^o_{x,y} \doteq P_{x,y}, \; \; x,y \in \Delta^o
\]
is irreducible. Then it is known that there is a unique QSD for the Markov chain $\{Y_n, \; n \in \NN_0\}$ \cite{qsdbook}.
In \cite{AFP88} a basic Monte-Carlo method for computing this QSD was introduced.
Define $G: \clp(\Delta^o) \to \clk(\Delta^o)$ as
    \begin{equation}\label{eq:qsdG}
    G(m)_{x,y} \doteq P_{x,y} + P_{x,0}m_y, \;\; x,y \in \Delta^o, m \in \clp(\Delta^o)
    \end{equation}
	and construct $\{X_n, L^{n+1}, \; n\in \NN_0\}$ as in 
 \eqref{eq:rmcdynamic1} and \eqref{eq:occmzr}. The self-interacting Markov chain $\{X_n, \; n \in \NN_0\}$ is a process that evolves according to the transition kernel $P$ until it reaches state $0$. Then, upon reaching state $0$,  the state of the  chain  immediately jumps to a state in $\Delta^o$ according to the probability distribution given by its current empirical occupation measure. Then \cite{AFP88, benclo} show that $L^n$ converges a.s. to the unique QSD $\pi^*$ of the chain $\{Y_n, \; n \in \NN_0\}$. 
\end{example}
}

\subsection{Statement of Results}
{In this section we present our key assumptions and the main result of this work (Theorem \ref{thm:ldp1}). One basic example where our result applies is provided as well. Additional examples are given in Section \ref{sec:examp}.}
\label{sec:statres}
We introduce the following two assumptions on  the operator $G$.
\begin{assumption} \label{assu:lip}
        $\textit{[Lipschitz Continuity]}$ There is  $L_G \in (0,\infty)$ such that for all $m, \tilde{m} \in \clp(\Delta^o)$,
        \begin{equation}\label{eq:246aaa}
         \sum\limits_{x,y \in \Delta^o} | G(m) _{x,y} -G(\tilde{m})_{x,y}| 
		 \le  L_G \|m - \tilde{m}\|.
         \end{equation}
                 \label{ass:contin}
\end{assumption}
The above assumption is the only requirement for the large deviation upper bound. 

A $d\times d$ matrix $A$ is called an \textit{adjacency matrix} if it has entries $0$ or $1$; and it is called an \textit{irreducible adjacency matrix} if 
 for each $x, y \in \Delta^o$ there is a $m \in \NN$ such that {$(A^m)_{x,y}>0$}. 
 {Adjacency matrices such as $A$ will be used to describe the communicability structure of the chain $\{X_n, \; n \in \NN_0\}$. 
 Specifically, we denote by $\cla$ the collection of all adjacency matrices $A$ that have the property that, for each $(x,y) \in \Delta^o \times \Delta^o$, {$A_{x,y}= 0$ implies $G(m)(x,y) =0$ for all 
 $m \in \clp(\Delta^o)$ (Recall from Section 1.2 that we use the notation $G(m)_{x,y}$ and $G(m)(x,y)$ interchangeably)}. Thus for any $A \in \cla$ if any entry $A_{x,y}$ is zero then we must have
 $P(X_n=x, X_{n+1}=y) =0$ for all $n \in \NN_0$ and so such a matrix records the permissible transition states for the chain $\{X_n, \;  n \in \NN_0\}$. Note that the class $\cla$ is nonempty as it contains the matrix $A_1$ consisting of  all ones (in that case the above property is vacuously true).  A matrix $A\in \cla$ will be used to formulate a key communication condition (Assumption \ref{ass1}(2) below) that is needed for the proof of the lower bound. One could simplify this condition and the statement of the large deviation principle by requiring that it hold for the matrix $A_1$ (the matrix with all $1$'s), but that is too restrictive for many interesting settings, e.g. the case where for each $m \in \clp(\Delta^o)$, $G(m)$ is the transition kernel of a nearest neighbor random walk. The condition we formulate allows for such settings where $G(m)_{x,y}$ need not be strictly positive for all $(x,y) \in \Delta^o\times \Delta^o$, e.g. in the Markovian setting where  $G(m) = G_0$ for all $m \in \Delta_0$, this condition simply requires that the transition matrix $G_0$ is irreducible, in which case the matrix $A$ can be simply taken to be $A_{x,y} = \bm{1}_{\{G_0(x,y)\neq 0\}}$, $(x,y) \in \Delta^o\times \Delta^o$.
}
 We write  \begin{equation}\label{eq:deltaplusplus}
            \tplus \doteq \{ (x,y ) \in \Delta^o \times \Delta^o : A_{x,y} = 1 \}.
            \end{equation}

For the lower bound we introduce the following additional assumption. Part 2 of the condition identifies  the `minimal' $A \in \cla$ associated with the map $G$.		
\begin{assumption} \label{ass1}$\,$
    \begin{enumerate}
        \item $\textit{[Linearity]}$
        \label{ass:linear}
        For all $\kappa \in [0,1]$ and $m, \tilde m \in \clp(\Delta^o)$,
        \[ G(\kappa m + (1-\kappa) \tilde m) = \kappa G(m) + (1-\kappa) G(\tilde m).\]
        \item $\textit{[Communication Structure]}$
        There is an irreducible adjacency matrix  $A$ such that the following hold: 
        \label{ass:aux}
        \begin{enumerate}
            \item \label{ass:222a} $A \in \cla$.
            \label{ass:aux0}
            \item There is a  $\delta_0^A \in (0,\infty)$ such that if $(x,y) \in \tplus$, then
            \[
            G(m)_{x,y} \ge \delta_0^A \min_{z \in \Delta^o} m_z, \; \; m \in \clp(\Delta^o).
            \]
            \label{ass:aux1}
        \end{enumerate}
         
       \item$\textit{[Positive Fixed Point]}$
        There is a  $\pi^* \in \clp_+(\Delta^o)$ such that $\pi^* G(\pi^*) = \pi^*$. \label{ass:qsdG}
        \item  $\textit{[Nondegeneracy of Empirical measure]}$
        For every $x \in \Delta^o$
        $$P(\om \in \Om: \mbox{ for some } n \in \NN, L^{n}(\om)(x)>  0)=1.$$
       \label{ass:conv-occ}
\end{enumerate} 
\end{assumption}

\begin{remark}\label{rem:conditions1} $\,$
\begin{enumerate}
\item Note that Assumption \ref{ass1} \eqref{ass:linear} implies that 
Assumption \ref{assu:lip} holds with $L_G=1$.
\item \AB{The linearity in Assumption \ref{ass1}(\ref{ass:linear}) is only used at one key step in the proof of the lower bound in  ensuring  certain nondegeneracy estimates in \eqref{eq:relentetarhom1} - \eqref{eq:1243} through a perturbation construction in \eqref{eq:gmks}. The need for this nondegeneracy property in the proof of the lower bound is discussed in Section \ref{sec:overview5.1}.}
{\item Assumption \ref{ass1}(2) is a mild condition. It will hold if, in addition to Assumption \ref{ass1} \eqref{ass:linear}, we have that   \begin{equation}\label{eq:newass}
\mbox{ for some } m \in \clp(\Delta^o),\; G(m) \mbox{ is irreducible.}
\end{equation}
Indeed, consider the adjacency matrix $A^*$ defined as 
\begin{equation}
    A^*_{x,y} = \begin{cases}
        1, & \text{if there exists }m \in \clp(\Delta^o) \text{ such that } G(m)_{x,y} > 0, \\
        0, & \text{if }G(m)_{x,y} = 0 \; \text{for all }m \in \clp(\Delta^o) .
    \end{cases}
\end{equation}
Then $A^*\in \cla$ and from the assumption that for some $m \in \clp(\Delta^o)$, $G(m)$ is irreducible, it follows that $A^*$ is irreducible. This assumption, together with the linearity of $G$ also implies that for all $x,y \in \Delta^o$, $\sum_{z\in \Delta^o} G(\delta_z)_{x,y}>0$. Let $\delta_0 \doteq \inf_{(x,y) \in A^*_+}\sum_{z\in \Delta^o} G(\delta_z)_{x,y}$. Then, for any $(x,y) \in A^*_+$ and $m \in \clp(\Delta^o)$, using the linearity of $G$,
$$G(m)_{x,y} = \sum_{z \in \Delta^o} m_z G(\delta_z)_{x,y} \ge
\min_{z \in \Delta^o} m_z \sum_{z \in \Delta^o} G(\delta_z)_{x,y}
\ge \delta_0 \min_{z \in \Delta^o} m_z
$$
which shows that Assumption \ref{ass1}(2) holds. In fact, in the presence of Assumption \ref{ass1} \eqref{ass:linear}, the condition in \eqref{eq:newass} and Assumption \ref{ass1}(2) are equivalent and there is a unique $A$ that satisfies this assumption which is given by $A^*$ defined above.}
    \item In many examples of interest we will have that a strong law of large numbers holds and that the limiting measure is non-degenerate, namely
    $L^{n} \to \pi^*$ a.s. for some $\pi^* \in \clp_+(\Delta^o)$ as $n \to \infty$. In such a case,  Assumption \ref{ass1} \eqref{ass:conv-occ} clearly holds.
    Also, in such a case, it is easy to verify that  $\pi^*$ is a fixed point, namely $\pi^*G(\pi^*)=\pi^*$. Thus, Assumption \ref{ass1} \eqref{ass:qsdG} holds as well.
    \item \label{rem:brouwercomment} Consider the map $T: \clp(\Delta^o) \to \clp(\Delta^o)$ given by $Tm \doteq m G(m)$.  Since $\clp(\Delta^o)$ is compact and convex, Assumption \ref{ass:contin} and Brouwer's fixed point theorem ensure that  there is some $\pi^* \in \clp(\Delta^o)$ such that $T \pi^* = \pi^* G(\pi^* ) = \pi^* $.  In addition, in many situations of interest $G(m)$ will be an irreducible transition probability kernel for all $m \in \clp(\Delta^o)$. In such cases we have in fact that $\pi^* \in \clp_+(\Delta^o)$ and
 so Assumption \ref{ass1} \eqref{ass:qsdG}   holds.
    Suppose the following stronger form of irreducibility holds: 
	\begin{equation}\label{eq:1059nn}
		\mbox{ For  some } K \in \NN \mbox{ and all } m_1, \ldots , m_K \in \clp(\Delta^o),\; \; 
    \sum_{j=1}^KG(m_1)G(m_2)\cdots G(m_j) >0.
	\end{equation}
	 Then, as shown in
    Lemma \ref{lem:irred} in the Appendix, in this case  Assumption \ref{ass1} \eqref{ass:conv-occ} holds as well.

\end{enumerate}
  In Section 
    \ref{sec:examp} we present several examples for which the assumptions of Theorem \ref{thm:ldp1} are satisfied; see also Example \ref{ex:operatorexamples} below.
\end{remark}
We now introduce the rate function that governs the large deviation asymptotics. This function will be defined in terms of a matrix $A \in \cla$.

Let
\begin{equation}\label{eq:nicemzr}
	\begin{aligned}
	\clp^*(\Delta^o)
	& \doteq \{m \in \clp(\Delta^o): \mbox{ for some } \bar m \in \clp(\Delta^o \times \Delta^o),\\
 &\quad \quad\bar m_{(2)}=\bar m_{(1)} = m \mbox{ and }
	\bar m(x,y)= 0 \mbox{ for all } (x,y) \in (A_+)^c\}.
	\end{aligned}
\end{equation}
The class $\clp^*(\Delta^o)$ consists of all probability measures $m$ that are invariant measures for some transition probability kernel for which, at each state charged by $m$, jumps can only occur to neighbors as defined by the adjacency matrix $A$.
 Let $\clu$ denote the collection of all measurable maps from $\RR_+$ to $\clp(\Delta^o \times \Delta^o )$. 
 {For  $\eta \in \clu$, we will use the notation
 $$\eta(s)(\{x\}\times \{y\}) = \eta(\{x\}\times \{y\}\mid s) =
 \eta(x,y \mid s), \, (x,y) \in \Delta^o\times \Delta^o.$$
 Note that for any such $\eta \in \clu$ and $s\in \RR_+$, the probability measure $\eta(\cdot \mid s)$ on $\Delta^o\times \Delta^o$ can be disintegrated as
 \[
\eta(x,y\mid s) = \eta_{(1)}(x \mid s) \eta_{2|1} (y \mid s,x), \; (x,y) \in \Delta^o\times \Delta^o.
\]
For $\eta \in \clu$ and $m \in \clp(\Delta^o)$ consider the equation
\begin{equation}\label{P3} 
	M(t) = m - \int_0^t \eta_{(1)}(s) ds + \int_0^t M(s) ds, \;\; t \in \RR_+, 
	\end{equation}
 where $\eta_{(1)}(s) \doteq \eta_{(1)}(\cdot \mid s) \in \clp(\Delta^o)$.
 Regarding $\eta_{(1)}$ as an $\RR^d$-valued vector field,  this equation has a unique solution in $C([0,\infty): \RR^d)$ and for each $t\in \RR_+$, $M(t)$ can be viewed as a signed measure on $\Delta^o$ with the property that $\sum_{x\in \Delta^o} M(t)(x) =1$. 
 For  $x,y \in \Delta^o$ and $t \in \RR_+$, let
\begin{equation}\label{eq:defnbeta}
\beta(\{x\}\times \{y\} \times [0,t]) \doteq \int_0^t \eta(s)(\{x\}\times \{y\})ds.
\end{equation}
 
 }

{For fixed $m \in \clp(\Delta^o)$ and  $ \eta \in \clu$ consider  the following set of three properties.
\begin{property}\label{prop:z1}
$\;$
\begin{enumerate}[label = (\alph*)] 
\item \label{cluma} 
For $\beta$ as in \eqref{eq:defnbeta},
\begin{equation}
    \beta(\{x\}\times \{y\} \times [0,t]) = 0,\; \; \text{for all } \; (x,y) \in (A_+)^c \mbox{ and } t \in \RR_+.
    \label{P1}
\end{equation}
{Note that the above is equivalent to $\eta(t)(\{x\}\times \{y\}) = 0$ for all $(x,y) \in (A_+)^c$ and
a.e. $t\in \RR_+$.
}

\item \label{clumb} For a.e. $s \in \RR_+$, the two marginals of $\eta(\cdot\mid s)$ are the same, namely,  
\begin{equation}
\begin{split}
\eta_{(1)}(x \mid s) &= \sum_{z \in \Delta^o}\eta(z,x\mid s) \doteq \eta_{(2)}(x\mid s), \mbox{ for all } x \in \Delta^o, \mbox{ a.e. } s \in \RR_+.
\label{P2}
\end{split}
\end{equation}
\item \label{clumc} 
{The function $M$ defined in \eqref{P3}  is in $C(\RR_+: \clp(\Delta^o))$.}
Furthermore, there is a $\MP \in C(\RR_+: \clp(\Delta^o\times \Delta^o))$ satisfying, for all $s,t  \in \RR_+$, $\|\MP(t)-\MP(s)\| \le 2 |t-s|$,
$M(t) = (\MP(t))_{(1)} = (\MP(t))_{(2)}$, and $\MP(t)(x,y) =0$ for all $(x,y) \in (A_+)^c$. In particular, 
 for all $t \in \RR_+$, $M(t) \in \clp^*(\Delta^o)$.
\end{enumerate}
\end{property}

{Recall that the function defined by \eqref{P3} is always a continuous path in $\RR^d$ that satisfies $\sum_{x \in \Delta^o} M(t)(x)=1$ for all $t\in \RR_+$. Thus for (c) to be satisfied we need two additional properties: $M(t)(x) \ge 0$ for all $t\in \RR_+$ and $x \in \Delta^o$; and the probability measure $M(t)$ is an invariant measure for a suitable transition probability kernel with a communication structure that is consistent with that of $A$. A basic example for which this property holds is where $m = M(t) = \eta(\cdot \mid t) = \pi^*$ and $\MP(t)(x,y) = \pi^*(x)G(\pi^*)_{x,y}$, $(x,y) \in \Delta^o\times \Delta^o$, $t\in \RR_+$, where $\pi^*$ is the fixed point introduced in Assumption \ref{ass1}(3).}

\PZ{Recalling that $\clu$ denotes the collection of measurable maps from $\RR_+$ to $\clp(\Delta^o\times\Delta^o)$, we now define the set
\begin{equation}\label{eq:clumdef}
\clu(m) \doteq \{ \eta \in \clu: \eta \: \text{satisfies} \: \mbox{Property \ref{prop:z1}} \}.
\end{equation}

If, for some $m \in \clp(\Delta^o)$ and $\eta \in \clu(m)$, we have that $M \in C(\RR_+: \clp(\Delta^o))$ solves \eqref{P3}, we say that $M$ \textit{solves} $\clu(m,\eta)$.}
}

For $m \in \clp(\Delta^o)$, define ${I_A}: \clp(\Delta^o) \to [0,\infty]$ as
\begin{equation}\label{def:ratefunction1}
	{I_A(m)} \doteq \inf_{\eta \in \clu(m)} \int_0^{\infty} \exp(-s) \sum_{x \in \Delta^o} \eta_{(1)}(x\mid s) R\left(\eta_{2|1}(\cdot \mid s,x) \| G(M(s))(x, \cdot)\right) ds,
\end{equation}
where $M$ solves $\clu(m,\eta)$.  By the chain rule for relative entropy {(Theorem  \ref{thm:chainrule})} and \eqref{P2}, 
\begin{equation} \label{def:ratefunction2}
    {I_A(m)} = \inf_{\eta \in \clu(m)} \int_0^{\infty} \exp(-s) R\left(\eta( s) \| \eta_{(1)}( s) \otimes G(M(s))\right) ds,
\end{equation}
\AB{where $\eta_{(1)}( s) \otimes G(M(s)) \in \clp(\Delta^o \times \Delta^o)$ is defined as
\[
\eta_{(1)}( s) \otimes G(M(s))(\{x\}\times \{y\}) \doteq  \eta_{(1)}(x \mid s) G(M(s))_{x,y} , \; \; x,y \in \Delta^o.
\]}
{We remark that $A$ in the notation $I_A$ captures the fact that the collection $\clu(m)$ that appears on the right side of \eqref{def:ratefunction1} depends on the choice of $A \in \cla$. 
Viewing the function $\eta_{(1)}$ on the right side of 
\eqref{P3} as a control, $M(\cdot)$ can be regarded as a controlled measure valued path.
The function $I_A$, which will {be} shown to be the rate function associated with a LDP for the sequence $\{L^n, \; n \in \NN\}$, can then be viewed as the minimal cost associated with a suitable collection of controls and the associated controlled measure valued paths.
The exponential discount that arises in the expression for the rate function is a direct consequence of the time interpolation we use. This is explained further in Remark \ref{rem:expdisc}.
}

The following theorem is the main result of this work, which establishes an LDP for  $\{L^{n}, \;  n \in \NN\}$.
\begin{theorem}\label{thm:ldp1}
Fix $A \in \cla$ and let ${I_A} : \clp(\Delta^o) \to  [0,\infty]$ be the  function defined in \eqref{def:ratefunction1}. 
Suppose that Assumption \ref{assu:lip} is satisfied.
Then ${I_A}$ is a rate function and the sequence $\{L^{n+1},\; n \in \NN_0\}$ defined in \eqref{eq:occmzr} satisfies the LDP upper bound with rate function ${I_A}$, namely, for each closed set $F \subseteq \clp (\Delta^o)$,
\[
\limsup_{n\to\infty}n^{-1} \log P(L^{n+1} \in F) \le -\inf_{m \in F}{I_A}(m).
\]
Suppose in addition that Assumption \ref{ass1} is satisfied. Then, the LDP lower bound holds as well, with the  rate function ${I_A}$ and the adjacency matrix $A$ given as in Assumption \ref{ass1}, namely,
for each open set $G \subseteq \clp(\Delta^o)$,
\[
\liminf_{n\to\infty} n^{-1} \log P(L^{n+1} \in G) \ge - \inf_{m \in G}{I_A}(m).
\]
\end{theorem}

\begin{proof}
 Large deviation upper and lower bounds are proved in Theorem \ref{thm:laplaceupperbound1} and Theorem \ref{thm:lowbd}
 respectively.
  In Proposition \ref{lem:compactlevelsets} (Section \ref{sec:levelset}) we show that ${I_A}$ is a rate function, namely it has compact sublevel sets.
\end{proof}

\begin{remark}\label{rem:ivstili}
Recall the function $\tilde I$ from \eqref{eq:1230nn}. We now show that ${I_A(m)} \le \tilde I(m)$ for all $m \in \clp(\Delta^o)$.
Without loss of generality suppose that $\tilde I(m)<\infty$ and consider $\lambda \in \cli(m)$ with
$R(\lambda \| m \otimes G(m))<\infty$. {Then, since $G(m)(x,y)  =0$ for all $m$ whenever $(x,y) \in (A_+)^c$,  we must have that $\lambda(x,y) =0$ for all $(x,y) \in (A_+)^c$; otherwise, we would have that $R(\lambda \| m \otimes G(m)) = \infty$. We now claim that $I_A(m) \le R\left(\lambda \| m \otimes G(m)\right)$.  To see this, consider $\eta \in \clu$ defined as
$\eta(\cdot \mid s) \doteq \lambda$ for all $s\in \RR_+$. Then  \eqref{P1} and \eqref{P2} are satisfied.} Furthermore, since
$\eta_{(1)}(s) = \lambda_{(1)} = m$, \eqref{P3} is satisfied with $M(t)=m$, $t\in \RR_+$. This shows that
$\eta \in \clu(m)$ and that $M$ solves $\clu(m, \eta)$. Note that the cost on the right side of \eqref{def:ratefunction1}, with this choice of $\eta \in \clu(m)$ and {$M \in C(\RR_+: \clp(\Delta^o))$, 
is
$$\int_0^{\infty} \exp(-s) R\left(\lambda \| \lambda_{(1)} \otimes G(m)\right) ds = R\left(\lambda \| \lambda_{(1)} \otimes G(m)\right) \int_0^{\infty} \exp(-s) ds
=  R\left(\lambda \| m \otimes G(m)\right).$$}
This proves the claim ${I_A(m)} \le R\left(\lambda \| m \otimes G(m)\right)$, from which the inequality ${I_A(m)} \le \tilde I(m)$ follows on
taking infimum over $\lambda$. In Section \ref{sec:examp}, Example \ref{ex:sec8example1}, we show that when $G(m)$ is independent of $m$ then the reverse inequality holds as well. 
\end{remark}

{\begin{remark} {From Remark \ref{rem:conditions1}(2) recall that, if Assumption \ref{ass1} holds, then there is a unique $A^* \in \cla$ that satisfies this assumption. Furthermore, $I_A(m) \le I_{A^*}(m)$ for all $A \in \cla$ and $m \in \clp(\Delta^o)$. It then follows that
$$
\limsup_{n\to\infty}n^{-1} \log P(L^{n+1} \in F) \le -\inf_{m \in F}I_{A^*}(m) \le
-\inf_{m \in F}I_{A}(m), \mbox{ for all } A \in \cla.
$$
}
For notational convenience, throughout the rest of this work we will omit the dependence of $I_A$ on $A$, and will simply write $I$ instead.
\end{remark}
}

{We now return to  Example \ref{ex:introduceqsdalg1} to note that  Assumption \ref{ass1} is satisfied in this case. Several other examples are discussed in Section \ref{sec:examp}.}

\begin{example}[{\bf Quasi-Stationary Distributions}]\label{ex:operatorexamples}
{Let $\Delta, \Delta^o, P,P^o, \pi^*, G, \{Y_n, \; n \in \NN_0\}$, and $\{L^{n+1}, \; n \in \NN_0\}$ be as in Example \ref{ex:introduceqsdalg1}.  The current work establishes a large deviation principle for the sequence $\{L^{n+1}, \; n \in \NN_0\}$
 under the same irreducibility assumption on $P^o$ made in 
 \cite{benclo}. To see that Assumption \ref{ass1} is satisfied in this setting,} note that Part \ref{ass:linear} of this Assumption is clearly satisfied by $G$.  Part \ref{ass:qsdG} is also satisfied under the above irreducibility assumption (see \cite{qsdbook}).  From  \cite{AFP88},  $L^n$ converges a.s. to the unique QSD $\pi^*$, so Part \ref{ass:conv-occ} holds as well. Finally, for Part \ref{ass:aux}, define $\tplus \doteq \{ (x,y) \in \Delta^o \times \Delta^o : P_{x,y} + P_{x,0} > 0\}$ and $A_{x,y} = \bm{1}_{\{(x,y) \in \tplus\}}$. Clearly $A$ is irreducible and parts \ref{ass:aux0} and \ref{ass:aux1}  of Assumption \ref{ass1} are satisfied with this choice of the adjacency matrix $A$. Thus, the conditions for Theorem \ref{thm:ldp1} are satisfied and one has  a large deviation principle for the empirical measure associated with the self-interacting chain introduced in \cite{AFP88} for the approximation of the QSD of $\{Y_n, \; n \in \NN_0\}$. We remark that the model in \eqref{eq:qsdG} can also be viewed as a type of a vertex-reinforced random walk on $\Delta^o$. In this walk, given that at some instant the walker is at site $x$, it jumps to a site $y$ with probability that depends on the fraction of time the walker has previously visited the site $y$, as given by the formula $P(x,y)+P(x,0) m_y$.

\end{example}

{\section{A Stochastic Control Representation}\label{sec:controlproc}}
Throughout this section and next we fix $A \in \cla$. 
\AB{This section is notationally demanding and for the reader's convenience we have included a table of commonly used notation in Appendix \ref{sec:notationtablepart1}.}  We now introduce a pathwise construction of the collection $\{X_n, L^{n+1}, \; n \in \NN_0\}$, suitable for obtaining a tractable variational representation of the Laplace functionals of interest. For this construction it is useful to identify the state space with the space $\clv^d = \{ \bdelta_x : x \in \Delta^o\}$ introduced in Section \ref{sec:not}. 
{In particular, note that each $K \in \clk(\Delta^o)$  can be associated with a unique  $K^{\clv} \in \clk(\clv^d)$ 
through the identity
\[
K^{\clv}_{\bdelta_x, \bdelta_y} = K_{x,y}, \; x,y \in \Delta^o,
\]
where, as discussed in Section \ref{sec:not}, $\clk(\Delta^o)$ and $\clk(\clv^d)$ denote the space of transition kernels on $\Delta^o$ and $\clv^d$, respectively. 
}
Similarly, define the operator $G^{\clv} : \clp(\Delta^o) \to \clk(\clv^d)$ by 
\begin{equation}\label{eq:ggv}
G^{\clv}(m)_{\bdelta_x, \bdelta_y}  \doteq G(m)_{x,y}, \; x,y \in \Delta^o.
\end{equation}
Let $\{\nu^k(x, m), \; x \in \Delta^o, m\in \clp(\Delta^o), k \in \NN\}$ be iid $\clv^d$-valued random fields such that, for each $x \in \Delta^o$ and $m \in \clp(\Delta^o)$,
\begin{align*}
	 P(\nu^1(x, m) = \bdelta_y) = G^{\clv}(m)_{\bdelta_x, \bdelta_y} = G(m)_{x,y} , \quad  y \in \Delta^o.
\end{align*}
Then, the collection $\{X_n, L^{n+1}, \; n \in \NN_0\}$ has the following distributionally equivalent representation: $(X_0, L^1) = (x_0, \bdelta_{x_0})
$, 
\begin{equation}
\label{eq:defL}
	L^{k+1} = L^k + \frac{1}{k+1}\left[\nu^k(X_{k-1}, L^k) - L^k\right],\quad  \bdelta_{X_k} = \nu^k(X_{k-1}, L^k), \;\;
	k\in \NN.
\end{equation}
{The rest of this section is organized as follows. In Section \ref{sec:contrepx} we present a variational representation (Proposition \ref{prop:varrep0}) for Laplace functionals in terms of certain controlled analogues of $\{X_n, L^{n+1}, \; n \in \NN_0\}$. This representation will play a central role in the proofs. Section \ref{sec:twostep} introduces a two-step controlled chain that will be useful in giving suitable characterization of the limit points of controlled empirical measures. Section \ref{sec:simprep} gives a simpler representation for the cost in the variational representation from Proposition \ref{prop:varrep0} which is more amenable to a weak convergence analysis. Finally Section \ref{sec:timerev} gives an alternative form of the variational representation (Proposition \ref{prop:varrep}) using certain time reversed processes.
}

\AB{The representation in Proposition \ref{prop:varrep0} is more convenient for the proof of the LDP lower bound since it is given explicitly in terms of discrete time controlled sequences which can be constructed using the piecewise constant near optimal control obtained in Section 
\ref{sec:piecewiseconstant}.
On the other hand, the representation in Proposition \ref{prop:varrep} is better suited for the proof of the upper bound since it involves continuous time stochastic processes and random probability measures (on $\clv^d\times \clv^d \times \RR_+$) for which one can easily argue tightness and characterize  the limits on sending $n\to \infty$.}

\subsection{Controls and Controlled Sequences} \label{sec:contrepx} To prove the upper and lower bounds in Theorems \ref{thm:laplaceupperbound1} and \ref{thm:lowbd}, we rely on a certain  stochastic control representation for exponential moments of functionals of $\{L^n,\; n \in \NN\}$ presented below. This representation is given in Proposition \ref{prop:varrep0}, in  terms of certain controlled analogues of $\{\nu^k(X_{k-1}, L^k), L^{k+1}, \; k \in \NN\}$.

For each $n \in \NN$, the controlled stochastic system is a sequence $\{ \bar{L}^{n,k} , \; 1 \le k \le n\}$ of $\clp(\Delta^o)$-valued random variables which is defined recursively in terms of a collection of random probability measures on $\clv^d$, $\{\bar \mu^{n,k},\;  1 \le k \le n\}$,
where for each $1 \le k \le n$, $\bar \mu^{n,k}$ is $ \bar \clf^{n,k} \doteq \sigma(\{\bar L^{n,j}, \; 1\le j \le k\})$ measurable, as follows.
Define $\bar L^{n,1} \doteq \bdelta_{x_0}$,
and, having defined $\{ \bar{L}^{n,j}, 1\le j \le k\}$, 
$\bar L^{n,k+1}$ is defined as
\begin{equation}\label{eq:barlnk}
	\bar L^{n,k+1} \doteq
	\bar L^{n,k} +\frac{1}{k+1}\left[\bar \nu^{n,k}- \bar L^{n,k}\right], \;\;   k \le n,
\end{equation}
where $\bar \nu^{n,k}$ is a $\clv^d$-valued random variable such that
\begin{equation}\label{eq:condlawnu}
P[\bar \nu^{n,k} = \bdelta_x \mid \bar \clf^{n,k}] = \bar \mu^{n,k}(\bdelta_x), \;\; x \in \Delta^o.
\end{equation}
We set $\bar \nu^{n,0} \doteq \bdelta_{x_0}$, and, for each $n \in \NN$ and $0 \le k \le n,$ we let 
$\bar{X}^{n}_{k}$  denote the $\Delta^o$-valued random variable such that  $\bdelta_{\bar{X}^{n}_{k}} = \bar{\nu}^{n,k}$.
We note that the evolution equation \eqref{eq:barlnk} can be viewed as a `controlled' analogue of  \eqref{eq:defL}. In this equation, the superscript $n$ indexes a sequence of systems, and for each $n$, the index $0\le k \le n$ gives the first $n+1$ time steps of the $n$-th system.
For the stochastic control representation we give, it suffices, as discussed below \eqref{eq:varrep}, to consider controlled processes for which, a.s., for each $0 \le k \le n$, 
\begin{equation}\label{eq:prodmzrnumu0}
\bdelta_{\bar \nu^{n,k}} \otimes \bar \mu^{n,k+1}(\bdelta_x, \bdelta_y) = 0, \mbox{ for all } (x,y) \in ( \tplus)^c.
\end{equation}
We denote the collection of all such {\em control} sequences $\{\bar \mu^{n,k}, \;  1 \le k \le n\}$
as $\Theta^n$.

In order to study convergence behavior as $n\to \infty$ it is more convenient to work with processes indexed with a continuous time parameter, and so we consider the  following time interpolation
sequence 
$\{t_k, \; k \in \NN_0\}$ defined by 
\begin{equation}\label{eq:definetimeinterpolation}
t_0 \doteq 0, \; \; t_k = \sum_{j=1}^k (j+1)^{-1},\;\;  k \in \NN.
\end{equation}
The time stepping we use is motivated by the recursive definition of the controlled empirical measures $\bar L^{n,k}$ given in \eqref{eq:barlnk}
which says that the length of the  $k$-th interpolated interval should be of length $1/(k+1)$.
Such a time interpolation is standard in {the} study of stochastic approximation schemes \cite{benSA, BMP12, bor09, kus03}.

For each $n \in \NN$, define the $C(\RR_+ : \clp(\Delta^o))$-valued random variable $\bar{L}^n$   by linear interpolation: 
\begin{equation}\label{eq:411}\bar L^n(t) \doteq 
\bar L^{n, k+1} +(k+2) (t- t_k)  [\bar L^{n, k+2} - \bar L^{n,k+1}], \;\; t\in [t_k,t_{k+1}), \; {0 \le k \le n - 1}.
\end{equation}
The following variational representation follows from \cite[Theorem 4.2.2]{dupell}, \cite[Theorem 4.5]{buddupbook}. 
{\begin{proposition}\label{prop:varrep0}
\AB{For each $F \in  C_b(\clp(\Delta^o))$,
\begin{multline}\label{eq:varrep}
	-n^{-1}\log E \exp[-n F(L^{n+1})]
	\\
        = \inf_{\{\bar \mu^{n,i}\} \in \Theta^n}E \left[F(\bar L^n(t_n) ) + n^{-1} \sum\limits_{k=0}^{n-1} R\left(\bdelta_{\bar \nu^{n,k}}\otimes\bar \mu^{n,k+1} \|\bdelta_{\bar \nu^{n,k}}\otimes G^{\clv}(\bar L^{n,k+1})\right)\right].
\end{multline}}
\end{proposition}}
\AB{Note that $\bdelta_{\bar \nu^{n,k}}\otimes\bar \mu^{n,k+1}$ and $\bdelta_{\bar \nu^{n,k}}\otimes G^{\clv}(\bar L^{n,k+1})$ are random measures on $\clv^d \times \clv^d$ defined by, for $\bdelta_x, \bdelta_y \in  \clv^d$,
\[
\bdelta_{\bar \nu^{n,k}}\otimes\bar \mu^{n,k+1}(\bdelta_x , \bdelta_y) = \bdelta_{\bar \nu^{n,k}}(\bdelta_x) \bar \mu^{n,k+1}(\bdelta_y),
\]
and
\[
\bdelta_{\bar \nu^{n,k}}\otimes G^{\clv}(\bar L^{n,k+1})(\bdelta_x, \bdelta_y) = \bdelta_{\bar \nu^{n,k}}(\bdelta_x) G^{\clv}(\bar L^{n,k+1})(\bdelta_x, \bdelta_y)
\]
}
The fact that in the infimum on the right side of \eqref{eq:varrep} we can restrict, without loss of generality, to sequences
$\{\bar \mu^{n,i}, \; i \in \NN\} $ for which \eqref{eq:prodmzrnumu0} holds a.s. is because, if this property is violated, then the expression on the right side is $\infty$ since $A \in \cla$. 

{\subsection{ Two-step Controlled Chain} \label{sec:twostep} The proof of the LDP upper bound proceeds by characterizing the weak limit points of the controlled empirical measure processes $\bar L^n(\cdot)$ in terms of certain local (time dependent) stationary distributions. For this characterization it will be useful  
to  consider the following collection of $\clp(\Delta^o\times \Delta^o)$-valued random variables. Let  
$\barLP^{n,1} \doteq \bdelta_{(x_0, \bar X^{n}_{1})}$,
and, having defined $\{ \barLP^{n,j}, \; 1\le j \le k\}$, 
define $\barLP^{n,k+1}$ as
\begin{equation}
	\barLP^{n,k+1} \doteq
	\barLP^{n,k} +\frac{1}{k+1}\left[\bar \nu^{n,k}\otimes \bar \nu^{n,k+1}- \barLP^{n,k}\right], \;\;  k \le n,
 \label{eq:bartnk2}
\end{equation}
where,  $\bar \nu^{n,k}\otimes \bar \nu^{n,k+1}$ is the $\clp(\Delta^o\times \Delta^o)$-valued random variable defined, for $(x,y) \in \Delta^o \times \Delta^o$, as
$\bar \nu^{n,k}\otimes \bar \nu^{n,k+1}(x,y) =1$, if $\bar \nu^{n,k}=\bdelta_x$ and $\bar \nu^{n,k+1}=\bdelta_y$; and $0$ otherwise. 

We will also consider the continuous time interpolation of $\{\barLP^{n,k}, \;  1 \le k \le n\}$ defined as follows.}
For each $n \in \NN$, define the  $C(\RR_+ : \clp(\Delta^o\times \Delta^o))$-valued random variable $\barLP^n$ by linear interpolation: 
\begin{equation}\label{eq:411b}
	\barLP^n(t) \doteq 
\barLP^{n, k+1} +(k+2) (t- t_k)  [\barLP^{n, k+2} - \barLP^{n,k+1}],\; \; t\in [t_k,t_{k+1}), \;0 \le  {k \le n - 1}.
\end{equation}
\subsection{Simpler Representation for the Relative Entropy Cost}\label{sec:simprep}
 We will now give a simpler representation for the second term on the right side of \eqref{eq:varrep}.
 
 Consider random measures on $\clv^d \times [0,t_n]$  and
$\clv^d \times \clv^d \times [0,t_n]$
defined as follows: for $A \subseteq \clv^d$, $C \subseteq \clv^d$ and $B \in \clb[0,t_n]$,
\begin{equation}\label{eq:519}
\begin{aligned}
&\bar\Lambda^n(A\times B) \doteq \int_B \bar\Lambda^n(A\mid t) dt, \\
&	\bar \xi^n(A\times C \times B) \doteq \int_B \bar \xi^n(A \times C \mid t) dt,  \;\;
	\bar \Xi^n(A\times C \times B) \doteq \int_B \bar \Xi^n(A \times C \mid t) dt
\end{aligned}
\end{equation}
where, for $k \le n-1$ and $t \in [t_k, t_{k+1})$,
\begin{align}
\bar \Lambda^n(\cdot \mid t) \doteq \bdelta_{\bar \nu^{n,k+1}}(\cdot),\;\; 	\bar \xi^n(\cdot \mid t) \doteq \bdelta_{\bar \nu^{n,k}}\otimes \bar \mu^{n,k+1}(\cdot), \;\;
\bar \Xi^n(\cdot \mid t) \doteq \bdelta_{\bar \nu^{n,k+1}}\otimes \bdelta_{\bar \nu^{n,k+2}}(\cdot).
\label{eq:576new}
\end{align}
From \eqref{eq:prodmzrnumu0}, it follows that if $(x,y) \in (\tplus)^c$, then, for a.e. $t \in \RR_+$,
\begin{equation}\label{eq:anrest}
 \bar \xi^n(\bdelta_x, \bdelta_y \mid t) = 0, \;\; \bar \Xi^n(\bdelta_x, \bdelta_y \mid t) = 0, \; \mbox{a.s.}
 \end{equation}

Using \eqref{eq:411} and the  representation in \eqref{eq:519}, we can now
 rewrite the right side of the identity in \eqref{eq:varrep} as follows.
For $s \in \RR_+$, let
\begin{equation}\label{def:stepinv}
m(s) \doteq \sup\{ k : t_k \leq s\}, \quad
a(s) \doteq t_{m(s)}.
\end{equation}
 For each $n \in \NN$, define the random measure $\bar \zeta^n$ on $\clv^d \times \clv^d \times [0,t_n]$ as follows:  for $A \subseteq \clv^d$, $C \subseteq \clv^d$ and $B \in \clb[0,t_n]$,
\begin{equation}\label{eq:519b}
\begin{aligned}
\bar \zeta^n(A\times C \times B) \doteq \int_B \bar \zeta^n\left(A \times C \mid t, \bar L^n(a(t))\right) dt,
\end{aligned}
\end{equation}
where for $k \le n-1$, $m \in \clp(\Delta^o)$ and $t \in [t_k, t_{k+1})$,
\begin{align}\label{eq:zetadef3}
	\bar \zeta^n\left(\cdot\mid t,m\right) \doteq \bdelta_{\bar \nu^{n,k}}\otimes G^{\clv}(m)(\cdot).
\end{align}

Define $\psi_e: \RR_+ \to  \{2,3,\dots\}$ as
\begin{equation}\label{eq:psiedef1}
\psi_e(t) \doteq \sum\limits_{k=0}^{\infty} (k+2) \bm{1}_{[t_k,t_{k+1})}(t),
\end{equation}
so that 
\begin{equation}
	\begin{aligned}
	&n^{-1}  \sum\limits_{k=0}^{n-1} R\left(\bdelta_{\bar \nu^{n,k}}\otimes\bar \mu^{n,k+1} \|\bdelta_{\bar \nu^{n,k}}\otimes G^{\clv}(\bar L^{n,k+1})\right) \\
	&=   n^{-1} \int_0^{t_n} \psi_e(s) R\left(\bar \xi^n(\cdot \mid s) \|  \bar \zeta^n\left(\cdot \mid  s, \bar L^n(a(s))\right)\right) ds.
	\end{aligned}
	 \label{eq:540}
\end{equation}
Define $\clp(\clv^d \times \RR_+)$
and $\clp(\clv^d \times \clv^d \times \RR_+)$-valued random variables as follows: for $t\in \RR_+$ and $x,y \in \Delta^o$, let
\begin{equation}\label{eq:525}
\begin{aligned}
\lambda^n(\{\bdelta_x\} \times [0,t]) &\doteq 	n^{-1} \int_{0}^{t_n \wedge t} \psi_e(t_n-s) \bar \Lambda^n(\bdelta_x \mid t_n-s) ds\\
\beta^n(\{\bdelta_x\} \times \{\bdelta_y\}\times [0,t])&\doteq n^{-1} \int_{0}^{t_n \wedge t} \psi_e(t_n-s)\bar \xi^n((\bdelta_x, \bdelta_y)  \mid t_n-s) ds\\
\rho^n(\{\bdelta_x\} \times \{\bdelta_y\}\times [0,t]) &\doteq n^{-1} \int_{0}^{t_n \wedge t}  \psi_e(t_n-s) \bar \zeta^n
\left((\bdelta_x, \bdelta_y)  \mid t_n-s, \bar L^n(a(t_n-s))\right) ds.
\end{aligned}
\end{equation}
Note that for the measures $\lambda^n, \beta^n, \rho^n$ we have reserved the last coordinate to denote time.

The fact that the quantities in  \eqref{eq:525} define probability measures on $\clv^d\times \RR_+$ (resp. $\clv^d\times \clv^d \times \RR_+$) follows on observing that, for each $n \in \NN$,
\[
n^{-1} \int_0^{t_n} \psi_e(s) ds=1.
\]
Also, from \eqref{eq:anrest} it follows that if $(x,y) \in (\tplus)^c$, then, for each $t \in \RR_+$, 
\begin{equation}\label{eq:389a}
	\beta^n(\{\bdelta_x\} \times \{\bdelta_y\}\times [0,t]) = 0, \mbox{ a.s. }
\end{equation}
From \eqref{eq:540} and chain rule for relative entropies (Theorem {\ref{thm:chainrule}}), it follows that{
\begin{equation}
	n^{-1}  \sum\limits_{k=0}^{n-1} R\left(\bdelta_{\bar \nu^{n,k}}\otimes\bar \mu^{n,k+1} \|\bdelta_{\bar \nu^{n,k}}\otimes G^{\clv}(\bar L^{n,k+1})\right)
	= R\left(\beta^n \| \rho^n\right).
	\label{eq:541}
\end{equation}
}

\PZ{In view of the identity in \eqref{eq:541} and 
Proposition \ref{prop:varrep0}, we have  the following  result.
\begin{lemma} \label{lemma:variational-repr}
For each $F \in  C_b(\clp(\Delta^o))$,
\begin{equation}\label{eq:148n}
	-n^{-1}\log E \exp[-n F(L^{n+1})]
        = \inf_{\{\bar \mu^{n,i}\} \in \Theta^n}E \left[F(\bar L^n(t_n) ) + R\left(\beta^n \| \rho^n\right)\right],
\end{equation}
where $\{\beta^n, \; n \in \NN\},\{\rho^n, \; n \in \NN\}$ were defined in \eqref{eq:525}.
\end{lemma}
}

\subsection{A Time Reversed Representation}
\label{sec:timerev}
From \eqref{eq:barlnk}, \eqref{eq:bartnk2}, \eqref{eq:411}, and \eqref{eq:576new} it follows that, for $t \in [0,t_n]$, 
\begin{equation}\label{eq:522}
\begin{split}  
	\bar L^n(t) &= \bar L^n(0) + \int_0^t \sum\limits_{v 
	\in \clv^d} (v-\bar L^n(a(s))) \bar \Lambda^n(v\mid s) ds,
	\end{split}
\end{equation}
and, 
\begin{equation}\label{eq:522b}
\begin{split}
 \barLP^n(t) &= \barLP^n(0) + \int_0^t \sum_{v,v' \in \clv^d}(v\otimes v'
 -\barLP^n(a(s))) \bar \Xi^n(v, v'\mid s) ds.
	\end{split}
\end{equation}
The representation of the rate function in \eqref{def:ratefunction1} is given in terms of an optimal control problem associated with a given $m \in \clp(\Delta^o)$ in which the associated control $\eta$ satisfies the admissibility conditions {described in Property \ref{prop:z1}}. {This control problem is more tractable to analyze than the one that would emerge from a direct weak convergence analysis of the controlled process $\bar L^n$,  and which will take the form of an optimal control problem in which the goal is to control a trajectory so that it asymptotically approaches a given $m \in \clp(\Delta^o)$.} 
\AB{The control problem for the rate function $I_A(m)$ in \eqref{def:ratefunction1} involves controlled trajectories with initial state $m$. In the weak convergence proof of the  upper bound (Theorem \ref{thm:laplaceupperbound1}), the term $F(m)$ on the right side of Laplace upper bound arises from the weak limit of $F(\bar L^n(t_n))$. In order to relate $m$, which appears as an initial condition for a dynamical system in the definition of the rate function, with $\bar L^n(t_n)$, which is the state of the controlled empirical measure at a large time instant $t_n$, it is natural to view the dynamics of $\bar L^n$ backwards in time starting from the instant $t_n$ so that for the time reversed system $\bar L^n(t_n)$ becomes the initial condition to closely mirror the dynamical system in the definition of $I_A(m)$ with initial condition $m$.
} Towards that end,  for each $n \in \NN$, define the $C(\RR_+: \clp(\Delta^o))$-valued (resp. $C(\RR_+: \clp(\Delta^o \times \Delta^o))$-valued) random variable $\check \bfL^n$ (resp. $\cbLP^n$) by
\begin{equation} \label{eq:526}
(\check \bfL^n(t),  \cbLP^n(t)) \doteq  \begin{cases} (\bar L^n(t_n-t),  \barLP^n(t_n-t)) & 0 \leq t \le t_n\\
(\bar L^n(0),  \barLP^n(0))& t \ge t_n.\end{cases}
\end{equation}
Also, for each $n \in \NN$, define the $\clm(\clv^d\times \RR_+)$-valued random variable $\check \Lambda^n$ by, for $A \subseteq \clv^d$ and $t \in \RR_+$,
\begin{equation} \label{eq:527}
	\check \Lambda^n(A \times [0,t]) \doteq \int_{t_n-t}^{t_n} \bar \Lambda^n(A\mid s) ds
	= \int_{0}^{t} \check\Lambda^n(A\mid s) ds,
\end{equation}
where $\bar \Lambda^n(A\mid s) \doteq 0$ for $s \le 0$, and $\check\Lambda^n(A\mid s) \doteq \bar \Lambda^n(A\mid t_n-s)$ for $s\in \RR_+$. For each $n \in \NN$, define the quantities $\check \Xi^n$ and $\check \Xi^n(\cdot \mid s)$ similarly.
From \eqref{eq:522} we see that these time-reversed controlled processes satisfy the following evolution equation:
for  $t \in \RR_+$ and $n \ge m(t)$,
\begin{equation} \label{eq:timrev}
\begin{split}
\check \bfL^n (t) 
&= \check \bfL^n(0) - \int_{0}^{t} \sum\limits_{v\in\clv^d} v \check \Lambda^n(v \mid s)ds + \int_{t_n-t}^{t_n} \check \bfL^n(t_n - a(s)) ds,\\
\end{split}
\end{equation}
and, 
\begin{equation} \label{eq:timrevb}
\begin{split}
\cbLP^n (t)
&= \cbLP^n(0) - \int_{0}^{t} \sum_{v,v' \in \clv^d} v\otimes v'\,
\check \Xi^n(v, v' \mid s)ds + \int_{t_n-t}^{t_n} \cbLP^n(t_n - a(s)) ds.\\
\end{split}
\end{equation}
Combining the above with \eqref{eq:varrep} and \eqref{eq:148n} we now have the following key proposition.
\begin{proposition}\label{prop:varrep}
For each $n \in \NN$ and $F \in C_b(\clp(\Delta^o))$
\begin{equation}
	-n^{-1}\log E \exp[-n F(L^{n+1})] = \inf_{\{\bar \mu^{n,i}\} \in \Theta^n}
	E\left[F(\check \bfL^n (0)  ) + R\left(\beta^n \| \rho^n\right)\right].
\end{equation}
\end{proposition}
{\begin{remark}\label{rem:expdisc}
We now motivate the exponential discount that appears in the definition of the rate function.  Essentially, this is due to the specific continuous time interpolation used in our analysis.  To see this,  consider a function $g: \NN \to \RR$ with continuous time interpolation defined as $\bar{g} :  \RR_+ \to \RR$ defined by 
\[
\bar{g}(s) \doteq g(n) , \quad h_n  \le s \le h_{n+1},
\]
where $h_n = \sum_{j=1}^n j^{-1}$ denotes the $n$-th harmonic number. Letting $\tilde{e}(s) = \sup\{ n  : s \ge h_n\}$ and noting that  $h_n \approx \log n$ and $\tilde{e}(s) \approx e^s$, one can consider a sequence of time-reversals of $\bar{g}$ defined as
\[
\check{g}^n(s) \doteq \bar{g}(\log n - s), \; \; s \in [0, \log n].
\]
Then, under suitable  conditions on $g$, 
\begin{multline}
    \frac{1}{n} \sum\limits_{i=1}^n g(i) \approx  \frac{1}{n} \int_0^{h_n} \tilde{e}(s) \bar{g}(s)ds 
    \approx \frac{1}{n} \int_0^{\log n} e^{s} \bar{g}(s)ds \\
    \approx \frac{1}{n} \int_0^{\log n} e^{\log n - s} \check{g}^n(s) ds 
    \approx \int_0^{\log n}  e^{-s} \check{g}^n(s) ds,
\end{multline}
 This leads to   an asymptotic expression of the form
\[
\lim\limits_{n\to\infty} \frac{1}{n} \sum\limits_{i=1}^{n} g(i) = \int_0^{\infty} e^{-s} \check{g}^* (s) ds, 
\]
where $g^*$ is formally the limit of the sequence $\{\check{g}^n, \; n \in \NN\}$.
Note that  the second term in the cost on the right side of equation \eqref{eq:varrep} takes the form of a normalized sum as in the first line of the previous display. Thus, after considering limits of suitably interpolated and time reversed relative entropy terms (see e.g. \eqref{eq:540}) one arrives at an  approximation of such sums by infinite horizon discounted costs as in the definition of the rate function.

\end{remark}

}
\section{Laplace Upper Bound}
\label{sec:lub}
Recall that we fix $A \in \cla$ and suppress $A$ in the notation $I_A$ for the rate function.
The main result of the section is the following theorem, which gives the large deviations upper bound.
\begin{theorem}\label{thm:laplaceupperbound1}
	Suppose that Assumption \ref{assu:lip} is satisfied. Then,
	for every $F \in C_b(\clp(\Delta^o))$,
	\begin{equation*}
	\liminf_{n\to \infty}-n^{-1} \log E \exp[-n F(L^{n+1})] \ge \inf_{m \in \clp(\Delta^o)} [F(m) + I(m)].
	\end{equation*}
\end{theorem}
Assumption \ref{assu:lip} will be taken to hold for the rest of this section. 

{The section is organized as follows. In Section \ref{sec:prelimresults1} we study tightness properties of the controls, controlled empirical measures and other related objects. We also give a useful characterization of the weak limit points of these quantities. Using this characterization we then complete the proof of the Laplace upper bound in Section \ref{sec:laplaceupperbound}.}
\subsection{Tightness and Weak Convergence}\label{sec:prelimresults1}
A key step in the proof of Theorem \ref{thm:laplaceupperbound1} will be establishing the tightness of suitable controlled quantities and identifying their weak limit points. In preparation for that we 
first establish an elementary property of the function $\psi_e$ introduced in \eqref{eq:psiedef1}.
Recall the following estimate for the harmonic series (cf. \cite{BudWat3}).
For any $n\ge 2$
\begin{equation}\label{eq:eulermaschestimate}
\newgamma +\frac{1}{2(n+1)} < \sum\limits_{k=1}^n k^{-1} - \log n < \newgamma + \frac{1}{2(n-1)},
\end{equation}
where $\newgamma \approx 0.57721$ is the Euler-Mascheroni constant. Recall the map $m: \RR_+ \to \NN_0$ (resp. $\psi_e : \RR_+ \to \{2,3,\ldots\}$) from \eqref{def:stepinv} (resp. \eqref{eq:psiedef1}). 
 {Recalling the definition of $\{t_k, \; k \in \NN_0\}$ in \eqref{eq:definetimeinterpolation}}, by \eqref{eq:eulermaschestimate} and the observation that $t_n -s \leq t_{m(t_n-s)+1}$, we see that for all $n \in \NN$ and $s \in \RR_+$,
\begin{equation}\label{eq:harmonest}
\log(n+1 ) + \frac{1}{2(n+2)} - (s+1) \le t_n-s - \newgamma \le t_{m(t_n-s)+1} - \newgamma
\le \log(m(t_n-s)+2) -1 + \frac{1}{2(m(t_n-s)+1)}.
\end{equation}
As a consequence of this inequality we have the following lemma.
\begin{lemma}\label{lem:mtpsieasymptotics} 
For each $t \in \RR_+$, as $n \to \infty$, $n^{-1} m(t_n-t) \to \exp(-t).$
Additionally, for each $t \in \RR_+$, as $n \to \infty$,
\begin{equation}\label{eq:532to0}
\sup_{s\in [0,t]}\left| n^{-1} \psi_e(t_n-s) - \exp(-s) \right| \to 0.
\end{equation}

\end{lemma}

\begin{proof}
The first statement in the lemma is immediate from the	second on observing that for all $t \in \RR_+$,
$m(t_n-t) - \psi_e(t_n-t)=2$. We now prove the second statement. For  $s \in \RR_+$ and $n \in \NN$, let $k_{s,n} \doteq m(t_n-s)$, so that
\begin{equation}\label{eq:128nn}
n^{-1} \psi_e(t_n-s) =  n^{-1}( k_{s,n}+2).
\end{equation}
From \eqref{eq:harmonest},  for all $s \in \RR_+$ and $n \in \NN$,
\begin{multline}\label{eq:241eq1}
\newgamma + \log(n+1) + (2n+4)^{-1} - (s+1)\leq t_n - s \le t_{k_{s,n}+1} \leq \log(k_{s,n}+2)-1  + \newgamma + (2k_{s,n}+2)^{-1},
\end{multline}
from which it follows that $e^{-s} \leq (n+1)^{-1}(k_{s,n}+2)e^{\frac{1}{2k_{s,n}+2}}$,
and therefore that
\begin{equation}\label{eq:260}
(e^{-\frac{1}{2k_{s,n}+3}} - 1) e^{-s}
  \leq  n^{-1}(k_{s,n}+2) - e^{-s}.
\end{equation}
Next, using the estimate
$$
\newgamma + \log(k_{s,n}+1) \le \newgamma + \log(k_{s,n}+1) + (2k_{s,n}+4)^{-1} \le t_{k_{s,n}}+1
\le t_n-s+1 \le \newgamma +\log(n+1) + \frac{1}{2n} - s,$$
we have that $\log\left(\frac{k_{s,n}+1}{n+1}\right) \leq (2n)^{-1} - s$, which, along with the fact that $k_{s,n} \le n$, ensures that
$n^{-1}k_{s,n} \le e^{-s}e^{\frac{1}{2n}}$
Consequently,
\begin{equation}\label{eq:280}
n^{-1} (k_{s,n}+2) - e^{-s} \leq e^{-s}(e^{\frac{1}{2n}} - 1) + 2n^{-1} \leq e^{\frac{1}{2n}} + n^{-1}(2-n).
\end{equation}
Once more using \eqref{eq:241eq1}, we see that $\log(n+1) - s  \leq \log(k_{s,n}+2) + (2k_{s,n}+2)^{-1}$,
so, for fixed $t \in \RR_+$ and for each  $s \in [0,t]$ and all $n \in \NN$,
\begin{equation}\label{eq:288}
(n+1) e^{-t-1 } \leq (n+1)e^{-s-1} \leq k_{s,n}+2 .
\end{equation}
Using \eqref{eq:288}, we see that for each $n \in \NN$,
\[
\sup\nolimits_{s\in[0,t]} \left| e^{-\frac{1}{2k_{s,n}+3}} - 1\right| 
\leq 1 - \exp\left( - ( 2(n+1)e^{-t-1} -1)^{-1}\right)
\]
which shows that, as $n\to\infty$,
\begin{equation}\label{eq:296}
\sup\nolimits_{s\in[0,t]} \left| e^{-\frac{1}{2k_{s,n}+3}} - 1\right| \to 0.
\end{equation}
Combining \eqref{eq:260} and \eqref{eq:280}, we see that for each $s \geq 0$ and $n \in \NN$,
$$
\left(e^{-\frac{1}{2k_{s,n}+3}} - 1\right)e^{-s} \leq \frac{k_{s,n}+2}{n} - e^{-s} \leq e^{\frac{1}{2n}} + \frac{2-n}{n},
$$
so, from \eqref{eq:128nn}, for each $n \in \NN$,
\begin{equation}\label{eq:305}
\sup\nolimits_{s\in[0,t]} \left| n^{-1} \psi_e(t_n - s) - e^{-s}\right| 
 \leq  \max\left\{ \sup\nolimits_{s\in[0,t]}  (1-e^{-\frac{1}{2k_{s,n}+3}}), e^{\frac{1}{2n}} +n^{-1}(2-n)\right\}.
\end{equation}
Combining \eqref{eq:296} and \eqref{eq:305}, we obtain \eqref{eq:532to0}.
\end{proof}

The next lemma shows that the  sequences of various quantities, introduced in Section \ref{sec:controlproc},  associated with a sequence  of controls $\{\bar \mu^{n,i}\} \in \Theta^n$, is tight.
\begin{lemma}\label{lem:lemtight}
Let  for $n \in \NN$, $\{\bar \mu^{n,i}\} \in \Theta^n$.
The collection $\{(\check\bfL^n, \cbLP^n, \check \Lambda^n, \lambda^n, \beta^n, \rho^n), \; n \in \NN\}$,
associated with the sequence of controls $\{\bar \mu^{n,i},  \; n\in \NN\}$, as defined in Section \ref{sec:controlproc},
is tight in $C(\RR_+:\clp(\Delta^o))\times C(\RR_+:\clp(\Delta^o\times \Delta^o)) \times \clm(\clv^d\times \RR_+) \times \clp(\clv^d \times\RR_+) 
\times (\clp(\clv^d \times \clv^d\times\RR_+) )^2$.
\end{lemma}

\begin{proof}
We begin by showing that $\{ \check \bfL^n , \; n \in \NN\}$ is tight. Since $\clp(\Delta^o)$ is compact, it suffices to show that for some $C \in (0,\infty)$,  and for all $n \in \NN$ and $s,t \in \RR_+$,
$\| \check \bfL^n (t) - \check \bfL^n(s) \| \leq C |t - s|$, a.s.
However, this is immediate from \eqref{eq:timrev} (or equivalently \eqref{eq:522}), on using the fact that
$\|v-\bar L^n(a(s))\| \le 2$ for all $s\in \RR_+$ and $v \in \clv^d$.
The tightness of $\{ \cbLP^n , \; n \in \NN\}$ is argued similarly.

The tightness of $\{ \check \Lambda^n , \; n \in \NN\}$  in $\clm(\clv^d \times \RR_+)$
under the vague topology is immediate on observing that  for each $k \in \NN$,
$
\sup_{n \in\NN} \check \Lambda^n(\clv^d \times [0,k]) = k$.
Next, since $\clv^d$ is compact, the sequences $\{\lambda^n_{(1)}, \; n \in \NN\}$,  $\{\beta^n_{(1)},\;  n \in \NN\}$, $\{\beta^n_{(2)},\;  n \in \NN\}$, $\{\rho^n_{(1)}, \; n \in \NN\}$ and  $\{\rho^n_{(2)}, \; n \in \NN\}$, are obviously tight. Also, for each $n \in \NN$, 
$
\lambda^n_{(2)} = \beta^n_{(3)} = \rho^n_{(3)}$,
so to complete the proof it suffices to show that the sequence $\{ \lambda^n_{(2)} ,\;  n \in \NN\}$ is tight. Observe that, for each $n \in \NN$,  if $n \ge m(t)$, then, since $ t_n - t \le t_{m(t_n-t)+1}$,
\begin{multline*}
\lambda^n_{(2)}([0,t]) = n^{-1} \int_0^t \psi_e(t_n-s)ds = n^{-1} \int_{t_n-t}^{t_n} \psi_e(s)ds\\
 \geq n^{-1} \sum\limits_{k=m(t_n-t)+1}^{n-1} \int_{t_k}^{t_{k+1}} \psi_e(s)ds = 1 -n^{-1}(m(t_n-t) +1).
\end{multline*}
Fix $\veps > 0$ and  $t > \log (3 \veps^{-1})$. Then, from Lemma \ref{lem:mtpsieasymptotics},   we can find some  $n_0  > 2\veps^{-1}$ such that $n_0 \ge m(t)$ and
\[
\sup_{n\ge n_0}| n^{-1}m(t_{n}-t)- e^{-t}| \leq 2^{-1}\veps.
\]
Thus, 
$
\inf\nolimits_{n\geq n_0}\lambda^n_{(2)}([0,t]) \geq 1- \veps$.
Since $\veps>0$ is arbitrary, the desired tightness follows.
\end{proof}
The next lemma provides a useful characterization of  the weak limit points of the tight collection in Lemma \ref{lem:lemtight}. {Recall that $\beta^n$ and $\rho^n$ are random variables with values in 
$\clp(\clv^d\times \clv^d \times \RR_+)$.}
\begin{lemma}\label{lemchar}
	Let the sequence $\{(\check\bfL^n, \cbLP^n, \check \Lambda^n, \lambda^n, \beta^n, \rho^n), \; n \in \NN\}$ be as in Lemma \ref{lem:lemtight} and let
	$(\check\bfL^*, \cbLP^*, \check \Lambda^*, \lambda^*, \beta^*, \rho^*)$
be a weak limit point of the sequence.
Then, the following hold a.s.
\begin{enumerate}[label = (\alph*)]
	\item \label{lem:chara} The measure $\check \Lambda^*$ can be disintegrated as
	$\check \Lambda^*(dv, ds) = \check \Lambda^*(dv\mid s) ds$.
\item \label{lem:charb} For  $t\in\RR_+$
	\begin{equation}\label{eq:eq617}
		\check \bfL^*(t) = \check \bfL^*(0) - \int_0^t \sum\limits_{v\in \clv^d} v \check \Lambda^*(v\mid s) ds + \int_0^t \check \bfL^*(s) ds.
	\end{equation}
	\item \label{lem:charc} $\beta^*_{(1,3)} = \beta^*_{(2,3)} = \rho^*_{(1,3)}=\lambda^*$.
\item \label{lem:chard} For $t \in \RR_+$ and $x \in \Delta^o$,
$
\lambda^*( \{\bdelta_x\}\times [0,t]) = \int_0^t \exp(-s) \check \Lambda^*(\bdelta_x\mid s) ds$.
\item \label{lem:chare} For all $t \in \RR_+$ and $(x,y) \in (\tplus)^c$,
$
\beta^*(\{\bdelta_x\}\times \{\bdelta_y\} \times [0,t]) = 0$.

\item \label{lem:charf} For $t \in \RR_+$ and $x ,y\in \Delta^o$,
\[
\rho^*( \{\bdelta_x\}\times \{\bdelta_y\} \times [0,t]) = \int_0^t \exp(-s) \check \Lambda^*(\bdelta_x \mid s) G^{\clv}(\check \bfL^*(s))(\bdelta_x, \bdelta_y) ds.
\]
\item \label{lem:charg} For $t \in \RR_+$, $(\cbLP^*(t))_{(1)} = (\cbLP^*(t))_{(2)} = \check\bfL^*(t)$. Furthermore, for all $t \in \RR_+$ and $(x,y) \in (A_+)^c$, $\cbLP^*(t)(x,y) =0$, and for all $s,t \in \RR_+$, $\|\cbLP^*(t)- \cbLP^*(s)\| \le 2 |t-s|$.
\end{enumerate}
\end{lemma}

\begin{proof}
Fix a weakly convergent subsequence of $\{(\check\bfL^n, \check \Lambda^n, \lambda^n, \beta^n, \rho^n), \; n \in \NN\}$ and relabel it as $\{n\}$. We now prove the various statements in the lemma for the limit $(\check\bfL^*, \check \Lambda^*, \lambda^*, \beta^*, \rho^*)$ of this sequence.
\begin{enumerate}[label = (\alph*)]
\item This is immediate on noting that for each $n \in \NN$ and $t \in \RR_+$,  
$\sum_{x \in \Delta^o}\check \Lambda^n(\bdelta_x \times [0,t]) = t.$
\item By  appealing to Skorohod's representation theorem, we  assume without loss of generality  that $\{(\check \bfL^n, \check \Lambda^n), \; n \in\NN\}$ converges almost surely to $(\check \bfL^*, \check \Lambda^*)$. For  $t \in \RR_+$ and $m(t) \le n$, recall the evolution equation \eqref{eq:timrev}.
Also note that
{\[
\int_{t_n -t}^{t_n} \check{\bfL}^n(t_n - s)ds = \int_0^{t} \check{\bfL}^n(s)ds,
\]
so}
\begin{equation}\label{eq:weaklimitleqn1}
\int_{t_n-t}^{t_n} \check \bfL^n(t_n -a(s))ds
= \Bigg(\int_{t_n-t}^{t_n} \check \bfL^n(t_n -a(s))ds 
- \int_{t_n-t}^{t_n} \check \bfL^n(t_n -s)ds\Bigg)
+ \int_{0}^t \check \bfL^n(s)ds,
\end{equation}
and, for each $t \in \RR_+$, 
\begin{align}\label{eq:weaklimitleqn2}
\left\| \int_{t_n-t}^{t_n} \check \bfL^n(t_n -a(s))ds  - \int_{t_n-t}^{t_n} \check \bfL^n(t_n -s)ds\right\|
&\le t \sup_{s \in[ t_n- t, t_n]} \|\check \bfL^n(t_n -a(s)) - \check \bfL^n(t_n -s)\|\nonumber\\
&=t \sup_{s \in[ t_n- t, t_n]} \| \bar L^n( a(s)) -  \bar L^n(s)\|.
\end{align}
{ where the last equality follows on noting from \eqref{eq:526}  that  that $\check{\bfL}^n(t_n-u)
= \bar{L}^n(u)$ for $u \in [0, t_n]$.}
As in the proof of Lemma \ref{lem:lemtight},  for all $n \in \NN$ satisfying $t_n \geq t$ and $s \in [t_n-t,t_n]$, 
\begin{equation}\label{eq:weaklimitleqn3}
\| \bar{L}^n(a(s)) - \bar{L}^n(s)\| 
\leq  2 (m(t_n-t) +2)^{-1}.
\end{equation}
Combining \eqref{eq:weaklimitleqn1},  \eqref{eq:weaklimitleqn2}, and  \eqref{eq:weaklimitleqn3} with \eqref{eq:timrev}, and using the almost sure convergence of $\{(\check \bfL^n, \check\Lambda^n), \; n\in\NN\}$ to $(\check \bfL^*, \check \Lambda^*)$ we see that, as $n \to \infty$, for each $t\in\RR_+$,
\[
\check \bfL^n(t) \to \check \bfL^*(0) - \int\limits_{0}^t \sum\limits_{v\in\clv^d} v \check \Lambda^*(v \mid s)ds + \int_0^t \check \bfL^*(s)ds,
\]
almost surely. The result follows.

\item
As in part (b), we assume that the convergence of $\{(\lambda^n,\beta^n,\rho^n), \; n \in \NN\}$ holds in the almost sure sense. Observe that, for each $n \in \NN$,
$
\beta^n_{(1,3)} = \rho^n_{(1,3)}$,
so the identity $\beta^*_{(1,3)} = \rho^*_{(1,3)}$ follows. Now we show that  $\beta^*_{(1,3)} = \lambda^*$. Towards this end, fix $x \in \Delta^o$ and $t \in \RR_+$, and observe that, for each $n \geq m(t)$,
\begin{equation*}
\begin{split}
\left| \beta^n_{(1,3)} ( \{\bdelta_x \} \times [0,t] )   - \lambda^n ( \{\bdelta_x \} \times [0,t] )\right| &\leq n^{-1}  + n^{-1}\left|\sum\limits_{k=m(t_n-t) }^{n-1}\left( \bdelta_{\bar\nu^{n,k}}(\bdelta_x) -   \bdelta_{\bar\nu^{n,k+1}}(\bdelta_x)\right)\right|
\leq 2n^{-1}.
\end{split}
\end{equation*}
The desired identity follows  on letting $n \to \infty$. Now we show that $ \beta^*_{(2,3)} = \lambda^*$. Once more, fix $t \in \RR_+$ and $x \in \Delta^o$, and observe that, for each $n \geq m(t)$,
\begin{equation}\label{eq:595a}
\begin{split}
 &\lambda^n( \{\bdelta_x\} \times [0,t])  - \beta^n_{(2,3)} (\{\bdelta_x\} \times [0,t])\\
 &=  n^{-1} \int_0^{t} \psi_e(t_n -s ) \left(  \bar \Lambda^n(\bdelta_x \mid t_n - s) - \bar \xi^n_{(2)}( \bdelta_x \mid t_n - s)\right) ds\\
&= n^{-1} \int_{t_n-t}^{t_n} \psi_e(s) \left(   \bar \Lambda^n ( \bdelta_x \mid s) - \bar \xi^n_{(2)}(\bdelta_x \mid s) \right) ds.
\end{split}
\end{equation}
Additionally,  for $1 \leq l \leq m \le n$, 
\begin{equation}
\begin{split}
n^{-1} \int_{t_l}^{t_{m+1}} \psi_e(s) \left( \bar \Lambda^n (\bdelta_x \mid s) -  \bar \xi^n_{(2)}(\bdelta_x \mid s) \right) ds &= n^{-1} \sum\limits_{k=l}^{m} \left(  \bdelta_{ \bar \nu^{n,k+1}}(\bdelta_x) - \bar \mu^{n, k+1} (\bdelta_x)  \right)
\label{eq:602a}
\end{split}
\end{equation}
so, recalling \eqref{eq:condlawnu} and using the martingale-difference property, we see that
\begin{equation}\label{eq:608a}
\EE \left( n^{-1} \sum\limits_{k=l}^m \left( \bdelta_{\bar \nu^{n,k+1}}(\bdelta_x) - \bar \mu^{n,k+1}(\bdelta_x)\right)\right)^2 \leq n^{-2} (m - l + 1).
\end{equation}
Combining  \eqref{eq:595a}, \eqref{eq:602a}, and  \eqref{eq:608a}, we see that for some $C_1 \in (0,\infty)$, and all
$n\ge m(t)$,
\[
\EE \left( \left|\beta^n_{(2,3)} ( \{\bdelta_x\} \times [0,t])  - \lambda^n (\{\bdelta_x\} \times [0,t])  \right|^2\right) \leq C_1 n^{-1}.
\]
The statement in (c) follows on letting $n \to \infty$.

\item \label{item:gammanlimit} Once more we assume, without loss of generality, that  the sequence 
$
\{(\check\bfL^n, \check \Lambda^n, \lambda^n, \beta^n, \rho^n), \; n \in \NN\}$,
converges almost surely to $(\check\bfL^*, \check \Lambda^*, \lambda^*, \beta^*, \rho^*)$. Fix $x \in \Delta^o$ and $t \in \RR_+$, and observe that, for each $n \geq m(t)$, 
\begin{multline}\label{eq:629triangle}
\left |n^{-1}\int_0^{t} \psi_e(t_n-s) \check \Lambda^n(\bdelta_x\mid s) ds - \int_0^{t} \exp(-s) \check \Lambda^*(\bdelta_x\mid s) ds\right |\\
\leq \left | n^{-1}\int_0^{t} \psi_e(t_n-s) \check \Lambda^n(\bdelta_x\mid s) ds -  \int_0^t \exp(-s)\check \Lambda^n(\bdelta_x \mid s)ds\right |\\ 
 + \left |  \int_0^t \exp(-s) \check \Lambda^n(\bdelta_x \mid s)ds -  \int_0^{t} \exp(-s)\check \Lambda^*(\bdelta_x\mid s)ds \right |.
\end{multline}
Next, note that
\begin{multline}
\label{eq:988}
\left |n^{-1}\int_0^{t} \psi_e(t_n-s) \check \Lambda^n(\bdelta_x\mid s) ds -  \int_0^t \exp(-s) \check \Lambda^n(\bdelta_x \mid s)ds\right |\\
\leq \int_0^t \left| n^{-1}\psi_e(t_n-s) - \exp(-s)\right| \left | \check \Lambda^n (\bdelta_x \mid s)\right |ds
\leq  t \sup_{s\in[0,t]} \left| n^{-1} \psi_e(t_n-s) - \exp(-s)\right|,
\end{multline}
and, by convergence of $\check \Lambda^n$ to $\check \Lambda^*$, as $n\to \infty$,
\begin{equation}
\label{eq:989}
\left |  \int_0^t \exp(-s) \check \Lambda^n(\{\bdelta_x\} \times ds) -  \int_0^{t}  \exp(-s) \check \Lambda^*(\{\bdelta_x\}\times ds) \right |
\to 0.
\end{equation}
The displays in \eqref{eq:629triangle},\eqref{eq:988},\eqref{eq:989}, together with Lemma \ref{lem:mtpsieasymptotics} show that, as $n \to \infty$,
\[
n^{-1}\int_0^{t} \psi_e(t_n-s) \check \Lambda^n(\bdelta_x\mid s) ds \to \int_0^{t} \exp(-s) \check \Lambda^*(\bdelta_x\mid s) ds.
\]
On recalling the definition of $\lambda^n$ we now have the statement in (d).

\item This result follows immediately from the observation in \eqref{eq:389a}.

\item Once again, we assume that the a.e. convergence as in (d) holds. Fix  $t \in \RR_+$, $n \geq m(t)$, and $x,y \in\Delta^o$, and observe that,
\begin{equation}\label{eq:752a}
\begin{split}
&\left| \rho^n( \{\bdelta_x\} \times \{\bdelta_y\} \times [0,t]) - \int_0^t \exp(-s) \check \Lambda^*( \bdelta_x \mid s ) G^{\clv} ( \check \bfL^*(s))( \bdelta_x, \bdelta_y) ds\right|\\
&= \Bigg| n^{-1} \int_0^{t} \psi_e(t_n-s) \bar \zeta^n\left((\bdelta_x,\bdelta_y) \mid t_n - s, \bar L^n( a(t_n-s))\right)ds\\
& \quad - \int_0^t \exp(-s) \check \Lambda^*(\bdelta_x \mid s) G^{\clv}( \check \bfL^*(s))(\bdelta_x,\bdelta_y) ds\Bigg|\\
&\leq \left| \int_0^t  \left( n^{-1} \psi_e(t_n-s) - \exp(-s)\right)  \bar \zeta^n\left((\bdelta_x,\bdelta_y) \mid t_n - s, \bar L^n( a(t_n-s))\right)ds\right|\\
& \quad + \left| \int_0^t \exp(-s)  \left( \bar \zeta^n\left((\bdelta_x,\bdelta_y) \mid t_n - s, \bar L^n( a(t_n-s))\right) - \check \Lambda^*(\bdelta_x \mid s) G^{\clv}( \check \bfL^* (s))(\bdelta_x, \bdelta_y)\right) ds\right|.
\end{split}
\end{equation}
From Lemma \ref{lem:mtpsieasymptotics}, 
\begin{equation}\label{eq:763a}
 \left| \int_0^t  \left( n^{-1} \psi_e(t_n-s) - \exp(-s)\right)  \bar \zeta^n\left((\bdelta_x,\bdelta_y) \mid t_n - s, \bar L^n( a(t_n-s))\right)ds\right| \to 0,
\end{equation}
as $n \to \infty$. Additionally,  for each $s \in [0,t]$,
\begin{equation*}
\begin{split}
&  \bar \zeta^n((\bdelta_x,\bdelta_y) \mid t_n - s, \bar L^n( a(t_n-s))) - \check \Lambda^*(\bdelta_x \mid s) G^{\clv}( \check \bfL^* (s))(\bdelta_x, \bdelta_y)\\
& \quad =  \bdelta_{\bar \nu^{n, m(t_n-s)}} \otimes G^{\clv} (\bar L^n (a(t_n-s))) ( \bdelta_x, \bdelta_y) -  \check \Lambda^*(\bdelta_x \mid s) G^{\clv} ( \check \bfL^*(s))(\bdelta_x,\bdelta_y)\\
&\quad =    \bdelta_{\bar \nu^{n, m(t_n-s)}} \otimes G^{\clv} (\bar L^n (a(t_n-s))) ( \bdelta_x, \bdelta_y) -  \bdelta_{\bar \nu^{n,m(t_n-s)}}\otimes G^{\clv} ( \check \bfL^*(s))(\bdelta_x,\bdelta_y)  \\
&\qquad +    \bdelta_{\bar \nu^{n,m(t_n-s)}}\otimes G^{\clv} ( \check \bfL^*(s))(\bdelta_x,\bdelta_y) -  \check \Lambda^*(\bdelta_x \mid s) G^{\clv} ( \check \bfL^*(s))(\bdelta_x,\bdelta_y),
\end{split}
\end{equation*}
and
\begin{equation}\label{eq:776a}
\begin{split}
& \left|  \bdelta_{\bar \nu^{n, m(t_n-s)}} \otimes G^{\clv} (\bar L^n (a(t_n-s))) ( \bdelta_x, \bdelta_y) -  \bdelta_{\bar \nu^{n,m(t_n-s)}}\otimes G^{\clv} ( \check \bfL^*(s))(\bdelta_x,\bdelta_y) \right|\\
&\quad \leq \left| G^{\clv} ( \bar L^n (a(t_n-s)))(\bdelta_x,\bdelta_y) - G^{\clv} ( \check \bfL^* (s))(\bdelta_x,\bdelta_y)\right|\\
&\quad \leq \left| G^{\clv} ( \bar L^n(a(t_n-s))) (\bdelta_x, \bdelta_y)  - G^{\clv} ( \bar L^n(t_n-s))(\bdelta_x, \bdelta_y)\right|  \\
&\qquad + \left|G^{\clv} (\check \bfL^n(s) (\bdelta_x, \bdelta_y) - G^{\clv}( \check \bfL^*(s))(\bdelta_x,\bdelta_y)\right|.
\end{split}
\end{equation}
Since by Assumption \ref{assu:lip} $G^{\clv}$ is a Lipschitz map and recalling from \eqref{eq:weaklimitleqn3} that
\begin{equation}\label{eq:tanbd}
\| \bar L^n(a(t_n-s)) - \bar L^n(t_n-s)\| \leq 2(m (t_n-t) + 2)^{-1}, \quad s \in [0,t],
\end{equation}
and that $\check \bfL^n \to \check \bfL^*$ almost surely as $n \to \infty$, it follows from \eqref{eq:776a} that
\begin{equation}\label{eq:787a}
 \left|  \bdelta_{\bar \nu^{n, m(t_n-s)}} \otimes G^{\clv} (\bar L^n (a(t_n-s)))( \bdelta_x, \bdelta_y) -  \bdelta_{\bar \nu^{n,m(t_n-s)}}\otimes G^{\clv} ( \check \bfL^*(s))(\bdelta_x,\bdelta_y) \right| \to 0,
\end{equation}
as $n \to \infty$. Now, observe that, for each $s \in [0,t]$,
\begin{equation}\label{eq:791a}
\begin{split}
&   \bdelta_{\bar \nu^{n,m(t_n-s)+1}}\otimes G^{\clv} ( \check \bfL^*(s))(\bdelta_x,\bdelta_y) -  \check \Lambda^*(\bdelta_x \mid s) G^{\clv} ( \check \bfL^*(s))(\bdelta_x,\bdelta_y) \\
&\quad =   \bar \Lambda^n ( \bdelta_x \mid t_n -s)  G^{\clv} ( \check \bfL^*(s))(\bdelta_x ,\bdelta_y) -  \check \Lambda^*(\bdelta_x \mid s) G^{\clv} ( \check \bfL^*(s))(\bdelta_x,\bdelta_y)\\
&\quad = \check \Lambda^n ( \bdelta_x \mid s)  G^{\clv} ( \check \bfL^*(s))(\bdelta_x ,\bdelta_y) -  \check \Lambda^*(\bdelta_x \mid s) G^{\clv} ( \check \bfL^*(s))(\bdelta_x,\bdelta_y)
\end{split}
\end{equation}
Letting
$$h(s)\doteq \exp(-s)G^{\clv} ( \check \bfL^*(s))(\bdelta_x,\bdelta_y), \; s \in \RR_+,$$
we see that 
\begin{align*}
&\int_0^t \exp(-s)\left(\bdelta_{\bar \nu^{n,m(t_n-s)+1}}\otimes G^{\clv} ( \check \bfL^*(s))(\bdelta_x,\bdelta_y) -  \check \Lambda^*(\bdelta_x \mid s) G^{\clv} ( \check \bfL^*(s))(\bdelta_x,\bdelta_y)\right) ds\\
&\quad = \int_0^t h(s) \check \Lambda^n ( \bdelta_x \mid s) ds - \int_0^t h(s) \check \Lambda^* ( \bdelta_x \mid s) ds
 = \int_0^t h(s) \check \Lambda^n ( \bdelta_x \times ds) - \int_0^t h(s) \check \Lambda^* ( \bdelta_x \times ds)
\end{align*}
which converges to $0$ a.s. since $\check \Lambda^n \to \check \Lambda^*$  a.s. in $\clm(\clv^d\times \RR_+)$
and $h$ is a continuous and bounded function. 
To complete the proof of (f) it now suffices to show that for all $x \in \Delta^o$ and $t \in \RR_+$,
\begin{equation}\label{eq:1126n}
	\left | \int_0^t h(s)\bdelta_{\bar \nu^{n,m(t_n-s)+1}}(\bdelta_x) ds -
	\int_0^t h(s)\bdelta_{\bar \nu^{n,m(t_n-s)}}(\bdelta_x) ds \right| \to 0, \; \mbox{ a.s. }
\end{equation}
Fix {$x \in \Delta^o$}, $t \in \RR_+$ and $\veps>0$, and let $\kappa>0$ be such that $|h(s)-h(s')| \le \veps$ whenever $|s-s'|\le \kappa$ and $s,s' \in [0,t]$.
Let, for {$n \in \NN$ and  $k \le n$,} $\sigma^n(k) \doteq  \bdelta_{\bar \nu^{n,k}}(\bdelta_x)$ and choose $n_0 \in \NN$
such that $m(t_{n_0}-t)^{-1} < \kappa$. Then, for $n \ge n_0$, the quantity on the left side of \eqref{eq:1126n} can be written as 
\begin{align*}
	&\left| \int_{t_n-t}^{t_n} h(t_n-s) (\sigma^n(m(s)+1) - \sigma^n(m(s))) ds \right|\\
	&\le \left| \sum_{k= m(t_n-t)}^{n-1} (\sigma^n(k+1) -\sigma^n(k)) \int_{t_k}^{t_{k+1}} h(t_n-s) ds\right| + m(t_n-t)^{-1} \|h\|_{t,\infty}\\
	&\le \left| \sum_{k= m(t_n-t)}^{n-1} \left(\frac{1}{k+2} \sigma^n(k+1) h(t_n - t_{k+1})
	- \frac{1}{k+2} \sigma^n(k) h(t_n - t_{k})\right)\right|
	+ 2\veps t + m(t_n-t)^{-1} \|h\|_{t,\infty},
\end{align*}
where $\|h\|_{t,\infty} \doteq \sup_{0\le s \le t}|h(s)|$.
The last expression can be bounded above by 
\[
2\veps t + 2m(t_n-t)^{-1} \|h\|_{t,\infty} + \|h\|_{t,\infty} \sum_{k= m(t_n-t)}^{n-1} \left(\frac{1}{k+1}-\frac{1}{k+2}\right) \le  2\veps t + 3 m(t_n-t)^{-1} \|h\|_{t,\infty}.
\]
Taking the limit as $n\to \infty$, we now have that
$$
	\limsup_{n\to \infty}\left | \int_0^t h(s)\bdelta_{\bar \nu^{n,m(t_n-s)+1}}(\bdelta_x) ds -
	\int_0^t h(s)\bdelta_{\bar \nu^{n,m(t_n-s)}}(\bdelta_x) ds \right| \le 2 \veps t.$$
	Since $\veps>0$ is arbitrary, the statement in \eqref{eq:1126n} follows.
\item The first statement is immediate on noting that for $t \in \RR_+$,
$\check\bfL^n(t) = (\cbLP^n(t))_{(1)}$, and for $k \le n$, 
\[
\| (\barLP^{n, k+1})_{(2)} - (\barLP^{n, k+1})_{(1)}\| \le 2(k+1)^{-1}.
\]
The second statement follows from the fact that whenever $(x,y) \in (A_+)^c$, $\barLP^{n, k+1}(x,y)=0$ a.e. for all $n \in \NN$ and $k \le n$. The final statement follows immediately from \eqref{eq:timrevb} and the fact that $\cbLP^n$ converges a.s. to $\cbLP^*$ as $n \to \infty$.
\end{enumerate}

\end{proof}

\subsection{Proof of Laplace Upper Bound}\label{sec:laplaceupperbound}
{In this subsection we will use the tightness and characterization results from the previous section to complete the proof of the Laplace upper bound, namely Theorem \ref{thm:laplaceupperbound1}.}

\begin{proof}[Proof of Theorem \ref{thm:laplaceupperbound1}]
	Fix
	$F \in C_b(\clp(\Delta^o))$ and $\veps>0$. From the variational representation in  \eqref{eq:varrep}, for each $n \in \NN$ we can find $\{\bar \mu^{n,i}\} \in \Theta^n$ such that
	{\begin{multline}\label{eq:lub1739}
		-n^{-1} \log E \exp[-n F(L^{n+1})]\\
		\ge E \left[F(\bar L^{n}(t_n)) + n^{-1} \sum\limits_{k=0}^{n-1} R\left(\bdelta_{\bar \nu^{n,k}}\otimes\bar \mu^{n,k+1} \|\bdelta_{\bar \nu^{n,k}}\otimes G^{\clv}(\bar L^{n,k+1})\right)\right] - \veps, 
	\end{multline}}
where the sequence $\{\bar L^{n, k+1}, \; k \le n\}$ is defined by \eqref{eq:barlnk}.

For each $n \in \NN$, define the $\clp(\Delta^o)$-valued continuous process $\bar L^n$ and random measures $\bar \Lambda^n$ on $\clv^d \times [0, t_n]$ according to \eqref{eq:411} and
\eqref{eq:519}  respectively. 
Also define, for each $n \in \NN$, $\lambda^n$, $\beta^n$, $\rho^n$, $\check \bfL^n$, $\cbLP^n$, and $\check \Lambda^n$ as in \eqref{eq:525},
\eqref{eq:526}, and \eqref{eq:527}. 
Recalling the identity in  \eqref{eq:541} we have that
\begin{equation}\label{eq:910}
		-n^{-1} \log E \exp[-n F(L^{n+1})] \ge E \left[F(\check \bfL^n(0)) +  R(\beta^n\| \rho^n)\right] - \veps.
\end{equation}
From Lemma \ref{lem:lemtight},
the collection $
\{ (\check\bfL^n, \cbLP^n, \check \Lambda^n, \lambda^n, \beta^n, \rho^n), \; n \in \NN\}$ is tight in $C(\RR_+:\clp(\Delta^o))\times C(\RR_+:\clp(\Delta^o\times \Delta^o))\times \clm(\clv^d\times \RR_+) \times (\clp(\clv^d \times \RR_+)) \times (\clp(\clv^d \times \clv^d \times \RR_+))^2 $.

Let $(\check\bfL^*, \cbLP^*, \check \Lambda^*, \lambda^*, \beta^*, \rho^*)$
be a weak limit point of the above sequence and suppose without loss of generality that the convergence holds along the full sequence and in the almost sure sense.

From parts \ref{lem:charc} and \ref{lem:chard} of Lemma \ref{lemchar}, the third marginal of $\beta^*$, namely $\beta^*_{(3)}(ds)$, equals $e^{-s} ds$ a.s.
Disintegrate $\beta^*$ as 
\begin{equation}\label{eq:851}\beta^*(\{\bdelta_x\}\times \{\bdelta_y\}\times [0,t]) =
\int_0^t \exp(-s) \check \eta^{*}(\{\bdelta_x\}\times \{\bdelta_y\}, s)  ds,
\end{equation}
where $s\mapsto \check \eta^{*}(\cdot, s)$ is a measurable map from
$\RR^+$ to $\clp(\clv^d\times \clv^d)$.
Let, for $s\in \RR_+$, $\eta^*(\cdot, s) \in \clp(\Delta^o\times \Delta^o)$ be defined as
\begin{equation}\label{eq:etastartoetastarcheck}
\eta^*(\{x\}\times \{y\}, s) \doteq  \check\eta^{*}(\{\bdelta_x\}\times \{\bdelta_y\}, s), \;\; x,y \in \Delta^o,
\end{equation}
and  write
$
\eta^*(x,y\mid s) \doteq \eta^*(\{x\}\times \{y\}, s)$,  and, for $i=1,2$, $\eta^*_{(i)}(\cdot \mid s) \doteq \eta^*_{(i)}(\cdot , s)$.

From Lemma \ref{lemchar}(c),
$\eta^*_{(1)}(s) = \eta^*_{(2)}(s)$ for a.e. $s \in \RR_+$. In particular $\eta^*$ satisfies 
Property \ref{prop:z1}(b) in Section \eqref{sec:statres} (with $\eta$ replaced by $\eta^*$). 

Moreover, if $(x,y) \in (\tplus)^c$, then by part (e) of Lemma \ref{lemchar}, a.s., 
\[ \beta^*(\{\bdelta_x\} \times \{\bdelta_y\} \times [0,t]) = 0, \quad  t \in \RR_+,
\]
which shows that Property \ref{prop:z1}(a)
 holds with $(\beta, \eta)$ replaced by $(\beta^*, \eta^*)$. 

Next, note that, from Lemma \ref{lemchar}(c) and (d), for 
$x \in \Delta^o$ and $t\in \RR_+$,
\begin{align*}
\int_0^t \exp(-s) \check \Lambda^*(\bdelta_x \mid s) ds
= \lambda^*(\{\bdelta_x\}\times [0,t]) = \beta^*_{(1,3)}(\{\bdelta_x\}\times [0,t])
= \int_0^t \exp(-s) \eta^*_{(1)}(x \mid s) ds,
\end{align*}
where the last equality is from \eqref{eq:851} and \eqref{eq:etastartoetastarcheck}.
This shows that
\begin{equation}\label{eq:xnew3}
\eta^*_{(1)}(x \mid s) = \check \Lambda^*(\bdelta_x \mid s) \mbox{ for all } x \in \Delta^o \mbox{ and a.e. } s \in \RR_+
\end{equation}
and so $\sum_{v \in \clv^d} v \check \Lambda^*(v \mid s)(\cdot) = \eta^*_{(1)}(\cdot, s)$. In particular, from parts \ref{lem:chara}, \ref{lem:charb}, and \ref{lem:charg} of  Lemma \ref{lemchar},  it follows that Property \ref{prop:z1}(c) 
also holds (with $\eta$ replaced by $\eta^*$, $M$ replaced by $\check\bfL^*$, and $\MP$ replaced by $\cbLP^*$).

Since all parts of Property \ref{prop:z1}  hold a.s., it follows that  $\eta^* \in \clu(\check \bfL^*(0))$ a.s. Furthermore, $\check \bfL^*$ solves $\clu( \check \bfL^*(0), \eta^*)$ a.s.

Note, from \eqref{eq:851}, and parts \ref{lem:chard} and \ref{lem:charf} of Lemma \ref{lemchar}, that
\begin{equation}\label{eq:615}
	 (\lambda^*, \beta^*,\rho^*) = \left(\exp(-s) \check \Lambda^*(\cdot\mid s) ds, \; \exp(-s) \check \eta^{*}(\cdot\mid s) ds, \;  \exp(-s) \check \Lambda^*(\cdot\mid  s) \otimes G^{\clv}(\check \bfL^*(s))( \cdot, \cdot) ds \right).
\end{equation}
For $s\in \RR_+$, disintegrate
\begin{equation}\label{eq:920}
\eta^*(x,y \mid s) = \eta^*_{(1)}(x \mid s) \eta^*_{2|1}(y \mid s,x), \;\; (x,y) \in \Delta^o \times \Delta^o.
\end{equation}
Then, using \eqref{eq:910},
\begin{multline*}
		\veps + \liminf_{n\to \infty}-n^{-1} \log E \exp[-n F(L^{n+1})] \\
  \ge \liminf_{n\to \infty} E \left[F(\check \bfL^n(0)) +  R(\beta^n\| \theta^n)\right]
\ge \left[F(\check \bfL^*(0)) +  R(\beta^*\| \theta^*)\right]\\
= \left[F(\check \bfL^*(0)) +  R\left(\exp(-s)  \check \eta^{*}(\cdot\mid s)  ds\big\| \exp(-s) \check \Lambda^*( \cdot \mid s) \otimes G^{\clv}(\check \bfL^*(s))(\cdot, \cdot ) ds \right)\right]\\
=  \left[F(\check \bfL^*(0)) +  \int_0^\infty \exp(-s) R\left(  \eta^*( s)  
 		\big\|  \eta^*_{(1)} ( s) \otimes G(\check \bfL^*(s)) \right) ds\right]\\
     =  \left[F(\check \bfL^*(0)) +  \int_0^\infty \exp(-s) \sum_{x \in \Delta^o}  \eta^*_{(1)} (x \mid s) R\left(  \eta^*_{2|1} ( \cdot \mid s, x) \big\| G(\check \bfL^*(s))(x,\cdot)  \right) ds\right]\\
	\ge \left[F(\check \bfL^*(0)) + I(\check \bfL^*(0))\right]
		\ge \inf_{m \in \clp(\Delta^o)}\left[F(m) + I(m)\right],
\end{multline*}

where the second inequality uses Fatou's lemma and the lower semicontinuity of relative entropy and the third line uses  \eqref{eq:615}. The fifth line again uses  the chain rule for relative entropies and  the disintegration in \eqref{eq:920}. The last line uses the fact that $\check \bfL^*$ solves $\clu(\check\bfL^*(0), \eta^*)$, and the expression of the rate function $I$ given in \eqref{def:ratefunction1}. {To see the equality on the  fourth line, note first that  using the chain rule for relative entropies (Theorem \ref{thm:chainrule}),
\begin{align*}
 &R\left(\exp(-s)  \check \eta^{*}(\cdot\mid s)  ds\big\| \exp(-s) \check \Lambda^*( \cdot \mid s) \otimes G^{\clv}(\check \bfL^*(s))(\cdot, \cdot ) ds 
 \right)\\
 &= \int_0^\infty \exp(-s) R\left( \check \eta^{*}(\cdot\mid s) \big\|
 \check \Lambda^*( \cdot \mid s) \otimes G^{\clv}(\check \bfL^*(s))(\cdot, \cdot )
 \right) ds.
\end{align*}
The claimed equality now follows on noting the relationship between $(\check\eta^*, \check \Lambda^*, G^{\clv})$ and 
$(\eta^*, \eta^*_{(1)}, G)$ noted in \eqref{eq:etastartoetastarcheck}, \eqref{eq:xnew3}, and \eqref{eq:ggv}.} 

The result follows on letting $\veps \to 0$.
\end{proof}

\section{Laplace Lower Bound}\label{sec:lowbd}
We now proceed to the   large deviation  lower bound.
The main result in this direction is the following.
\begin{theorem}\label{thm:lowbd}
	Suppose that Assumption  \ref{ass1} is satisfied.
	Then, for every $F \in C_b(\clp(\Delta^o))$,
	\begin{equation*}
	\limsup_{n\to \infty}-n^{-1} \log E \exp[-n F(L^{n+1})] \le \inf_{m \in \clp(\Delta^o)} [F(m) + I(m)],
	\end{equation*}
 where $I$ is defined using the matrix $A$ in Assumption  \ref{ass1}.
\end{theorem}
The proof of the above theorem is completed in this section and the next, and in both of these sections Assumption \ref{ass1} is assumed to hold throughout and the matrix $A$ is as in this assumption.

{\subsection{Overview}\label{sec:overview5.1} We begin with an overview of this section. The starting point is to select a $m^0$ that is $\varepsilon$-optimal for the right side in the above display, where $\varepsilon >0$ is a fixed small parameter. Next, we select a control $\eta^0 \in \clu(m^0)$ which is $\varepsilon$-optimal for the control problem characterizing $I(m^0)$ through the right side of 
\eqref{def:ratefunction1} (with $m$ there replaced by $m^0$). 
In order to prove the lower bound, we will use the variational representation in Proposition \ref{prop:varrep0}. For this, the basic idea is to construct controlled sequences such that the corresponding $ \bar L^n (t_n)$ is  close to $m^0$ and the relative entropy cost
on the right side of \eqref{eq:varrep} is close to the integral on the right side of \eqref{def:ratefunction1} (with $(m,\eta)$ replaced by $(m^0,\eta^0)$). However, without any a priori guarantees on the \AB{smoothness in time}  of the near optimal control $\eta^0$, constructing controlled sequences, \AB{which are discrete time stochastic processes} with the desired properties, is not straightforward. Addressing this is the main goal of this section. We proceed by providing a series of approximations to replace $(M^0, \eta^0, m^0)$ with quantities that have better regularity properties.  \AB{The eventual goal of this section is to replace $\eta^0$ with a piecewise constant control which is easier to approximate by constructing  stochastic control sequences. }

The first issue to deal with is the possible blowup of the relative entropy costs in \eqref{def:ratefunction1} when the measure in the second coordinate of the relative entropy places small measure on certain sets. \PZ{This step is needed in order to control the errors in the relative entropy costs when controls and state trajectories are replaced by their approximations. } This issue is addressed in Section \ref{subsec:non-degen} by a perturbation argument and by using properties of the fixed point $\pi^*$ in Assumption \ref{ass1}(3). This leads to an approximation 
$(M^1, \eta^1, m^1)$ for $(M^0, \eta^0, m^0)$ for which $m^1$ and the associated cost (as given by the integral on right side of \eqref{def:ratefunction1}) are  close to $m^0$ and the cost for
$(M^0, \eta^0, m^0)$, respectively (see \eqref{eq:1599zz}). Furthermore, \AB{for this approximation one has uniform positivity of the measure in the second argument of relative entropy and, consequently, uniform bounds  on relative entropy costs (see  \eqref{eq:relentetarhom1}, \eqref{eq:1243}).} \PZ{To be more precise, the objective here is to ensure a uniform bound (over $s$) of the form
\begin{equation}\label{eq:nondegNb}
   	R\left(\eta^1( s) \| \eta^{1}_{(1)}( s)\otimes  G(M^{1}(s)) \right)  \le  C, 
\end{equation}
which is done by establishing a uniform lower bound, for the measure in the second argument of the relative entropy, of the form
\begin{equation}\label{eq:nondegN}
    \inf_{s \in \RR^+} \inf_{(x,y) \in A_+} \eta^{1}_{(1)}(x \mid s) G(M^{1}(s)) (x,y) \ge \varepsilon,
\end{equation}
for some constants $C,\varepsilon \in (0,\infty)$; see \eqref{eq:relentetarhom1} and \eqref{eq:1243} for the precise bounds. The key idea in constructing $\eta^1$ (and the corresponding $M^1,m^1$) is to slightly perturb each $\eta^0(s)$ in the direction of  $\pi^*G(\pi^*)$, namely to set
\begin{equation*}
	\eta^{\kappa}(x,y \mid s) = (1-\kappa)\eta^{0}(x,y \mid s) + \kappa \pi^{*}_x G(\pi^*)_{x,y}
\end{equation*}
for some small $\kappa \in (0,1)$; see \eqref{eq:917m}. This is is the only place in the proof where the linearity assumption on $G$ (Assumption \ref{ass1}(1)) is invoked}.

Since the controlled empirical measures in the variational representation of Proposition \ref{prop:varrep0} and the measure valued trajectories $M(\cdot)$ in the definition of the rate function are asymptotically related through a time-reversal operation, in Section \ref{sec:fixedt} we also introduce the time reversal of the quantities $(M^1, \eta^1)$, denoted as $(\hat M^1, \hat \eta^1)$ and the associated cost given by the right side of \eqref{eq:sizeofT}. \PZ{These are given by the following formulas for a fixed sufficiently large $T$
\begin{equation*}
\hat M^1(t) \doteq M^1(T-t), \; \hat \eta^1(t) = \hat\eta^1(\cdot \mid t) \doteq \eta^1(\cdot \mid T-t).
\end{equation*}}
\PZ{The exact choice of $T$ is given in \eqref{eq:sizeofT}.} 
\AB{We remark that the role of this time change is to undo the time reversal that led to the expression for the rate function  from weak convergence of controlled sequences denoted by 
$\check{}$ in Section 
\ref{sec:timerev} (cf. Lemma \ref{lemchar}). For this reason we denote these time reversals by a different notation, namely 
 $\hat{}$.}
It is these time reversed controls that will be eventually used to construct the controlled chains with desired asymptotic properties.

However, it is not clear how to use a control given as a function of a continuous parameter $t \in \RR_+$ to construct a discrete time controlled sequence, and \AB{one would need to construct some  time discretization of this continuous parameter control. However, for such a discretization to well approximate the continuous parameter control, one needs the control to have continuous dependence on $t$,} which is not guaranteed a priori. This is the objective of Step 2 given in Section \ref{sec:step2cty} where we approximate the control $\hat \eta^1$ from Step 1 by a continuous control by using a   \PZ{time mollification of the form
\begin{equation*}
	\hat\eta^{1,\kappa}(s) \doteq
	 \kappa^{-1}\int_{s}^{\kappa +s} \hat\eta^1(u) du, \; s \in [0,T].
\end{equation*}}
The \AB{nondegeneracy estimates for the relative entropy terms from the previous step (of the form in \eqref{eq:nondegNb}-\eqref{eq:nondegN})} ensure that the cost associated with this approximation is  close to the cost from Step 1. This step culminates in the construction of  $(\hat M^2, \hat \eta^2 , m^2)$  that \PZ{gives a good approximation to} $(\hat M^1, \hat \eta^1 , m^1)$
(see \eqref{eq:5eps}) and has the additional property that $\hat \eta^2$ is continuous in $t$.

This time continuity is exploited in Section \ref{sec:piecewiseconstant} for Step 3 of the approximation where the  control $\hat \eta^2$ is approximated by a piecewise constant control \PZ{of the form
\begin{equation}
	\hat \eta^{2,\kappa} (\cdot \mid s) \doteq \sum\limits_{j=0}^{\lfloor T \kappa^{-1}\rfloor - 1} \hat \eta^2 (\cdot \mid j\kappa) \bm{1}_{[j\kappa, (j+1)\kappa)}(s) + 
 \hat \eta^2 (\cdot \mid \lfloor T \kappa^{-1}\rfloor\kappa) \bm{1}_{[\lfloor T \kappa^{-1}\rfloor\kappa, T]}(s)
 ,\;\; s \in [0,T].
\end{equation}} 
This step culminates with the construction of $(\hat M^3, \hat \eta^3 , m^3)$  that well approximates $(\hat M^2, \hat \eta^2 , m^2)$ \PZ{for small enough $\kappa$} (see \eqref{eq:601fin}) and with the property that $\hat \eta^3$ is piecewise constant.

The construction of such a piecewise near optimal control is the main objective of this section since such a control can be `approximately replicated' in a natural fashion for the controlled self interacting chains and used to obtain the desired lower bound. Details on how this is carried out will be given at the start of Section \ref{sec:llb}.

We now proceed to implement the above outline. We begin in the following subsection with the selection of a near-optimal control and trajectory and then in successive sections, by a series of approximations, we  modify  these quantities, culminating in  Section \ref{sec:piecewiseconstant} in the final form of the simple form piecewise constant controls  that will be used in the lower bound proof.  
}

\subsection{Near-Optimal Control}\label{sec:prelimest}
Fix $F \in C_b(\clp(\Delta^o))$ and $\veps \in (0,1)$.
In order to prove Theorem \ref{thm:lowbd} we can assume without loss of generality that $F$ is Lipschitz (see \cite[Corollary 1.10]{buddupbook}), i.e., for some $F_{\mbox{\tiny{lip}}} \in (0,\infty)$,
\[
|F(m) - F(\tilde m)| \le F_{\mbox{\tiny{lip}}} \|m - \tilde m\|, \;  m, \tilde m \in \clp(\Delta^o).
\]
Choose $m^0 \in \clp(\Delta^o)$ such that 
\begin{equation}
\label{eq:300}
	F(m^0) + I(m^0) \le  \inf_{m \in \clp(\Delta^o)} [F(m) + I(m)] + \veps.
\end{equation}

Recall the definition of the rate function from \eqref{def:ratefunction1} given in terms of $\eta \in \clu$.
In proofs it will sometimes be convenient to work with analogues of $\eta$ that are probability measures on $\clp(\clv^d \times \clv^d)$. In particular,  for $\eta \in \clu$, we define $\eta^{\clv} : \RR_+  \to \clp(\clv^d \times \clv^d)$ as 
\begin{equation}\label{eq:etacalv1}
\eta^{\clv}(s)(\bdelta_x, \bdelta_y) \doteq \eta(x,y \mid s) , \; x,y \in \Delta^o.
\end{equation}
Observe that the map defined in \eqref{def:ratefunction1} can be rewritten as
\begin{equation}\label{eq:ratefunction2}
I(m) = \inf\limits_{\eta \in \clu(m)} \int_0^{\infty} \exp(-s)  \sum\limits_{x \in \Delta^o}  \eta^{\clv}_{(1)}(\bdelta_x  \mid s) R\left( \eta^{\clv}_{2|1} (\cdot \mid s, \bdelta_x ) \|  G^{\clv}(M(s))(\bdelta_x , \cdot)\right) ds,
\end{equation}
where $\eta^{\clv}(s)(\bdelta_x, \bdelta_y) = \eta^{\clv}_{(1)}(\bdelta_x  \mid s)  \eta^{\clv}_{2|1} (\cdot \mid s, \bdelta_x )$.
Also note that the relative entropy in \eqref{def:ratefunction1} is computed for probability measures on $\Delta^o$ while the relative entropy in \eqref{eq:ratefunction2} is computed for probability measures on $\clv^d$.

We choose $\eta^0 \in \clu(m^0)$ such that, with $\eta^{0,\clv}$ defined by the right side of \eqref{eq:etacalv1} (with $\eta$ replaced with $\eta^0$),

\begin{equation}
\label{eq:301}
	\int_0^{\infty} \exp(-s) \sum_{x \in \Delta^o}   \eta^{0,\clv}_{(1)}(\bdelta_x  \mid s) R\left( \eta^{0,\clv}_{2|1} (\cdot \mid s, \bdelta_x ) \|  G^{\clv}(M^0(s))(\bdelta_x , \cdot)\right)
 \le I(m^0) + \veps,
\end{equation}
where $M^0$ solves $\clu(m^0,\eta^0)$. From the definition of $\clu(m_0)$, there is a $\MP^0 \in C([0, \infty): \clp(\Delta^o\times \Delta^o))$ such that, for each $t \in \RR_+$,  
\begin{equation}\label{eq:xnew1}
(\MP^0(t))_{(1)} = (\MP^0(t))_{(2)} = M^0(t) \mbox{ and } \MP^0(t)(x,y) = 0 \mbox{ whenever } (x,y) \in (A_+)^c.
\end{equation}

We now modify $M^0$ and $\eta^0$ to construct a more tractable near-optimal trajectory. 
{The key facts used in the next section are equations \eqref{eq:301} and \eqref{eq:xnew1}.}

\subsection{Step 1: Ensuring Nondegeneracy}
\label{subsec:non-degen} 
Our first approximation step ensures that the probability measure  that appears in the second argument of the relative entropy terms  of the form in \eqref{eq:301} are suitably nondegenerate. \AB{Specifically, we construct approximations $M^1, \eta^1$ to $M^0, \eta^0$ that satisfy \eqref{eq:1243} in the lemma below. This ensures that the relative entropies that appear in 
\eqref{eq:358} have uniform upper bounds, which will be needed in the next step of the approximation.}
Let, for each $z \in \Delta^o$,
$\Delta_{+}(z) \doteq \{y \in \Delta^o : A_{z,y}  > 0\}$,
and recall the constant $\delta_0^A \in (0,\infty)$ from Assumption \ref{ass1}\eqref{ass:aux1}.

\AB{\begin{lemma}\label{lem:rev2A}
There exist $(m^1, \eta^1, M^1, \MP^1)$ with $m^1 \in \clp(\Delta^0)$, 
$\eta^1 \in \clu(m^1)$, $M^1$ solving $\clu(m^1, \eta^1)$ and $\MP^1 \in C(\RR_+: \clp(\Delta^0\times \Delta^0))$ such that for all $t \in \RR_+$,
$(\MP^1(t))_{(1)} = (\MP^1(t))_{(2)} = M^1(t)$, $\mbox{supp}(\MP^1(t)) = A_+$, and for some $\delta, \delta_0^{M_1} \in (0,\infty)$,
 \begin{multline}
  F(m^{1}) + \int_0^{\infty} \exp(-s) \sum_{x \in \clv^d} \eta^{1}_{(1)}(x\mid s) R\left(\eta^{1}_{2|1}(\cdot \mid s,x) \| G(M^{1}(s))(x, \cdot)\right) ds\\
  \le \inf_{m \in \clp(\Delta^o)} [F(m) + I(m)] + 3 \veps,\label{eq:358}
  \end{multline}
  and for all $s \in \RR_+$,
  
  \begin{equation}\label{eq:relentetarhom1}
	R\left(\eta^{1}(s) \| \eta^{1}_{(1)}( s)\otimes G(M^{1}(s)) \right)  \le  \left |\log \delta_0^{M_1}\right|,
\end{equation}
and
\begin{equation}\label{eq:1243}
\inf\limits_{x\in \Delta^o} M^1(s)(x) \geq \delta, \quad
\inf\limits_{(x,y) \in \tplus} \eta^1(x,y \mid s) \geq  \delta, \; \mbox{ and }
\text{supp}( \eta^1(\cdot \mid s)) =\tplus.
\end{equation}

\end{lemma}}
\begin{proof}
Note that from  Assumption \ref{ass1}\eqref{ass:qsdG}, there is a   $\pi^* \in \clp_+(\Delta^o)$ satisfying
\begin{equation}\label{eq:defnqsd}
	\sum_{x \in \Delta^o} \pi^*_x G(\pi^*)_{x,y} =  \pi^*_y, \; y \in \Delta^o.
\end{equation}
\PZ{To ensure notational consistency, it}  will be helpful to consider the  measure $ \pi^{\clv,*} \in \mathcal{P}(\clv^d)$ defined by
$
\pi^{\clv, *}_{\bdelta_x} \doteq \pi^*_x$, $x \in \Delta^o$,
so that
\[
\sum\limits_{v \in \clv^d} \pi^{\clv,*}_v G^{\clv}(\pi^*)_{v,u} = \pi^{\clv,*}_u,  \; u \in \clv^d.
\]
Let
\begin{equation}
\delta_0^{\pi^*} \doteq \inf\limits_{x\in\Delta^o}\pi^*_x, \; \; 
\delta_0^{G,\pi^*} \doteq  \inf\limits_{(x,y)\in \tplus}\pi^*_x G(\pi^*)_{x,y},
\end{equation}
and note that, {from Assumption \ref{ass1}.2(b),}
\begin{equation}
\label{eq:dGpi}
    \delta_0^{G,\pi^*} \geq\delta_0^A (\delta_0^{\pi^*})^2 > 0.
\end{equation}
 
Let, for $x,y \in \Delta^o$ and $s\in \RR_+$,
\[
M^*(s) \doteq \pi^*, \quad \eta^*( x,y \mid s) \doteq  \pi^*_x G(\pi^*)_{x,y},
\]
and observe that $\eta^* \in \clu(\pi^*)$ and
$M^*$ solves $\clu(\pi^*,\eta^*)$. 
 Define, for $\kappa \in (0,1)$ and $t \in \RR_+$,
\begin{align}\label{eq:mkappacomb}
	M^{\kappa}(t) \doteq (1-\kappa) M^0(t) + \kappa M^*(t),\;
	\eta^{\kappa}(\cdot \mid t) &\doteq
	(1-\kappa)\eta^{0}(\cdot \mid t) + \kappa \eta^{*}(\cdot \mid t), \;
	m^{\kappa} \doteq (1-\kappa) m^0 + \kappa \pi^*,
\end{align}
and observe that with
\[
\MP^*(t) \doteq \pi^* G(\pi^*), \; \MP^{\kappa}(t) \doteq (1-\kappa)\MP^0(t) + \kappa \MP^*(t), \;\; t \in \RR_+,
\]
we have that
$(\MP^{\kappa}(t))_{(1)} = (\MP^{\kappa}(t))_{(2)} = M^{\kappa}(t)$. 
Also note that, since $M^0(t) \in \clp^*(\Delta^o)$ and $M^*(t) \in \clp^*(\Delta^o)$, we have that $M^{\kappa}(t) \in \clp^*(\Delta^o)$ for every $t \in \RR_+$, in fact we have that $\text{supp}(\MP^{\kappa}(t)) = A_+$ for each $t \in \RR_+.$
From these observations we see that
$\eta^{\kappa} \in \clu(m^{\kappa})$ and $M^{\kappa}$ solves $\clu(m^{\kappa},\eta^{\kappa})$. 

 For each $t \in \RR_+$ and $x,y\in \Delta^o$,  define
\begin{equation}\label{eq:taukdef}
	\tau^{\kappa}(x,y \mid t) = \frac{ \kappa(1 - \kappa) \eta^0(x,y \mid t) +  \kappa \pi^*_x G(\pi^*)_{x,y}}{2 \kappa(1 - \kappa) + \kappa^2},
\end{equation}
and note that for each $x,y \in \Delta^o$ and $ t \in \RR_+$,
\begin{equation}\label{eq:917m}
\begin{aligned}
	\eta^{\kappa}(x,y \mid t) &= (1-\kappa)\eta^{0}(x,y \mid t) + \kappa \pi^{*}_x G(\pi^*)_{x,y}\\
	&= (1-\kappa)^2\eta^{0}(x,y \mid t) + (2\kappa(1-\kappa)+\kappa^2) \tau^{\kappa}(x,y \mid t).
\end{aligned}
\end{equation}
Also, observe that
\begin{equation}
\eta^{\kappa}_{(1)}(x\mid s) = (1-\kappa) \eta^0_{(1)}(x\mid s) + \kappa \pi^*_x, \label{eq:etakb1}
\end{equation}
and, by {the linearity property from} 
\PZ{Assumption \ref{ass1}\eqref{ass:linear},
\begin{equation}\label{eq:gmks}
\begin{aligned}
	G(M^{\kappa}(s))_{x,y} = G((1-\kappa)M^{0}(s) + \kappa  \pi^*)_{x,y}  
	= (1-\kappa)G(M^{0}(s))_{x,y} + \kappa G(\pi^*)_{x,y}.
\end{aligned}
\end{equation}}
Thus, from the previous two displays, 
\begin{equation}\label{eq:916m}
	\begin{aligned}
	\eta^{\kappa}_{(1)}(x\mid s)G(M^{\kappa}(s))_{x,y} &=
	(1-\kappa)^2 \eta^0_{(1)}(x\mid s)G(M^{0}(s))_{x,y} + \kappa(1-\kappa)\eta^0_{(1)}(x\mid s)G(\pi^*)_{x,y}\\
	&\quad + \kappa(1-\kappa) \pi^*_xG(M^{0}(s))_{x,y}  + \kappa^2 \pi^*_xG(\pi^*)_{x,y}\\
	&= (1-\kappa)^2 \eta^0_{(1)}(x\mid s)G(M^{0}(s))_{x,y} + 
	\left(2\kappa(1-\kappa)+\kappa^2\right)\sigma^{\kappa}(x,y \mid s),
	\end{aligned}
\end{equation}
where for $x,y \in \Delta^o$ and $t\in \RR_+$,
\begin{equation}\label{eq:sigmkdef}
	\begin{aligned}
	\sigma^{\kappa}(x,y \mid t) &\doteq \frac{  \kappa(1-\kappa) \eta^0_{(1)}(x\mid s)G(\pi^*)_{x,y} +
\kappa(1-\kappa) \pi^*_xG(M^{0}(s))_{x,y} + \kappa^2 \pi^{*}_x G(\pi^*)_{x,y} }{2\kappa(1-\kappa)+\kappa^2}.
\end{aligned}
\end{equation}
\AB{Using the convexity of the map $(P,Q) \mapsto R(P\|Q)$ 
(cf. \cite[Lemma 2.4(b)]{buddupbook})}, and combining \eqref{eq:917m} and \eqref{eq:916m}, we see that {\begin{equation}
\label{eq:545}
\begin{aligned}
 &\int_0^{\infty} \exp(-s) \sum_{x \in \Delta^o} \eta^{\kappa}_{(1)}(x\mid s) R\left(\eta^{\kappa}_{2|1}(\cdot \mid s,x) \| G(M^{\kappa}(s))(x, \cdot)\right) ds\\
 &\quad =  \int_0^{\infty} \exp(-s)  R\left(\eta^{\kappa}( s) \| \eta^{\kappa}_{(1)}( s)\otimes G(M^{\kappa}(s))\right)  ds\\
  &\quad \le (1-\kappa)^2 \int_0^{\infty} \exp(-s)  R\left(\eta^{0}( s) \| \eta^{0}_{(1)}( s)\otimes G(M^{0}(s))\right) ds\\
  &\qquad + \left(2\kappa(1-\kappa) +\kappa^2\right) \int_0^{\infty} \exp(-s)  R\left(\tau^{\kappa}(\cdot\mid s) \| \sigma^{\kappa}(\cdot\mid s)\right) ds\\
    &{\quad = (1-\kappa)^2 \int_0^{\infty} \exp(-s) \sum\limits_{x \in \Delta^o} \eta^0_{(1)}(x \mid s) R\left(\eta^{0}_{2 \mid 1}(\cdot\mid s, x) \| G(M^{0}(s))(x, \cdot)\right) ds}\\
  &\qquad {+ \left(2\kappa(1-\kappa) +\kappa^2\right) \int_0^{\infty} \exp(-s)  R\left(\tau^{\kappa}(\cdot\mid s) \| \sigma^{\kappa}(\cdot\mid s)\right) ds,}
 \end{aligned}
  \end{equation}
  }
  {where we have  used the chain rule for relative entropies (Theorem \ref{thm:chainrule}) to obtain the first and second equalities}.
 Observe from Assumption \ref{ass1}\eqref{ass:222a} that, for each $s \in \RR_+$,
$\text{supp}(\sigma^{\kappa}( \cdot \mid s)) =  \tplus$,
and for all $t \in \RR_+$  and $(x,y)  \in \tplus$,
\[
| \log \sigma^{\kappa}( x,y \mid t) | \leq \left| \log \left(\frac{\kappa^2}{2\kappa(1-\kappa) + \kappa^2}  \pi^*_x G(\pi^*)_{x,y}\right)\right|.
\]
Combining the observation  in the previous display
with \eqref{eq:taukdef}, and \eqref{eq:sigmkdef},
 it follows that, for $s\in \RR_+$ and $\kappa \in (0,1/2)$
 \begin{align*}
 	R\left(\tau^{\kappa}(\cdot\mid s) \| \sigma^{\kappa}(\cdot\mid s)\right)
	&\le \sum_{(x,y)  \in \tplus} \tau^{\kappa}(x,y\mid s) |\log \sigma^{\kappa}(x,y\mid s)|\\
	&\le  d\left|\log \left(\frac{\kappa^2}{2\kappa(1-\kappa)+\kappa^2}\right)\right| + \sum_{(x,y) \in \tplus}
 \tau^{\kappa}(x,y\mid s) |\log (\pi^{*}_x G(\pi^*)_{x,y})| \\
	&\le  d\left(\left|\log \left(\frac{\kappa}{2}\right)\right| + \left|\log \delta_0^{G, \pi^*}\right|\right).
 \end{align*}
 Thus,
 \begin{align}\label{eq:1599zz}
	\left(2\kappa(1-\kappa) +\kappa^2\right) \int_0^{\infty} \exp(-s)  R\left(\tau^{\kappa}(\cdot\mid s) \| \sigma^{\kappa}(\cdot\mid s)\right) ds
	\le  \left(2\kappa(1-\kappa) +\kappa^2\right)  d\left(\left|\log \left(\frac{\kappa}{2}\right)\right| +\left|\log \delta_0^{G, \pi^*}\right|\right).
\end{align}
Choose $\kappa_1 \in (0,1/2)$  such that 
\begin{equation}\label{eq:1172}
\| m^0 - m^{\kappa_1}\|  \le  \frac{\min\{(F_{\mbox{\tiny{lip}}})^{-1} , 1\}\veps}{2}, \;\;
\left(2\kappa_1(1-\kappa_1) +(\kappa_1)^2\right)  d\left(\left|\log \left(\frac{\kappa_1}{2}\right)\right| + \left|\log \delta_0^{G, \pi^*}\right|\right) \le \veps/2.
\end{equation}
For  convenience, write $ (m^1, \eta^1,  M^1, \MP^1) \doteq (m^{\kappa_1}, \eta^{\kappa_1}, M^{\kappa_1}, \MP^{\kappa_1})$.
 Then, 
  \begin{multline}
  F(m^{1}) + \int_0^{\infty} \exp(-s) \sum_{x \in \clv^d} \eta^{1}_{(1)}(x\mid s) R\left(\eta^{1}_{2|1}(\cdot \mid s,x) \| G(M^{1}(s))(x, \cdot)\right) ds\\
  \le F(m^0) + \int_0^{\infty} \exp(-s) \sum_{x \in \clv^d} \eta^{0}_{(1)}(x\mid s) R\left(\eta^{0}_{2|1}(\cdot \mid s,x) \| G(M^{0}(s))(x, \cdot)\right) ds+ \veps\\
  \le \inf_{m \in \clp(\Delta^o)} [F(m) + I(m)] + 3 \veps,\label{eq:358b}
  \end{multline}
where we have used the Lipschitz property of $F$, \eqref{eq:545}, \eqref{eq:1599zz}, and \eqref{eq:1172} for the first inequality, and the displays in \eqref{eq:300} and \eqref{eq:301} for the second inequality.
This proves \eqref{eq:358b}.

Note that, for each $s \in \RR_+$, 
$
 \text{supp}( \eta^1_{(1)}( s) \otimes G(M^1(s))) = \tplus,
$
and that, {from  the first equality in \eqref{eq:916m} and \eqref{eq:dGpi}}, with  $\delta_0^{M_1} \doteq \kappa^2_1 \delta_0^{G,\pi^*}$,
\begin{equation}\label{eq:rhoboundlower}
\begin{split}
\inf_{s \in \RR_+}\inf_{(x,y) \in \tplus}  G(M^1(s))(x, y)  \ge \inf_{s \in \RR_+} \inf_{(x,y) \in \tplus} \eta^{1}_{(1)}(x\mid s) G(M^1(s))(x, y) \ge \delta_0^{M_1} > 0, 
\end{split}
\end{equation}
which implies that for each $s \in \RR_+$ \eqref{eq:relentetarhom1} is satisfied.

Also, from \eqref{eq:dGpi} it follows that $\kappa_1 \delta_0^{G,\pi^*} \le \kappa_1 \delta^{\pi^*}_0$ and so, for each $s \in \RR_+$, with 
 $\delta \doteq \kappa_1 \delta_0^{G,\pi^*}$, \eqref{eq:1243} is satisfied.
 This completes the proof of the lemma.
\end{proof}
\subsection{Time Reversal}\label{sec:fixedt}
Recall from the proof of the upper bound that the trajectories in the variational problem in the Laplace upper bound  are 
related to the limit controlled trajectories by  time reversal (see e.g., \eqref{eq:526} and last display in Section \ref{sec:laplaceupperbound}). Thus, we now introduce a time reversal of $M^1$, which, after further approximations, will be used to construct suitable controlled
processes in the proof of the lower bound.
Fix $T \in (0,\infty)$  large enough so that 
\begin{equation}\label{eq:sizeofT}
\exp(-T+1)\max\left\{\left|\log \left(\frac{\delta \delta_0^A}{8}\right)\right|,  \left|\log \delta_0^{M_1}\right|\right\} \le \veps,
\end{equation}
\PZ{where $\delta$ and $\delta_0^{M_1}$ are as in the statement of Lemma \ref{lem:rev2A} and $\delta_0^A$ is as in Assumption \ref{ass1}\eqref{ass:aux1}.}
Throughout this section and the next,  \PZ{these values of $T$, $\delta$, and $\delta_0^{M_1}$ are fixed.}  Define, for $t \in [0,T]$,
\begin{equation}\label{eq:mhatetahatdef1A}
\hat M^1(t) \doteq M^1(T-t), \; \hat \eta^1(t) = \hat\eta^1(\cdot \mid t) \doteq \eta^1(\cdot \mid T-t),
\end{equation}
\AB{where $M^1$ and $\eta^1$ are as in the statement of Lemma \ref{lem:rev2A}. Throughout this section, these definitions of $M^1, \eta^1, \hat{M}^1$, and $\hat{\eta}^1$ are fixed.} Note that, since $M^1$ solves $\clu(m^1, \eta^1)$, 
\begin{equation}\label{eq:hatmit}
\begin{split}
 \hat{M}^1(t) = M^1(T) + \int_0^t \hat\eta_{(1)}^1(s) ds - \int_0^t \hat M^{1}(s) ds, \; t \in [0,T],
 \end{split}
\end{equation}
where $\hat\eta_{(1)}^1(s) = \eta_{(1)}^1(T-s)$.
Recalling the non-negativity of relative entropy, note that,
\begin{multline}\label{eq:firscostapp}
\int_0^{\infty} \exp(-s) \sum_{x \in \Delta^o} \eta^{1}_{(1)}(x \mid s)R\left(\eta^{1}_{2|1}(\cdot \mid s,x) \| G(M^{1}(s))(x,\cdot)\right) ds \\
\ge \exp(-T) \int_0^T \exp(s) \sum_{x \in \Delta^o} \hat \eta^{1}_{(1)}(x \mid s) R\left(\hat\eta^{1}_{2|1}(\cdot \mid s,x) \big\| G(\hat M^{1}(s)) (x,\cdot)\right) ds,
\end{multline}
where $\hat\eta^{1}_{2|1} (\cdot \mid s,x)= \eta^{1}_{2|1} (\cdot \mid T-s,x)$.  

Finally, disintegrating $\MP^1(T)$ as 
\begin{equation}\label{eq:Qdef1}
\MP^1(T)(x,y) = M^1(T)(x) Q(x,y), \; \;  x,y \in \Delta^o,
\end{equation}
 we have that $M^1(T)$ is a stationary distribution of the Markov chain with transition probability kernel $Q$. Also, on recalling that $\text{supp}(M^1(T)) = \Delta^o$ and $\text{supp}(\MP^1(T)) = A_+$, we see that  $Q(x,y) = 0$ if and only if $(x,y) \in (A_+)^c$, from which it follows that the kernel $Q$ is irreducible and has unique stationary distribution $M^1(T)$.

\subsection{Step 2: Continuity of Control}
\label{sec:step2cty}

Our next step mollifies the control $\hat \eta^1$ in a suitable manner so that it can be discretized at a later step. \AB{This  step provides an estimate of the form in \eqref{eq:220} for the mollified approximation $\hat \eta^2$ of $\hat \eta^1$  which allows us to control the discretization error at the next step of the approximation.}

\AB{\begin{lemma}\label{lem:rev2B}
There exist $(m^2, \hat \eta^2, \hat M^2)$ with $m^2 \in \clp(\Delta^0)$, 
$\hat \eta^2 \in \clu$, $\hat M^2 \in C([0,T]: \clp(\Delta^0))$  such that
\begin{multline}
  F(m^{2}) + e^{-T}\int_0^{T} \exp(s) \sum_{x \in \Delta^o} \hat \eta^{2}_{(1)}(x \mid s) R\left(\hat\eta^{2}_{(2)}(\cdot \mid s,x) \big\| G(\hat M^{2}(s)) (x,\cdot)\right) ds\\
 \le \inf_{m \in \clp(\Delta^o)} [F(m) + I(m)] + (4 + L_G) \veps,\label{eq:5eps}
  \end{multline}
  \begin{equation}\label{eq:1233aa}
	\hat M^2(t) =  M_1(T) + \int_0^t \hat\eta_{(1)}^2(s) ds - \int_0^t \hat M^{2}(s) ds, \; t \in [0,T],\;\;
 \hat M^2(T) = m_2,
\end{equation}
\begin{equation}\label{eq:220}
	\|\hat \eta^2(s) - \hat \eta^2(t)\|  \le C_1 |s-t|, \; \; s, t \in [0,T],
\end{equation}
and with $\delta$ as in Lemma \ref{lem:rev2A},
\begin{equation} \label{eq:219aa}
	\inf\limits_{s\in [0,T]} \inf\limits_{(x,y) \in \tplus} \hat\eta^{2}( x,y \mid s) \ge \delta, \;\; \inf_{s \in [0,T], x \in \Delta^o}\hat M^{2}(s)(x) \ge \delta /2.
\end{equation}

\end{lemma}}

\begin{proof}
For $\kappa >0$, define
\begin{equation}\label{eq:etahat1kappadef}
	\hat\eta^{1,\kappa}(s) \doteq
	 \kappa^{-1}\int_{s}^{\kappa +s} \hat\eta^1(u) du, \; s \in [0,T],
\end{equation}
where $\hat \eta^1(u) \doteq \hat \eta^1(T)$ for $u\ge T$.
Also, define for $t \in [0,T]$,
\begin{equation}\label{eq:1153}
	\hat M^{1,\kappa}(t) \doteq M_1(T) + \int_0^t \hat\eta_{(1)}^{1,\kappa}(s) ds - \int_0^t \hat M^{1,\kappa}(s) ds.
\end{equation}
Note that there is a unique $\hat M^{1,\kappa} \in C([0,T]: \RR^d)$ that solves \eqref{eq:1153}, and that this $\hat M^{1,\kappa}$ satisfies, for each $s \in [0,T]$, 
$
\sum\limits_{x \in \Delta^o}\hat M^{1,\kappa}(s)(x) = 1$.
We now show that for $\kappa$ sufficiently small we have 
$
\inf_{s \in [0,T]} \inf_{x\in\Delta^o}\hat M^{1,\kappa}(s)(x) > 0$,
namely that the solution to \eqref{eq:1153} in fact belongs to $C([0,T]: \clp_+(\Delta^o))$.
We can write, for $t \in [0,T]$,
\begin{equation}\label{eq:1134}
\hat M^{1,\kappa}(t) 
= M_1(T) + \int_0^{t}  \hat\eta^1_{(1)}(s)  ds 
- \int_0^t \hat M^{1,\kappa}(s) ds + \clr_1^{\kappa}(t),
\end{equation}
where 
\[
\clr_1^{\kappa}(t) \doteq \int_0^{t+\kappa}\hat\eta_{(1)}^1(u)\kappa^{-1}\int_{(u-\kappa)^+}^{u\wedge t} dsdu - \int_0^t\hat\eta^1_{(1)}(u)du.
\]
Observe that, for each $t \in [0,T]$,
$\| \clr^{\kappa}_1(t)\| \le 3\kappa$. Combining this estimate with \eqref{eq:hatmit}
and \eqref{eq:1134}, we have,  for $t \in [0,T]$,
\begin{equation*}
\|  \hat{M}^{1,\kappa}(t) - \hat{M}^1(t)\| \le 3\kappa +  \int_0^t \|  \hat{M}^{1,\kappa}(s) - \hat{M}^1(s)\| ds,
\end{equation*}
from which we see, by an application of Gr\"{o}nwall's lemma, that
\begin{equation}\label{eq:mkappam1diffbound}
\sup_{t \in [0,T]} \|\hat M^{1,\kappa}(t) - \hat M^1(t)\| \le 3\kappa \exp(T).
\end{equation}
Recall the definition of $\delta$ from above \eqref{eq:1243}, and let
\begin{equation}\label{eq:c1defn}
c_1 \doteq  2 \left(\delta^2 \delta_0^A\right)^{-1}.
\end{equation}
Assume that $\kappa$ is small enough so that 
\begin{equation}\label{eq:ledelt2}
3\kappa \exp(T) \le  \min\left\{ \frac{\veps}{2} \min\{1, F_{\text{lip}}^{-1}\}, \frac{\veps}{2c_1}, \frac{\delta}{2}\right\},
\end{equation}
and
\begin{equation}\label{eq:1417aa}
2c_1L_G\kappa + \kappa (e^{1 - T} + 1) \left| \log \delta^{M_1}_0\right| < \frac{\veps}{2}.
\end{equation}
This, in particular, in view of \eqref{eq:1243} and \eqref{eq:mkappam1diffbound}, ensures that $\hat M^{1,\kappa} \in C([0,T]: \clp_+(\Delta^o))$, and in fact
\begin{equation}\label{eq:945n}
\inf_{s \in [0,T]}\inf_{ x \in \Delta^o}\hat M^{1,\kappa}(s)(x) \ge \delta /2. 
\end{equation}
{Next, recalling the definition of $\hat{\eta}^{1,\kappa}$ from \eqref{eq:etahat1kappadef}, we write }
\begin{align}\label{eq:1340aa}
&\int_0^T \exp(s) R\left(\hat \eta^{1, \kappa}( s) \big\| \hat \eta_{(1)}^{1, \kappa}( s) \otimes G(\hat M^{1, \kappa}(s))\right)ds \nonumber\\
&= \int_0^T \exp(s) R\left(\kappa^{-1} \int_s^{\kappa + s} \hat \eta^{1}( u) du \big\| \kappa^{-1} \int_s^{\kappa + s}\hat \eta_{(1)}^{1}( u) \otimes G(\hat M^{1}(u)) du\right)ds + \clr_1,
\end{align}

where 
\begin{align}\label{eq:1341a}
\clr_1 & = 
 \int_0^T \exp(s) R\left(\kappa^{-1} \int_s^{\kappa + s} \hat \eta^{1}( u) du \big\| \kappa^{-1} \int_s^{\kappa + s}\hat \eta_{(1)}^{1}( u) \otimes G(\hat M^{1, \kappa}(s)) du\right) ds \nonumber\\
&\quad - \int_0^T \exp(s) R\left(\kappa^{-1} \int_s^{\kappa + s} \hat \eta^{1}( u) du \big\| \kappa^{-1} \int_s^{\kappa + s}\hat \eta_{(1)}^{1}( u) \otimes G(\hat M^{1}(u)) du\right) ds .
\end{align}

Observe from \eqref{eq:1243}, \eqref{eq:c1defn} and \eqref{eq:945n} that, for $s \in [0,T]$ and $(x,y) \in \tplus$,
\begin{align}\label{eq:1351a}
&\left| \log \left(\kappa^{-1} \int_s^{\kappa + s}\hat \eta_{(1)}^{1}(x \mid u)  G(\hat M^{1, \kappa}(s))(x, y) du\right)- 
\log \left(\kappa^{-1} \int_s^{\kappa + s}\hat \eta_{(1)}^{1}(x \mid u)G(\hat M^{1}(u))(x, y) du\right) \right| \nonumber \\
& \le c_1   \kappa^{-1} \left|  \int_s^{\kappa + s}\hat \eta_{(1)}^{1}(x \mid u)  G(\hat M^{1, \kappa}(s))(x, y) du - 
 \int_s^{\kappa + s}\hat \eta_{(1)}^{1}(x \mid u)G(\hat M^{1}(u))(x, y) du\right|.
\end{align}
{Also,  for each $s \in [0,T]$,
\begin{align}\label{eq:1370a}
& \kappa^{-1}\sum\limits_{(x,y) \in \tplus}   \left| \int_s^{\kappa + s}\hat \eta_{(1)}^{1}(x \mid u)  G(\hat M^{1, \kappa}(s))(x, y) du - 
\int_s^{\kappa + s}\hat \eta_{(1)}^{1}(x \mid u)G(\hat M^{1}(u))(x, y) du\right| \nonumber \\
& \le \kappa^{-1} \sum\limits_{(x,y) \in \tplus}\int_s^{\kappa + s}\hat \eta_{(1)}^{1}(x \mid u) \left| G(\hat M^{1, \kappa}(s))(x, y) - G(\hat{M}^1(s))(x,y)\right| du\\
&\quad +  \kappa^{-1} \sum\limits_{(x,y) \in \tplus}\int_s^{\kappa + s}\hat \eta_{(1)}^{1}(x \mid u) \left| G(\hat{M}^1(s))(x,y) - G(\hat{M}^1(u))(x,y)\right| du
\\
&\quad \leq    L_G \left( \| \hat M^{1,\kappa}(s) - \hat M^{1}(s)\| + \kappa^{-1} \int_s^{\kappa + s} \| \hat M^{1}(s) - \hat M^1(u)\|du \right) \; \le  L_G \left(  \frac{\veps}{2 c_1}  + 2\kappa \right),
\end{align}
where the second inequality follows from Assumption \ref{assu:lip} and the fact that $\hat{\eta}^{1}_{(1)}(x \mid u) \le 1$ for all $x \in \Delta^o$ and $u \in \RR_+$, and the last inequality follows from \eqref{eq:hatmit}, \eqref{eq:mkappam1diffbound},  and \eqref{eq:ledelt2}, and the observations that $\|\hat \eta^1_{(1)}\|=1$, $\|\hat M^1(s)\|=1$.}
Combining  \eqref{eq:1341a}, \eqref{eq:1351a}, and \eqref{eq:1370a} we see that 
\begin{align}\label{eq:1383aa}
e^{-T}\clr_1 \le e^{-T} \int_0^T L_G \exp(s)\left( \frac{\veps}{2}  + 2\kappa c_1 \right) ds \leq   L_G \left( \frac{\veps}{2} + 2 c_1 \kappa \right) .
\end{align}
For $u \in\RR_+$, let

\begin{equation}\label{eq:r2udef}
\clr_2(u) \doteq R\left( \hat \eta^{1}( u) \big\| \hat \eta_{(1)}^{1}( u) \otimes G(\hat M^{1}(u)) \right).
\end{equation}

Using \eqref{eq:relentetarhom1}, we see that, for each $u\in \RR_+$,
$
\clr_2(u) \leq  \left| \log \delta_0^{M_1}\right| .$
Using the convexity of relative entropy and \eqref{eq:r2udef}, we now have that
\begin{align}\label{eq:1396aa}
&\int_0^T \exp(s) R\left(\kappa^{-1} \int_s^{\kappa + s} \hat \eta^{1}( u) du \big\| \kappa^{-1} \int_s^{\kappa + s}\hat \eta_{(1)}^{1}( u) \otimes G(\hat M^{1}(u)) du\right) ds \nonumber \\
&\quad \le \int_0^T \exp(s) \kappa^{-1} \int_s^{\kappa + s}  R\left( \hat \eta^{1}( u) \big\| \hat \eta_{(1)}^{1}( u) \otimes G(\hat M^{1}(u)) \right) du \, ds \nonumber \\
&\quad = \int_0^T \exp(s) \kappa^{-1} \int_s^{\kappa + s} \clr_2(u) du\, ds,
\end{align}
Next, on recalling that $\kappa^{-1}(1 - e^{-\kappa}) \leq 1$, it is easily checked that  
\begin{multline}\label{eq:1401aa}
\int_0^T \exp(s) \kappa^{-1} \int_s^{\kappa + s} \clr_2(u) du\, ds
= \kappa^{-1}\int_0^{T+\kappa} \clr_2(u) \int_{(u-\kappa)^+}^{u\wedge T} \exp(s)  ds \, du\\
\le \kappa^{-1}\int_0^{\kappa} (\exp(u) - 1)\clr_2(u)   \, du + 
\kappa^{-1}e^T(1- e^{-\kappa})\int_T^{T+\kappa} \clr_2(u) du
+ \kappa^{-1} (1- e^{-\kappa})\int_\kappa^T  \exp(u) \clr_2(u) du\\
\le \kappa ( e + e^T) \left| \log \delta^{M_1}_0\right|+ \int_0^T \exp(u) R\left( \hat \eta^{1}( u) \big\| \hat \eta_{(1)}^{1}( u) \otimes G(\hat M^{1}(u)) \right) du.
\end{multline}
Combining the estimates in \eqref{eq:1340aa}, \eqref{eq:1383aa}, \eqref{eq:1396aa}, and \eqref{eq:1401aa},  we have 
\begin{equation*}
\begin{split}
&e^{-T}\int_0^T \exp(s) R\left(\hat \eta^{1, \kappa}( s) \big\| \hat \eta_{(1)}^{1, \kappa}( s) \otimes G(\hat M^{1, \kappa}(s))\right)\\
&\quad \le L_G \left(\frac{\veps}{2} + 2c_1\kappa \right) + \kappa (e^{1 - T} + 1) \left| \log \delta^{M_1}_0\right|\\
&\qquad + e^{-T} \int_0^T \exp(u) R\left( \hat \eta^{1}( u) \big\| \hat \eta_{(1)}^{1}( u) \otimes G(\hat M^{1}(u)) \right) du.
\end{split}
\end{equation*}
Now denote by $\kappa_2$ the constant $\kappa$ that satisfies \eqref{eq:ledelt2} and \eqref{eq:1417aa}.
Let
\[
(\hat\eta^{2}, \hat M^{2}, m^2) \doteq (\hat\eta^{1,\kappa_2}, \hat M^{1,\kappa_2}, \hat M^{1,\kappa_2}(T)).
\]
Note from \eqref{eq:mkappam1diffbound} and \eqref{eq:ledelt2} that 
\[
|F(m^2) - F(m^1)| = |F(\hat M^{1,\kappa_2}(T)) - F(\hat M^{1}(T))| \le \veps/2.
\]
Thus,  
 \begin{multline}
  F(m^{2}) + e^{-T}\int_0^{T} \exp(s) \sum_{x \in \Delta^o} \hat \eta^{2}_{(1)}(x \mid s) R\left(\hat\eta^{2}_{(2)}(\cdot \mid s,x) \big\| G(\hat M^{2}(s)) (x,\cdot)\right) ds\\
  \le L_G \veps/2   + \veps/2 + F(m^1)  + e^{-T} \int_0^T \exp(u) R\left( \hat \eta^{1}( u) \big\| \hat \eta_{(1)}^{1}( u) \otimes G(\hat M^{1}(u)) \right) du\\
  \le \inf_{m \in \clp(\Delta^o)} [F(m) + I(m)] + (4 + L_G) \veps,\label{eq:5epsb}
  \end{multline}
  where for the last inequality we have used \eqref{eq:358} and \eqref{eq:firscostapp}. This proves \eqref{eq:5eps}.
  Furthermore,
  from
\eqref{eq:1153} we see that \eqref{eq:1233aa} is satisfied.

By construction, $\hat \eta^2 \in \mathcal{C}(\RR_+ : \mathcal{P}(\Delta^o \times \Delta^o))$, and we can find  $C_1 \doteq C_1(\kappa_2) \in (0,\infty)$ such that \eqref{eq:220} holds
and, on recalling \eqref{eq:1243} and \eqref{eq:945n}, observe that the estimates in \eqref{eq:219aa} are satisfied. This completes the proof of the lemma.
\end{proof}
    
\subsection{Step 3: Piecewise Constant Approximation}
\label{sec:piecewiseconstant}
Now we carry out the last step in the approximation, which is to replace continuous controls by piecewise constant controls as in the last statement in the next lemma. \PZ{This step  allows us to reduce a possibly uncountable family of measures indexed by time $s \in \RR_+$ to a family of finitely many measures. This makes the construction of near-optimal controls in the next section  a more tractable task.}
\AB{\begin{lemma} \label{lem:rev2C}
There exist $\kappa_3 >0$ and $(m^3, \hat \eta^3, \hat M^3)$ with $m^3 \in \clp(\Delta^0)$, 
$\hat \eta^3 \in \clu$, $\hat M^3 \in C([0,T]: \clp(\Delta^0))$  such that
\begin{multline}
  F(m^{3}) + e^{-T}\int_0^{T} \exp(s) \sum_{x \in \Delta^o} \hat \eta^{3}_{(1)}(x \mid s) R\left(\hat\eta^{3}_{(2)}(\cdot \mid s,x) \big\| G(\hat M^{3}(s)) (x,\cdot)\right) ds\\
  \le  \inf_{m \in \clp(\Delta^o)} [F(m) + I(m)] + (6 + 2L_G) \veps, \label{eq:601fin}
  \end{multline}
  \begin{equation}\label{eq:1233}
	\hat M^3(t) =  M^1(T) + \int_0^t \hat\eta_{(1)}^3(s) ds - \int_0^t \hat M^{3}(s) ds, \; t \in [0,T],\;\;
 \hat M^3(T) = m_3.
\end{equation}
Furthermore, with $\delta$ as in  Lemma \ref{lem:rev2A}
\begin{equation} \label{eq:219}
	\inf\limits_{s\in [0,T]} \inf\limits_{(x,y) \in \tplus} \hat\eta^{3}( x,y \mid s) \ge \delta, \;\; \inf_{s \in [0,T], x \in \Delta^o}\hat M^{3}(s)(x) \ge \delta /4.
\end{equation}
Finally, for each $j \in\{ 0, 1, \ldots, \lfloor T\kappa_3^{-1}\rfloor -1\}$, the map $t \mapsto \hat \eta^3(\cdot \mid t)$ is constant over the interval $[j\kappa_3, (j+1)\kappa_3)$, as well as over the interval 
$[\lfloor T\kappa_3^{-1}\rfloor \kappa_3 , T]$.
\end{lemma}}
\begin{proof}
\AB{Let $\hat{\eta}^2$ and $\hat{M}^2$ be as in the statement of Lemma \ref{lem:rev2B}.} 
 For $\kappa > 0$, define $\hat \eta^{2,\kappa}$ as
\begin{equation}
	\hat \eta^{2,\kappa} (\cdot \mid s) \doteq \sum\limits_{j=0}^{\lfloor T \kappa^{-1}\rfloor - 1} \hat \eta^2 (\cdot \mid j\kappa) \bm{1}_{[j\kappa, (j+1)\kappa)}(s) + 
 \hat \eta^2 (\cdot \mid \lfloor T \kappa^{-1}\rfloor\kappa) \bm{1}_{[\lfloor T \kappa^{-1}\rfloor\kappa, T]}(s)
 ,\;\; s \in [0,T].
\end{equation}
  Let $\hat M^{2,\kappa}$ solve the equation
\begin{equation}\label{eq:1208}
	\hat M^{2,\kappa}(t) =  M^1(T) + \int_0^t \hat\eta_{(1)}^{2,\kappa}(s) ds - \int_0^t \hat M^{2,\kappa}(s) ds, \; t \in \RR_+.
\end{equation}
Then, with $\clr^{\kappa}_2(t) \doteq \int_0^t \hat\eta_{(1)}^{2,\kappa}(s)ds - \int_0^t \hat\eta_{(1)}^2(s)ds$,
we have that, for $t \in [0,T]$,
\begin{equation}\label{eq:1232}
\hat M^{2,\kappa}(t) =  M_1(T) + \int_0^t \hat\eta_{(1)}^2(s) ds - \int_0^t \hat M^{2,\kappa}(s) + \clr_2^{\kappa}(t).
\end{equation}
From \eqref{eq:220} and the definition of $\hat \eta^{2,\kappa}$, 
\[
\sup_{t \in [0,T]}\|\clr_2^{\kappa}(t) \| \le T \left(\sup_{s,t \in [0,T], |s-t| \le \kappa} \|\hat\eta^2(s) - \hat\eta^2(t)\| 
\right)\le C_1\kappa T,
\]
Combining the last estimate,  \eqref{eq:1233aa},  and  \eqref{eq:1232},  we have from Gr\"{o}nwall's lemma that
\begin{equation}\label{eq:221}
	\sup_{t \in [0,T]} \|\hat M^{2,\kappa}(t) - \hat M^2(t)\| \le C_1\kappa T\exp({T}).
\end{equation}
Assume that $\kappa$ is sufficiently small so that
\begin{equation}\label{eq:c1288}
C_1\kappa T\exp({T}) \le \min\left\{ \frac{\veps}{2} \min\{1, F_{\text{lip}}^{-1}\}, \frac{\veps}{4c_1}, \frac{\delta}{4}\right\},
\end{equation}
and 
\begin{equation}\label{eq:c1288b}
2\kappa (2L_G + C_1)(2 | \log c_1| + c_1) + 4 \kappa c_1 L_G \le \veps,
\end{equation}
where $c_1$ is defined in \eqref{eq:c1defn}. Then, it follows from \eqref{eq:221} and \eqref{eq:219aa} that for this choice of $\kappa$,
\begin{equation}
    \label{eq:nonfinit-M2k}
    \inf_{s \in [0,T]}\inf_{ x \in \Delta^o} \hat M^{2,\kappa}(s)(x) \ge \frac{ \delta}{4} , \;\;
   \inf\limits_{s\in [0,T]} \inf\limits_{(x,y) \in \tplus} \hat\eta^{2,\kappa}(x,y \mid s) \ge  \delta.
\end{equation}
This shows that $\hat M^{2,\kappa}$ belongs to  $C([0,T]: \clp_+(\Delta^o))$. 
For $t \in [0,T]$, let $\alpha_{\kappa}(t) \doteq \lfloor t \kappa^{-1}\rfloor\kappa$
and
 write 
\begin{align}\label{eq:1712aa}
&\int_0^T \exp(s) R\left(\hat \eta^{2, \kappa}( s) \big\| \hat \eta_{(1)}^{2, \kappa}( s) \otimes G(\hat M^{2, \kappa}(s))\right)ds \nonumber\\
&\quad = \int_0^T \exp(s) R\left(\hat \eta^{2, \kappa}( s) \big\| \hat \eta_{(1)}^{2, \kappa}( s) \otimes G(\hat M^{2}(\alpha_{\kappa}(s)))\right)ds + \clr_3,
\end{align}
where 
\begin{align}\label{eq:1717aa}
\clr_3 & \doteq 
 \int_0^T \exp(s) R\left(\hat \eta^{2, \kappa}( s) \big\| \hat \eta_{(1)}^{2, \kappa}( s) \otimes G(\hat M^{2, \kappa}(s))\right)ds \nonumber\\
&\quad - \int_0^T \exp(s) R\left(\hat \eta^{2, \kappa}( s) \big\| \hat \eta_{(1)}^{2, \kappa}( s) \otimes  G(\hat M^{2}(\alpha_{\kappa}(s)))\right)ds.
\end{align}
{Recalling the definition of $c_1$ from \eqref{eq:c1defn}, we have, from Assumption \ref{ass1} \eqref{ass:aux} and \eqref{eq:nonfinit-M2k}, that for each $s \in [0,T]$,
\begin{equation}\label{eq:etahatgxyunifbound1}
\inf\limits_{(x,y) \in \tplus} \hat{\eta}^{2,\kappa}_{(1)} (x \mid s) G(\hat{M}^{2,\kappa})(x,y) \ge \frac{\delta^2 \delta_0^A}{4} = \frac{1}{2c_1}.
\end{equation}
Together, \eqref{eq:etahatgxyunifbound1}, \eqref{eq:219aa}, and the mean value theorem ensure that 
\begin{align}\label{eq:boundlogdiffeta2gm2}
 &\left|\log\left(\hat \eta_{(1)}^{2, \kappa}(x \mid s)G(\hat M^{2, \kappa}(s))(x, y) \right)
- \log\left(\hat \eta_{(1)}^{2, \kappa}(x \mid s)G(\hat M^{2}(\alpha_{\kappa}(s)))(x, y)\right)\right|\\
&\quad \le 2c_1\left|\hat \eta_{(1)}^{2, \kappa}(x \mid s)G(\hat M^{2, \kappa}(s))(x, y) 
- \hat \eta_{(1)}^{2, \kappa}(x \mid s)G(\hat M^{2}(\alpha_{\kappa}(s)))(x, y)\right|.
\end{align}
Then, for $(x,y) \in \tplus$ and $s\in [0,T]$,
\begin{align}\label{eq:1723aa}
& \sum\limits_{(x,y) \in \tplus} \left|\log\left(\hat \eta_{(1)}^{2, \kappa}(x \mid s)G(\hat M^{2, \kappa}(s))(x, y) \right)
- \log\left(\hat \eta_{(1)}^{2, \kappa}(x \mid s)G(\hat M^{2}(\alpha_{\kappa}(s)))(x, y)\right)\right| \nonumber\\
&\quad \le 2 c_1 \sum\limits_{(x,y) \in \tplus}  \left|\hat \eta_{(1)}^{2, \kappa}(x \mid s)G(\hat M^{2, \kappa}(s))(x, y) - \hat \eta_{(1)}^{2, \kappa}(x \mid s)G(\hat M^{2}(\alpha_{\kappa}(s)))(x, y)\right| \nonumber\\
&\quad \le 2 L_G c_1\left( \| \hat M^{2,\kappa}(s) - \hat M^2(s)\| +\|  \hat M^2(s) -\hat M^2(\alpha_{\kappa}(s))\|\right)  \le 2L_G c_1\left(\frac{\veps}{4 c_1} + 2\kappa\right)  = L_G \left( \frac{\veps}{2} + 4\kappa c_1 \right) ,
\end{align}
where the first inequality is due to \eqref{eq:boundlogdiffeta2gm2}, the second inequality is due to Lipschitz property in Assumption \ref{assu:lip} and the triangle inequality, and the last inequality is due to  \eqref{eq:1233aa}, \eqref{eq:221}, \eqref{eq:c1288},
  the definition of $\alpha_k$, and the fact that  $\|\hat \eta^2_{(1)}(s)\| = \|\hat M^2(s)\|=1$.}
Combining  \eqref{eq:1717aa}, and \eqref{eq:1723aa}, we have that
\begin{equation}\label{eq:1735aa}
e^{-T} \clr_3 \le e^{-T}\int_0^T L_G \exp(s) \left(\frac{\veps}{2}+ 4\kappa c_1\right)ds \le L_G \left( \frac{\veps}{2} + 4 \kappa c_1 \right).
\end{equation}
Next, using   Assumption \ref{assu:lip}, \eqref{eq:1233aa}, and \eqref{eq:220}, note that, for $s,u \in [0,T]$ such that $|s - u| \le \kappa$,
\begin{align}\label{eq:1748aa}
&\sum\limits_{(x,y) \in \tplus} |\hat \eta_{(1)}^{2}(x \mid s)G(\hat M^{2}(s))(x, y)-
\hat \eta_{(1)}^{2}(x \mid u)G(\hat M^{2}(u))(x, y)| \nonumber \\
&\quad \le L_G \|\hat M^2(s) - \hat M^2(u)\| + \|\hat \eta_{(1)}^{2}(\cdot \mid s) - \hat \eta_{(1)}^{2}(\cdot \mid u)\|  \le \kappa(2 L_G + C_1),
\end{align}
and, using \eqref{eq:nonfinit-M2k}, observe that
\begin{equation}\label{eq:1755aa}
\inf_{s \in [0,T]} \inf\limits_{(x,y) \in \tplus}\hat \eta_{(1)}^{2}(x \mid s)G(\hat M^{2}(s))(x, y) \ge \frac{\delta^2  \delta_0^A }{4}= \frac{c_1^{-1}}{2}.
\end{equation}
Moreover, note that if, for some $\tilde{c} \in (0,1)$,  $a, \tilde a, b, \tilde b \in (\tilde{c}, 1]$, then
\begin{equation}\label{eq:1759aa}
\left|a\log (a/b) - \tilde a \log(\tilde a/\tilde b)\right| \le \left(2 |\log \tilde{c} | + \tilde{c}^{-1}\right)\left(|a-\tilde a| + |b- \tilde b|\right). 
\end{equation}
Using  \eqref{eq:1748aa}, \eqref{eq:1755aa}, and \eqref{eq:1759aa}, we have 
\begin{multline*}
\int_0^T \exp(s) R\left(\hat \eta^{2, \kappa}( s) \big\| \hat \eta_{(1)}^{2, \kappa}( s) \otimes G(\hat M^{2}(\alpha_{\kappa}(s)))\right)ds  \nonumber\\
\le \int_0^T \exp(s) R\left(\hat \eta^{2}( s) \big\| \hat \eta_{(1)}^{2}( s) \otimes G(\hat M^{2}(s))\right)ds\\ +
 2\kappa \left(2L_G + C_1 \right)\left(2 | \log c_1 | + c_1\right)\int_0^T \exp(s) ds.
\end{multline*}
Combining the estimate in the last display with \eqref{eq:1712aa} and \eqref{eq:1735aa}, we have 
\begin{multline}\label{eq:1769aa}
e^{-T}\int_0^T \exp(s) R\left(\hat \eta^{2, \kappa}( s) \big\| \hat \eta_{(1)}^{2, \kappa}( s) \otimes G(\hat M^{2, \kappa}(s))\right)ds  \\
\le e^{-T}\int_0^T \exp(s) R\left(\hat \eta^{2}( s) \big\| \hat \eta_{(1)}^{2}( s) \otimes G(\hat M^{2}(s))\right)ds  \\
+ 2 \kappa(2L_G + C_1 )(2 |\log c_1| + c_1) + L_G \left( \veps/2 + 4\kappa c_1 \right).
\end{multline}
Now, denote by $\kappa_3$  the constant $\kappa$ that satisfies  \eqref{eq:c1288} and  \eqref{eq:c1288b}
and let 
\begin{equation}\label{eq:mhat3defn}
(\hat\eta^{3}, \hat M^{3}, m^3) \doteq (\hat\eta^{2,\kappa_3}, \hat M^{2,\kappa_3}, \hat M^{2,\kappa_3}(T)).
\end{equation}
From \eqref{eq:221} and \eqref{eq:c1288} it follows that
$$\left|F(m^3)-F(m^2)\right| = \left|F(\hat M^{2, \kappa_3}(T)) - F(\hat M^2(T))\right| \le \veps/2.$$
Combining the last display with the estimate in \eqref{eq:1769aa} and recalling our choice of $\kappa_3$, we have 
 \begin{multline}
  F(m^{3}) + e^{-T}\int_0^{T} \exp(s) \sum_{x \in \Delta^o} \hat \eta^{3}_{(1)}(x \mid s) R\left(\hat\eta^{3}_{(2)}(\cdot \mid s,x) \big\| G(\hat M^{3}(s)) (x,\cdot)\right) ds\\
  \le L_G \veps/2 + \veps/2 + \veps + F(m^2)  + e^{-T} \int_0^T e^s  R\left( \hat \eta^{2}( u) \big\| \hat \eta_{(1)}^{2}( u) \otimes G(\hat M^{2}(u) \right) du\\
  \le \inf_{m \in \clp(\Delta^o)} [F(m) + I(m)] + (6 + 2L_G) \veps, \label{eq:601finb}
  \end{multline}
  where for the last inequality we have used \eqref{eq:5eps}. This proves \eqref{eq:601fin}.
Furthermore,
  from
\eqref{eq:1208}, we see that \eqref{eq:1233} is satisfied.
By construction, $\hat \eta^3$ is piecewise constant as in the statement of the lemma.

and, from \eqref{eq:nonfinit-M2k}, \eqref{eq:219} is satisfied.
This completes the proof of the lemma.
\end{proof}

\section{Proof of Laplace Lower Bound}
\label{sec:llb}
In this section we will prove Theorem \ref{thm:lowbd} by constructing a  sequence of controlled processes based on the quantities $\hat \eta^3,\hat M^3$ from Section \ref{sec:piecewiseconstant}.  {The only facts needed from the previous section for this section are those in equations \eqref{eq:601fin}-\eqref{eq:219}.}
{\subsection{Outline}
We begin with a proof outline. We will use the variational representation in Proposition \ref{prop:varrep0}. The idea is to construct a specific controlled sequence for which the expectation on the right side of \eqref{eq:varrep} is \PZ{arbitrarily} close to the left side of \eqref{eq:601fin}. In particular, we want the interpolated controlled process $\bar L^n$ to be such that $\bar L^n(t_n)$ is approximately $m^3$ and the cost in the second term on the right side of \eqref{eq:varrep} is close to the second term on the first line of  \eqref{eq:601fin}.

In Section \ref{sec:constructcontrols} we provide the detailed construction of the controlled chain with the above properties. This section first gives a high level idea for this construction and then proceeds to describe the key steps in the construction in Sections \ref{sec:step1} - \ref{sec:step3}. The precise definition of the controlled chain is given in Section \ref{sec:contchain}. Figure \ref{fig:construction} gives an algorithmic representation 
of the construction. Section \ref{subsec:conv-proc} is devoted to studying the asymptotic behavior of the continuous time, time reversed controlled process $\hat M^n$ defined in \eqref{lem:mbardeflem7}. \AB{The key fact that $\hat M^n(T)$ is close to $\hat M(T) = m^3$ is made precise in Corollary \ref{cor:1101}}, which is the main result of this section. Section \ref{subsec:conv-costs} establishes the key fact on approximation of costs showing that the cost associated with the controlled chain given by the second term on the right side of \eqref{eq:varrep} is asymptotically close to the second term on the first line of \eqref{eq:601fin}. The main result which establishes this fact is Lemma \ref{lem:costest}. Finally in Section \ref{sec:pflowbd} we put together the results from Corollary \ref{cor:1101} and Lemma \ref{lem:costest} to complete the proof of the lower bound. \AB{A table summarizing some notation used throughout this section is provided in Appendix \ref{sec:notationtablepart2}.}

}
\subsection{Construction of Suitable Controls}\label{sec:constructcontrols}
To simplify notation, write
\begin{equation}\label{eq:etahatmhatcq}
\hat \eta \doteq \hat \eta^3, 
\;\; 
 \hat M \doteq \hat M^3,
\;\; c \doteq \kappa_3, \;\; q \doteq M^1(T),
\end{equation}
 where ${M}^1$ is as in Section \ref{subsec:non-degen}, $\hat \eta^3, \hat M^3,$ and $\kappa_3$ are as in Section \ref{sec:piecewiseconstant}, \PZ{and $T$ is as in \eqref{eq:sizeofT}}.

The precise construction of the controlled processes will be given in a recursive fashion in Construction \ref{constr}. An informal outline of this construction is as follows:
\begin{enumerate}[label=Step \arabic*:, ref=Step \arabic*]
\item \label{constoutline:step1} Use the original uncontrolled dynamics until the empirical measure charges each point in $\Delta^o$. This is needed to ensure that the relative entropy costs are well controlled and it can be done due to Assumption \ref{ass1}\eqref{ass:conv-occ}. Note that this incurs zero cost since no control is exercised.
\item \label{constoutline:step2} Once the empirical measure of the original uncontrolled dynamics has charged each point in $\Delta^o$, proceed as follows. Recall the irreducible  transition probability kernel $Q$ as given at the end of Section \ref{sec:fixedt}  and note that $q$ is the unique stationary distribution for $Q$. Until time step $m(t_n-T) -1$, the controlled chain will use the kernel $Q$. Note that this approximately corresponds to evolving according to $Q$ until the interpolated continuous time instant  $t_n-T$. {By the ergodic theorem, the  empirical measure of the controlled process at time step $m(t_n - T)-1$ will, with high probability,  be very close to $q$.}
\item \label{constoutline:step3} Now, over the last $T$ units of interpolated time (i.e., until time $t_n$), we construct the chain successively in a manner that closely shadows the piecewise linear trajectory $\hat M$. Specifically, over each time segment $[lc, (l+1)c]$ over which $\hat M$ is linear and $\hat \eta(\cdot \mid s)$ is constant (in $s$) we will construct the chain, for time instants that correspond to -- in the  continuous time interpolation (as described in Section \ref{sec:controlproc}) --  the interval $[lc, (l+1)c]$, using the transition probability kernel 
$\hat \eta(\cdot \mid lc)$. Using the ergodic theorem again, the empirical measures  that are formed using this construction will be close to the trajectory $\hat M$ with high probability.
\end{enumerate}
In addition to the above steps, over the small probability events where deviations from the ergodic limits occur we will modify controls so that we expend no control cost.
The reader may want to keep the above rough outline in mind in what follows.

Recall the constants $\delta_0^A$ and $\delta$ defined in Part \ref{ass:aux} of Assumption \ref{ass1} and above \eqref{eq:1243}, respectively. Let $l_0 \doteq \lfloor Tc^{-1} \rfloor$ and define
\begin{equation}\label{eq:impconst}
\begin{aligned}
&b_1 \doteq  4+c, \;\; d_1 \doteq e^c(12 + c), \;\; d_2=6, \;\; d_3 \doteq d_1 + l_0 b_1 e^c, \;\; d_4 \doteq 2^{l_0}(3+d_2), \; \;
\delta_1 \doteq \delta \delta^A_0/8,\\
&A_1 \doteq  | \log \delta_1 | + (2 + d_3) \delta_1^{-1}, \; \; B_1 \doteq 2d_4 \delta_1^{-1}, \; \; C_1 \doteq | \log \delta_1| (l_0 + 3)^2.
\end{aligned}
\end{equation}
Fix $\veps_0, \veps_1>0$ sufficiently small so that
\begin{equation}\label{eq:conscho}
\veps_0 < \min\{c, \delta/16\}, \; F_{\mbox{\tiny{lip}}} (d_3 \veps_0 + 2d_4 \veps_1) \le \veps, \; C_1 \veps_1 + (l_0+1) (A_1 \veps_0 + B_1\veps_1) \le \veps.
\end{equation}
\subsubsection{Step 1}
\label{sec:step1}
We now proceed to {\ref{constoutline:step1}} in the above outline. Let $\{Z_i, \; i \in \NN_0\}$ be a sequence of $\Delta^o$-valued random variables defined recursively as follows. Recall $x_0 \in \Delta^o$ as fixed in Section \ref{sec:modeldescription} and let $Z_0\doteq x_0$. Having defined $Z_0,\ldots, Z_n$ for some $n \in \NN_0$, define the conditional law of $Z_{n+1}$ given $Z_0,\dots,Z_n$ by
\[
P\left(Z_{n+1}=x \mid \sigma\{Z_0, \ldots Z_n\}\right) \doteq  G(L^{n+1,Z})(Z_n,x), \; x \in \Delta^o,
\]
where $\{L^{n,Z}, \; n \in \NN\}$ is defined by 
$
L^{n,Z} \doteq \frac{1}{n} \sum_{i=0}^{n-1} \bdelta_{Z_i}$,  $n \in \NN$. 
From Assumption \ref{ass1}(4) and Lemma \ref{lem:posit} in the Appendix, it follows that  there is an $a^*>0$ and $r_1 \in \NN$, such that with
\begin{equation}\label{eq:255nn}
N_1(\omega) \doteq \inf\{k \in \NN: L^{k, Z}(\omega)(x)> a^* \mbox{ for all } x \in \Delta^o \},
\end{equation}
we have
\begin{equation}\label{eq:r1choice2050}
P(N_1 > r_1) \le \veps_1.
\end{equation}
{Later  (see Construction \ref{constr}(i)), we will use the sequence $\{Z_n, L^{n+1,Z}, \; n \in \NN_0\}$ to carry out {\ref{constoutline:step1}} of the outline above.}

\subsubsection{Step 2} Next, we proceed to {\ref{constoutline:step2}} of the outline given at the start of the section.
Recall $q = M^1(T)$
and the irreducible  transition probability kernel $Q$ as given at the end of Section \ref{subsec:non-degen}  which has $q$ as the unique stationary distribution.
Also recall that
\begin{equation}\label{eq:1945aaa}
\left(\tplus\right)^c =   \{ (x,y ) \in \Delta^o \times \Delta^o : Q(x,y) = 0\}.
\end{equation}
\AB{\begin{defn}
Let $Q^{\clv} \in \clk(\clv^d)$ be defined as
\[
Q^{\clv}(\bdelta_x,\bdelta_y) \doteq Q(x,y), \;\; x, y \in \Delta^o.
\]
Let $\{Y_i(x),\;  x \in \Delta^o, i \in \NN_0\}$ be iid $\Delta^o$-valued random vectors, independent of the sequence $\{Z_i, \; i \in \NN_0\}$, such that
$$P(Y_i(x) = y) = Q(x,y), \; x,y \in \Delta^o, \; i \in \NN_0.$$
For each $j \in \NN_0$, let
\[
\clg_j \doteq \sigma\left\{Z_l,  \; l\in \NN_0\right\} \vee \sigma\left\{ Y_i(x), \; x \in \Delta^o, 0\le i \le j\right\},
\]
and, for each $y \in \Delta^o$, define the sequence $\{\bar Y_i^y, \; i \in \NN_0\}$ of $\Delta^o$-valued random variables as $\bar Y^y_0 \doteq y$, and 
\[
P\left(\bar Y_{i+1}^y = z\mid \clg_{i+1} \vee \sigma\{\bar Y^y_j, 0\le j \le i\} \right)\doteq Q(\bar{Y}^y_{i}, z), \;\; z \in \Delta^o,   i \in \NN_0.
\]
\end{defn}}

By using the ergodic theorem for the transition probability matrix $Q$, we can find $k_0 \in \NN$ such that, with  
\[
A_{\veps_0} \doteq \left\{\om: \max_{y \in \Delta^o} \sup_{m\ge k_0} \left\|\frac{1}{m} \sum_{i=0}^{m-1} \bdelta_{\bar Y_i^y(\omega)} - q \right\| \ge \veps_0\right\},
\]
we have
\begin{equation}\label{eq:k0defn}
P(A_{\veps_0}) \le \veps_1, \mbox{ and } k_0 > 4 \epsilon_0^{-1}.
\end{equation}

\subsubsection{Step 3}\label{sec:step3} We now go on to {\ref{constoutline:step3}} of the outline given at the start of the section.
For that we   introduce some notation that is useful in describing the construction.
\AB{\begin{defn}
For each $n \in \NN$ let
\[
m^n_j \doteq m(t_n -T +jc), \;\;j = 0,1,\dots,l_0,
\]
where $l_0$ is defined above \eqref{eq:impconst}. For each $n \in \NN$ and   $j = 0, \dots, l_0$, let 
\begin{equation}\label{eq:clinjdef}
\begin{split}
\cli^{n,j} &\doteq \left\{ i \in \NN_0 : t_i \in [t_{m^n_j+1}, t_{m^n_{j+1}})\right\} = \{ m^n_j+1, m^n_j + 2, \dots, m^n_j + l^{n,j}\},
\end{split}
\end{equation}
where $l^{n,j} \doteq |\mathcal{I}^{n,j}|$ denotes the cardinality of each of these sets. 
\end{defn}}
Note that,
for all $n \in \NN$,
\[
m^n_j \doteq m^n_{j-1} + l^{n,j-1}+1,
\; \; j = 1,\dots, l_0.
\]

For $j = 0,1,\dots, l_0$, we define $\beta^j \in \clk(\Delta^o)$ as
\begin{align}\label{eq:betjdef4343}
\beta^j(x,y) \doteq \hat \eta(x,y \mid cj), \; x,y \in \Delta^o.
\end{align}
Such a $\beta^j$ can be disintegrated as
\begin{equation}\label{eq:betajdisinteg}
\beta^j(x,y) = \beta^j_{(1)} (x) \beta^j_{2|1}(y \mid x), \; x,y \in \Delta^o.\end{equation}
Recall from \eqref{eq:219} (and the fact that $A$ is irreducible from Part \ref{ass:aux} of  Assumption \ref{ass1}) that 
\begin{equation}\label{eq:2013aa}
\inf_{j =0 ,\dots,l_0}\inf_{x\in\Delta^o}\beta^j_{(1)}(x) \ge \delta, \;\; \inf_{j=0,\dots,l_0} \inf_{(x,y)\in A_{+}}\beta^j_{2|1}(y\mid x) \ge \delta.
\end{equation}
Also, by our construction of $\hat \eta$, for each $j = 0,1,\dots,l_0$,
\[
\sum_{x \in \Delta^o} \beta^j_{(1)} (x) \beta^j_{2|1}(y \mid x) =  \beta^j_{(1)} (y), \;\;  y\in \Delta^o.
\]
This is a consequence of the fact that $\eta^0$ introduced in Section \ref{sec:prelimest} belongs to $\clu(m^0)$  (see Property \ref{prop:z1}(b) in Section \ref{sec:statres}).
The above identity, together with \eqref{eq:2013aa}, says that $\beta^j_{(1)}$ is the unique stationary distribution of the Markov chain with an irreducible transition probability function
$\beta^j_{2|1}(\cdot \mid \cdot)$.
\AB{\begin{defn}
Let, for $j = 0,1,\dots,l_0$,  $\{U^j_i, \; i \in \NN_0\}$ be sequences of $\Delta^o$-valued random variables that are mutually independent of one another for different $j$, independent of $\{Z_j, \bar Y^y_j, \; j \in \NN_0, y \in \Delta^o\}$, and are distributed according to $U^j_0 \sim \beta^j_{(1)}$, and 
\[
 P(U^j_i=y \mid \sigma\{U^j_m, \; 0 \le m \le i-1\}) =
\beta^j_{2|1}(y\mid U^j_{i-1}), \;\; i \in \NN.
\]
\end{defn}}
Using the ergodic theorem, we can find $k^* \in \NN$ such that 
\begin{equation}\label{eq:236n}
P\left(\max_{j=0, \ldots , l_0}\,\sup_{m\ge k^*}\left\|\frac{1}{m+1} \sum_{i=0}^m \bdelta_{U^j_i}- \beta^j_{(1)}\right\| \ge \veps_0\right) \le \veps_1 
\end{equation}
and
\begin{equation}\label{eq:236nb}
\max_{j=0, \ldots , l_0}\max_{x \in \Delta^o}
E\left(\left. \left\|\frac{1}{k^*+1} \sum_{i=0}^{k^*} \bdelta_{U^j_i}- \beta^j_{(1)}\right\| \;\;\right|
U^j_0 = x\right) \le \veps_0. 
\end{equation}
 Now, fix $n_0 \in \NN$ large enough so that, for all $n \geq n_0$,
 \begin{equation}\label{eq:2055aa}
 m^n_0> k_0 +  r_1 + \lfloor 4(r_1+1)/\veps_0 \rfloor + 1 , \; \;  2(m^n_0+2)^{-1} \le \veps_0, \; m^n_0 c> k^*.
 \end{equation}
 \subsubsection{Controlled Chain}\label{sec:contchain}
Now, we piece the above main ingredients together to construct the controlled collection $\{\bar \nu^{n,k}, \bar L^{n,k}, \bar \mu^{n,k}, \; n \in \NN\}$ as follows.

\begin{construction} $\,$ Fix $n \ge n_0$.
 \label{constr}
\begin{enumerate}[label = (\roman*)]
    \item Let $x_0 \in \Delta^o$ be as fixed in Section \ref{sec:modeldescription}. Let $\bar X^n_0\doteq x_0$, $\bar{\nu}^{n,0} \doteq \bdelta_{x_0}$, $\bar L^{n,1} \doteq
    \bdelta_{x_0}$.
    Define, for $ k \in\{ 1,\dots, N_1(\om)\wedge r_1\}$,
    \[
    \bar X^n_k(\om) \doteq Z_k(\om), \; \bar \nu^{n,k}(\om) \doteq \bdelta_{\bar X^n_k(\om)}, \;
    \bar L^{n,k+1}(\om) \doteq L^{k+1,Z}(\om).
    \]
    Also, set 
    \[
    \bar \mu^{n, k}(\om)(\bdelta_y) \doteq G^{\clv}(\bar L^{n,k}(\om))(\bdelta_{\bar X^n_{k-1}(\om)}, \bdelta_y),\;
    y \in \Delta^o.
    \]
    On the `low probability' set $\{r_1<N_1(\om)\}$, we once more define $\bar X^n_k, \bar \nu^{n,k},  \bar L^{n,k+1}, \bar \mu^{n, k}$ by the above formulas for all $k>r_1$.
    
    \item On the `high probability' set $\{N_1(\om) \le r_1\}$ the construction proceeds as follows.
    Let $k_0$ be as introduced above \eqref{eq:k0defn} and let
    $
    k_1 \doteq  k_0 + \lfloor 4 (r_1+1)/\veps_0\rfloor + 1
    $. 
    For $ k \in \{N_1(\om)+1, \ldots , N_1(\om) + k_1\}$, define
    $
    \bar X^n_k(\om) \doteq \bar Y^{\bar X^n_{N_1(\om)}(\om)}_{k- N_1(\om)} (\om)$,
    and let 
    \begin{equation}\label{eq:repe}
    \bar \nu^{n,k}(\om) \doteq \bdelta_{\bar X^n_k(\om)}, \; \bar L^{n,k+1}(\om) \doteq \frac{1}{k+1} \sum_{j=0}^{k} \bdelta_{\bar X^n_j(\om)},
    \end{equation}
    and set
    \begin{equation}\label{eq:repeb}
    \bar \mu^{n, k}(\om)(\bdelta_y) \doteq Q(\bar X^n_{k-1}(\om), y), \; y \in \Delta^o.
    \end{equation}
    \item Again, on the set $\{N_1(\om) \le r_1\}$, let 
    $
    k_2(\om) \doteq N_1(\om)+ k_1$,
    and 
define 
\[
\tau^n(\om) \doteq \inf\{k \ge k_2(\om): \|\bar L^{n,k+1}(\om)- q\| > 2\veps_0\}, \;\;  k_3(\om) \doteq \tau^n(\om) \wedge (m^n_0-1).
\]
For $k \in \{k_2(\om)+1, \ldots,  k_3(\om)\}$, let 
$\bar X^n_k(\om) \doteq \bar Y^{\bar X^n_{k_2(\om)}(\om)}_{k- k_2(\om)} (\om)$,
and define $\bar \nu^{n,k}, \bar L^{n,k+1}, \bar \mu^{n,k}$ by \eqref{eq:repe} and \eqref{eq:repeb}.
Let  $\clj_0^n(\om) \doteq \bm{1}_{\{\tau^n(\om) \le m^n_0-1\}}$. Define
\begin{equation}\label{eq:2443aa}
\cld^n_0 \doteq \{N_1 (\om) \le r_1\} \cap \{ \clj_0^n(\om) = 1\},
\end{equation}
which, in view of \eqref{eq:k0defn}, is again a `low probability set'.
 On $\cld^n_0$, for $k\ge k_3(\om)$ let $\bar \mu^{n,k}$ and $\bar X^{n}_k$ be defined so that
\[
\bar \mu^{n, k}(\om)(\bdelta_y) \doteq G^{\clv}(\bar L^{n,k}(\om))(\bdelta_{\bar X^n_{k-1}(\om)}, \bdelta_y), \;\; y \in \Delta^o,
\]
and 
\[
P( \bar X^n_k  = y \mid \bar\clf^{n,k}) = \bar \mu^{n,k}(\bdelta_y), \;\; y \in \Delta^o,
\]
where, as in Section \ref{sec:controlproc}, $\bar \clf^{n,k} = \sigma \{ \bar L^{n,i}, \; 1 \leq i \leq k\}$, and $\bar \nu^{n,k}, \bar L^{n,k+1}$ are defined by \eqref{eq:repe}. This ensures that no cost is incurred on this low probability event.
\item  Now we give the construction on the `high probability' set
 $ \cle_0^n \doteq \{N_1(\om)\le r_1\}\cap \{\clj_0^n(\om)=0\}$.
\begin{itemize}
\item For $k =m_0^n$,  let 
\[
\bar \mu^{n, k}(\om)(\bdelta_y) \doteq 
\beta^0_{(1)}(y), \;\; y \in \Delta^o, 
\]
 define $\bar X^n_k \doteq U^0_0$, and note that, a.e. on $\cle^n_0$,
\[
P(\bar X^n_k = y \mid \bar \clf^{n, k}) =
\beta^0_{(1)}(y), \;\;  y \in \Delta^o,
\]
where $\bar \clf^{n,k} = \sigma \{ \bar L^{n,j}, 1 \le j \le k\}$, and  $\bar \nu^{n,k}, \bar L^{n,k+1}$ are defined by \eqref{eq:repe}.
\item Recall the set $\cli^{n,0}$ introduced in \eqref{eq:clinjdef}. Also, recall that we have chosen $k^*$ so that \eqref{eq:236n} and \eqref{eq:236nb} hold, and let
\begin{equation}\label{eq:taun0def}
\tau^{n,0}(\om) \doteq \inf\left\{m \ge k^*: \left\|\frac{1}{m+1} \sum_{i=0}^m \bdelta_{U^0_i(\om)} - \beta^0_{(1)}\right\| > \veps_0\right\}.
\end{equation}
For $k \in \{m^n_0 + 1, \dots, m^n_0 + (l^{n,0} \wedge \tau^{n,0}(\om))\}$,
 let $\bar X^n_k \doteq U^0_{k-m^n_0}$, and note that, a.e. on $\cle_0^n$,
\[P(\bar X^n_k = y \mid \bar \clf^{n, k}) =
\beta^0_{2|1}(y \mid \bar X^n_{k-1}), \;\; y \in \Delta^o,
\]
where $\bar \clf^{n,k}$ is defined as above.
Also,  for $k \in \{m^n_0 + 1, \dots, m^n_0 + (l^{n,0} \wedge \tau^{n,0}(\om))\}$, define $\bar \nu^{n,k}, \bar L^{n,k+1}$ by \eqref{eq:repe}, and
\[ \bar \mu^{n, k}(\om)(\bdelta_y) \doteq 
\beta^0_{2|1}(y \mid \bar X^n_{k-1}(\om)), \;\; y \in \Delta^o.
\]
If $\{\tau^{n,0}(\om) \le l^{n,0}\}$ occurs, then let, for $k \ge m^n_0 + \tau^{n,0}(\om) +1$,
$\bar \mu^{n,k}(\om)$ and  $\bar X^n_k(\om)$  be defined so that
\[
\bar \mu^{n, k}(\om)(\bdelta_y) \doteq G^{\clv}(\bar L^{n,k}(\om))(\bdelta_{\bar X^n_{k-1}(\om)}, \bdelta_y), \;\; y \in \Delta^o,
\]
and
\[
P(\bar X^n_k = y \mid \bar \clf^{n,k} ) = \bar \mu^{n,k} ( \bdelta_y) , \; \; y \in \Delta^o,
\]
 define $\bar \nu^{n,k}, \bar L^{n,k+1}$ by \eqref{eq:repe}, and let  $\clj_1^n(\om) = \bm{1}_{\{\tau^{n,0}(\om) \le l^{n,0}\}}$. Note that, by \eqref{eq:236n}, $\{\tau^{n,0}(\om) \le l^{n,0}\}$
 is a `low probability' event for large $n$ and so once more we are using uncontrolled (zero cost) dynamics on this event.
\end{itemize}
\item We now recursively extend the construction. Towards this end, suppose that, for some $l \in \{0,\dots, l_0-1 \}$, we have
defined the quantities
\begin{equation*}
\left\{\bar X^n_k(\om), \bar \nu^{n,k}(\om), \bar L^{n,k}(\om), \bar \mu^{n, k}(\om),\;k \in  \cup_{i=0}^l \left( \{m^n_i\} \cup  \mathcal{I}^{n,i}\right)\right\}
\end{equation*}
and  $\{\clj_i^n, \; 0 \leq i \leq l+1\}$. Let
\begin{equation}\label{eq:cale1}
\cle_{l+1}^n \doteq  \cle_{l}^n \cap \{ \clj^n_{l+1}(\om) = 0 \} = \{N_1(\om)\le r_1\}\cap \left(\cap_{i=0}^{l+1}\{\clj_i^n(\om)=0\}\right).
\end{equation}
Then on the `high probability' set $\cle_{l+1}^n$ we proceed as follows.
\begin{itemize}
\item For $k = m^n_{l+1}$, let 
\[
\bar \mu^{n,k}(\om)(\bdelta_y) \doteq \beta^{l+1}_{(1)}(y), \;\; y \in \Delta^o,
\]
define $\bar X^n_k \doteq U^{l+1}_0$, and note that, a.e. on $\cle_{l+1}^n$, 
\[
P(\bar X^n_k = y \mid \bar \clf^{n,k}) = \beta^{l+1}_{(1)}(y), \;\; y \in \Delta^o.
\] 
Also, define $\bar \nu^{n,k}, \bar L^{n,k+1}$ by \eqref{eq:repe}.

\item We now consider $k \in \mathcal{I}^{n,l+1}$. 
 Let
\[
\tau^{n,l+1}(\om) \doteq \inf\left\{m \ge k^*: \left\|\frac{1}{m+1} \sum_{i=0}^m \bdelta_{U^{l+1}_i(\om)} - \beta^{l+1}_{(1)}\right\| > \veps_0\right\}.
\]
For $k \in \{m^n_{l+1}  + 1, \dots, m^n_{l+1} + ( l^{n,l+1} \wedge \tau^{n,l+1}(\om))\}$, let $\bar X^n_{k} = U^{l+1}_{k-m^n_{l+1}}$, and note that, a.e. on $\cle_{l+1}^n$, 
\[
P( \bar X^n_k = y \mid \bar \clf^{n,k}) = \beta^{l+1}_{2|1}( y \mid \bar X^n_{k-1}),\;\; y \in \Delta^o,
\]
where $\bar \clf^{n,k}$ is defined as above. Also,  for $k \in \{m^n_{l+1} + 1, \dots, m^n_{l+1} + (l^{n,l+1} \wedge \tau^{n,l+1}(\om))\}$, define $\bar \nu^{n,k}, \bar L^{n,k+1}$ by \eqref{eq:repe}, and
\[ \bar \mu^{n, k}(\om)(\bdelta_y) \doteq 
\beta^{l+1}_{2|1}(y \mid \bar X^n_{k-1}(\om)), \;\; y \in \Delta^o.
\]
If $\{\tau^{n,l+1}(\om) \le l^{n,l+1}\}$ occurs, then let, for $k \ge m^n_{l+1} + \tau^{n,l+1}(\om) +1$,
$\bar \mu^{n,k}(\om)$ and  $\bar X^n_k(\om)$  be defined so that
\[
\bar \mu^{n, k}(\om)(\bdelta_y) \doteq G^{\clv}(\bar L^{n,k}(\om))(\bdelta_{\bar X^n_{k-1}(\om)}, \bdelta_y), \;\; y \in \Delta^o,
\]
and
\[
P(\bar X^n_k = y \mid \bar \clf^{n,k} ) = \bar \mu^{n,k} ( \bdelta_y) , \; \; y \in \Delta^o,
\]
 define $\bar \nu^{n,k}, \bar L^{n,k+1}$ by \eqref{eq:repe}, and let  $\clj_{l+2}^n(\om) \doteq \bm{1}_{\{\tau^{n,l+1}(\om) \le l^{n,l+1}\}}$.
 \end{itemize}
 
 \item On the event $\{l^{n,l_0} < \tau^{n,l_0}(\om)\}$ and for $k \ge m^n_{l_0} + \tau^{n,l_0}(\om) +1$, define $\bar \mu^{n,k}(\om)$ and  $\bar X^n_k(\om)$ by
\[
\bar \mu^{n, k}(\om)(\bdelta_y) \doteq G^{\clv}(\bar L^{n,k}(\om))(\bdelta_{\bar X^n_{k-1}(\om)}, \bdelta_y), \;\; y \in \Delta^o,
\]
and
\[
P(\bar X^n_k = y \mid \bar \clf^{n,k} ) = \bar \mu^{n,k} ( \bdelta_y) , \; \; y \in \Delta^o,
\]
and  define $\bar \nu^{n,k}, \bar L^{n,k+1}$ by \eqref{eq:repe}

 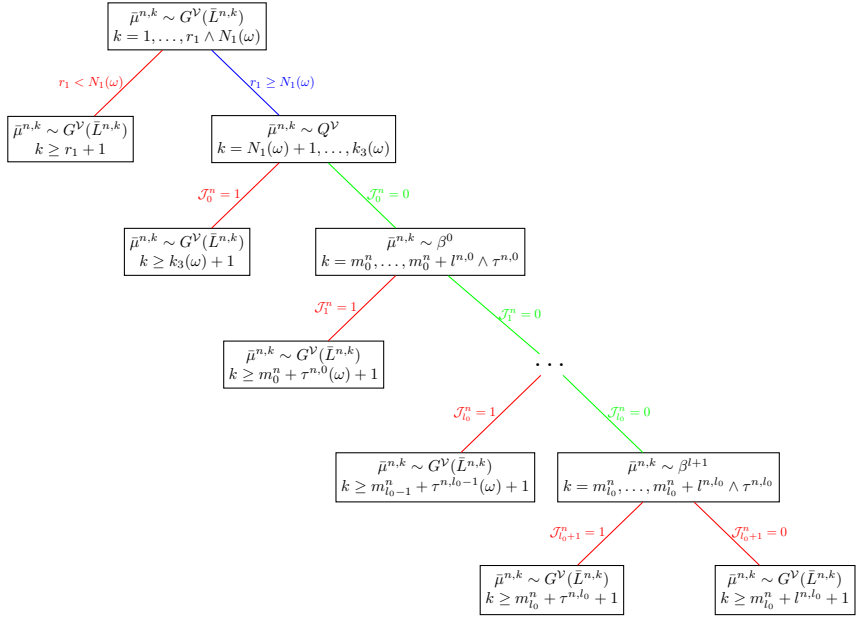
\begin{figure}
\centering
\begin{tikzpicture}
\node [draw,scale = .6,align = center] {$\Bar \mu^{n,k} \sim G^\mathcal{V}(\bar L^{n,k})$ \\ $k = 1,\dots, r_1 \wedge N_1(\om)$} [sibling distance = 2.5cm]
    child {node [black,draw,scale = .6,align = center,xshift=-0.5cm] {$\Bar \mu^{n,k} \sim G^\mathcal{V}(\bar L^{n,k})$ \\ $k \ge r_1 +1$}
    edge from parent [red] node [left,scale = .5] {$r_1 < N_1(\om)$}}
    child {node [draw,scale = .6,align = center,xshift=0.5cm] {$\Bar \mu^{n,k} \sim Q^\clv
    $ \\ $k = N_1(\om) +1,\dots,k_3(\om)$}
    child {node [black,draw,scale = .6,align = center,xshift = -0.5cm] {$\Bar \mu^{n,k} \sim G^\mathcal{V}(\bar L^{n,k})$ \\ $k \ge k_3(\om)+1$}
    edge from parent [red] node [left,scale = .5] {$\mathcal{J}_0^n = 1$}}
    child {node [black,draw,scale = .6,align = center, xshift = --0.5cm] {$\bar \mu^{n,k} \sim \beta^0$ \\ $k =
    m^n_0,\dots,m^n_0 + l^{n,0} \wedge \tau^{n,0}$}
    child {node [black,draw,scale = .6,align = center,xshift = -0.5cm] {$\Bar \mu^{n,k} \sim G^\mathcal{V}(\bar L^{n,k})$ \\ $k \ge m^n_0 + \tau^{n,0}(\om) +1$}
    edge from parent [red] node [left,scale = .5] {$\mathcal{J}_1^n = 1$}}
    child {node [xshift = 0.5cm] {\dots}
    child {node [black,draw,scale = .6,align = center,xshift=-0.5cm] {$\Bar \mu^{n,k} \sim G^\mathcal{V}(\bar L^{n,k})$ \\ $k \ge m^n_{l_0-1}+\tau^{n,l_0-1}(\om) +1$}
    edge from parent [red] node [left,scale = .5] {$\mathcal{J}_{l_0}^n = 1$}}
    child {node [black,draw,scale = .6,align = center,xshift=0.5cm] {$\bar \mu^{n,k} \sim \beta^{l+1}$ \\ $k =
    m^n_{l_0},\dots,m^n_{l_0} + l^{n,l_0} \wedge \tau^{n,l_0}$}
    child {node [black,draw,scale = .6,align = center,xshift=-0.5cm] {$\Bar \mu^{n,k} \sim G^\mathcal{V}(\bar L^{n,k})$ \\ $k \ge m^n_{l_0}+ \tau^{n,l_0} +1$}    edge from parent [red] node [left,scale = .5] {$\clj^n_{l_0+1}=1$}}
    child {node [black,draw,scale = .6,align = center,xshift=0.5cm] {$\Bar \mu^{n,k} \sim G^\mathcal{V}(\bar L^{n,k})$ \\ $k \ge m^n_{l_0}+ l^{n,l_0} +1$}    edge from parent [red] node [right,scale = .5] {$\clj^n_{l_0+1}=0$}}
    edge from parent [green] node [right,scale = .5] {$\mathcal{J}_{l_0}^n = 0$}}
    edge from parent [green] node [right,scale = .5] {$\mathcal{J}_1^n = 0$}}
    edge from parent [green] node [right,scale = .5] {$\mathcal{J}_0^n = 0$}}
    edge from parent [blue] node [right,scale = .5] {$r_1 \geq N_1(\om)$}};
 
\end{tikzpicture}
\caption{An algorithmic representation of Construction \ref{constr} of the controls. From the root of the tree to the bottom: red edges (resp. blue, green) correspond to `low probability' (resp. `high probability') events and lead to uncontrolled dynamics based on  $G$ (resp. the controls based on the kernel $Q$, the kernel $\beta_{2|1}$).} \label{fig:construction}
\end{figure}

\end{enumerate}
\end{construction}

\subsection{Convergence of Controlled Processes} \label{subsec:conv-proc}
{The main result of this section is Corollary \ref{cor:1101} which gives the convergence of $\hat M^n$ to $\hat M$.} \PZ{This says that the controlled process we have constructed closely mirrors the behavior of the near-optimal trajectory that was identified in Section \ref{sec:lowbd}.} 
Let, for  $n \in \NN$, $\{\bar L^{n, k+1}, \bar \mu^{n, k+1}, \bar \nu^{n,k}, \; k\in \NN_0\}$ be defined as in Construction \ref{constr}.
Using these, for $n \in \NN$ and $t\in [0,T]$, define $\bar L^n(t)$ as in \eqref{eq:411},
$\{\bar \Lambda^n, \bar \xi^n, \bar \zeta^n\}$ as in \eqref{eq:519} -- \eqref{eq:anrest}, \eqref{eq:519b}, and $\{\lambda^n, \beta^n, \rho^n\}$ as in \eqref{eq:525}.
Also, define $\check \bfL^n$ and $\check \Lambda^n$ as in \eqref{eq:526} and \eqref{eq:527},
and for $t \in [0,T]$, define $\sigma_{n,t}\doteq t_n-T+t$,
\begin{equation}\label{lem:mbardeflem7}
\hat M^n(t) \doteq \bar L^n(\sigma_{n,t}) = \bar L^n(t_n-T+t).
\end{equation}

For notational convenience, write $\cle^n \doteq \cle^n_{l_0+1}$, where $\cle^n_{l_0+1}$ is defined as in \eqref{eq:cale1}. Also, note that, for each $n \in \NN$,
\begin{equation}\label{eq:501n}
B^n \doteq \cup_{l=1}^{l_0+1} \{\clj_l^n=1\}\subseteq
\left\{\max_{j=0, \ldots , l_0}\,\sup_{m\ge k^*}\left\|\frac{1}{m+1} \sum_{i=0}^m \bdelta_{U^j_i}- \beta^j_{(1)}\right\| > \veps_0\right\},
\end{equation}
and thus, from \eqref{eq:236n},  for all $n \in \NN$,
\begin{equation}\label{eq:236nn}
P(B^n) \le \veps_1.
\end{equation}
The following lemma gives an estimate on the distance between $\hat M^n(0)$ and $q$. 
\begin{lemma}
\label{lem:initest}
For all $n \ge n_0$,
$$P(\|\hat M^n(0) - q\| \ge 3 \veps_0) \le 3 \veps_1.$$
\end{lemma}

\begin{proof}
Fix $n \ge n_0$ and note that
$$\hat M^n(0) = \bar L^n(t_n-T) = \bar L^n(a(t_n-T)) + \clr^n,$$
where, from \eqref{eq:tanbd} and our choice of $n_0$  above \eqref{eq:2055aa},
$$\|\clr^n\| = \|\bar L^n(t_n-T)-\bar L^n(a(t_n-T))\| \le 2(m(t_n-T)+2)^{-1}
= 2(m^n_0+2)^{-1} \le \veps_0.$$
Thus, on recalling the definition of $\{a(s) , \; s \in \RR_+\}$ from \eqref{def:stepinv}, we have
\begin{equation}
\label{eq:907}
P(\|\hat M^n(0)-q\|> 3 \veps_0) \le P(\|\bar L^n(a(t_n-T))-q\|> 2\veps_0)
= P(\|\bar L^{n, m^n_0}-q\|> 2\veps_0).
\end{equation}
Also,
\begin{align}
P(\|\bar L^{n, m^n_0}-q\|> 2\veps_0)&\le P\left((\cle^n)^c\right) + P(\cle^n,
\|\bar L^{n, m^n_0}-q\|> 2\veps_0),\label{eq:953}
\end{align}
and
\begin{equation}\label{eq:24862485}
\{\om:  \clj_0^n(\om) = 1\} \subseteq \left\{ \om : \sup\limits_{k_2(\om) \le k \le m_n^0 - 1}  \|\bar L^{n, k+1}(\om)-q\|> 2\veps_0\right\}
\end{equation}
Also, for all $k \ge k_2(\om)$, we have
\[
\bar L^{n, {k}+1}(\om) = \frac{N_1(\om)+1}{{k}+1} \bar L^{n, N_1(\om)+1}(\om)
+ \frac{{k}-N_1(\om)}{{k}+1} \frac{1}{{k}-N_1(\om)}\sum_{j= N_1(\om)+1}^{{k}} \bdelta_{\bar X^n_j(\om)}.
\]
From this, using the fact that, on $\{N_1 \le r_1\}$, from the definitions of $k_1$ and $k_2(\om)$,
\[
\frac{2(N_1(\om)+1)}{k+1} \le \frac{2(r_1+1)}{N_1+k_0 + \lfloor 4(r_1+1)/\veps_0 \rfloor +1}
\le \frac{\veps_0}{2}, \mbox{ for all } k \ge k_2(\om),
\]
we have, on $\cld^n_0$ (recall the definition of $\cld^n_0$ from \eqref{eq:2443aa}) that, 
for all $k \in \{k_2(\om), k_2(\om)+1, \ldots, k_3(\om)\}$,
\begin{equation}\label{eq:25078aa}
\begin{split}
\|\bar L^{n, k+1}(\om)-q\| &\le \frac{\veps_0}{2} + \frac{k-N_1(\om)}{k+1}  \left\| \frac{1}{k-N_1(\om)}\sum_{j= N_1(\om)+1}^{k} \bdelta_{\bar X^n_j(\om)} - q \right\|\\
&\le  \frac{\veps_0}{2} +  \max\limits_{y\in \Delta^o} \left\| \frac{1}{k-N_1(\om)}\sum_{j= N_1(\om)+1}^{k} \bdelta_{\bar Y^y_{j - N_1(\om)} (\om)} - q \right\|\\
&\le \frac{\veps_0}{2}+ \max\limits_{y\in \Delta^o}  \left\| \frac{1}{k - N_1(\om)+1}  \sum\limits_{j=0}^{k-N_1(\om)}  \bdelta_{\bar Y^y_{j} (\om)} - q \right\| + \frac{2}{k - N_1(\om) + 1}\\
&\le \veps_0 + \max\limits_{y\in \Delta^o}  \left\| \frac{1}{k - N_1(\om)+1}  \sum\limits_{j=0}^{k-N_1(\om)}  \bdelta_{\bar Y^y_{j} (\om)} - q \right\|,
\end{split}
\end{equation}
where the {third inequality follows from the triangle inequality and the observation  that 
\begin{align*}
   &\frac{1}{k-N_1(\om)} \sum\limits_{j=N_1(\om)+1}^k \bdelta_{\bar Y^y_{j - N_1(\om)} (\om)}\\
   &\quad =  \frac{1}{(k-N_1(\om))(k-N_1(\om)+1)} \sum\limits_{j=1}^{k-N_1(\om)} \bdelta_{\bar Y^y_{j} (\om)} + \frac{1}{k-N_1(\om) + 1} \left[\sum\limits_{j=0}^{k-N_1(\om)} \bdelta_{\bar Y^y_{j} (\om)}  - \bdelta_{\bar{Y}^y_0(\om)}\right],
\end{align*}
and the} last inequality follows on recalling \eqref{eq:k0defn} and noting that  $k_2(\om)-N_1(\om)> k_0$ on $\{N_1 \le r_1\}$. Once more using the fact that $k_2(\om)-N_1(\om)> k_0$ on $\{N_1 \le r_1\}$, we have, due to \eqref{eq:k0defn},  \eqref{eq:24862485}, and \eqref{eq:25078aa}, that 
\begin{align}\label{eq:954b}
\PP(\cld^n_0) &\le P\left( \left\{ \om : N_1(\om) \le r_1, \sup\limits_{k_2(\om) \le k \le m_n^0 - 1}  \|\bar L^{n, k+1}(\om)-q\|> 2\veps_0\right\}\right)\\
&\le P\left( \left\{ \om: N_1(\om) \le r_1,\,  \max_{y \in\Delta^o}\sup_{k\ge k_2(\om)}\left\| \frac{1}{k - N_1(\om)+1}  \sum\limits_{j=0}^{k-N_1(\om)}  \bdelta_{\bar Y^y_{j} (\om)} - q \right\| > \veps_0 \right\} \right)\nonumber\\
&\le P(A_{\veps_0}) \le \veps_1,
\end{align}
where the next to last inequality uses $k_2(\om) - N_1(\om) = k_1> k_0$.
This, together with \eqref{eq:r1choice2050} and \eqref{eq:236nn}, shows that
\[
P\left((\cle^n)^c\right) \le P(N_1 > r_1) + P(\cld^n_0) + P\left(\cup_{i=1}^{l_0+1} \{\clj_i^n=1\}\right) \le \veps_1 + \veps_1 + \veps_1 = 3 \veps_1.
\]
Also, on $\cle^n$, we have $\tau^n \ge m^n_0-1$, and therefore
\begin{align*}
 P(\cle^n,\, 
\|\bar L^{n, m^n_0}-q\|> 2\veps_0) &=0.
\end{align*}
 The result follows on using the above estimate together with \eqref{eq:907}, \eqref{eq:953}, and \eqref{eq:954b}.
\end{proof}

For $n \in \NN$ and $t \in [0,c]$, let
\begin{equation}\label{eq:2475aa}
p_0^n \doteq m(t_{n}-T + \veps_0) - m^n_0, \;\; 
p_0^n(t) \doteq m(t_n-T+t)  - m^n_0,
\end{equation}
and note that $p^n_0(\veps_0) = p^n_0$. Recall the definitions of $k^*$ and $n_0$ from above \eqref{eq:236n} and \eqref{eq:2055aa}, respectively, 
and let $n_1\ge n_0$ be such that, for each $n \ge n_1$ and all $i=0,\dots,l_0$,
\begin{equation}\label{eq:2479ab}
p_0^{n}  > k^*, \;\;  \frac{ 2 \max\{k^*, 2c, 1\}}{m^n_0} \le \frac{\veps_0}{2}, 
\;\; \left\| \frac{ m(t_n-T+(i+1)c)+1}{m(t_n-T+ic+\veps_0)} - e^{c-\veps_0} \right\| < \veps_0. 
\end{equation}
Note that the last estimate in the previous display is possible due to Lemma \ref{lem:mtpsieasymptotics}  and that it implies that 
\begin{equation}\label{eq:2558mepsic}
\frac{ m(t_n-T+(i+1)c)+1 }{m(t_n-T+ic+\veps_0)} \le e^{c-\veps_0} +\veps_0 \le 2 + c,  \; \; i = 0 ,\dots, l_0,
\end{equation}
since $c -\veps_0 \in (0,1)$. Also, note that the first inequality in \eqref{eq:2479ab} implies that 
\[
p_i^n = m(t_n-T+ic+\veps_0) - m_i^n \ge k^*, \; \; i = 0 ,\dots, l_0, 
\]
which is used in Lemma \ref{lem:recstep}. Further, recall from \eqref{eq:impconst} that $d_1 \doteq e^c(12 + c)$ and $d_2=6$.

The following lemma shows that the controlled state process has the correct asymptotic behavior over the time interval $[0,c]$.
\begin{lemma}
\label{lem:firint}
For $n\ge n_1$,  
\[P\left(\sup_{t \in[0,c]} \|\hat M^n(t)-\hat M(t)\|\ge d_1\veps_0\right) \le d_2 \veps_1.
\]
\end{lemma}

\begin{proof}
Fix $n \ge n_1$. Using \eqref{eq:522} and \eqref{lem:mbardeflem7}, note that for each $t \in [0,c]$,
\begin{align}\label{eq:2438a}
\hat M^n(t) &= \hat M^n(0)
 + \int_{t_n-T}^{t_n-T+t} \sum\limits_{v 
	\in \clv^d} (v-\bar L^n(a(s))) \bar \Lambda^n(v\mid s) ds \nonumber \\
 &= q  + t \beta^0_{(1)} - \int_0^t \hat M^n(s) ds
 + \bar \clr^n_1(t) + \bar \clr^n_2(t) + \bar \clr_3^n,
 \end{align}
where
\[
\bar\clr^n_1 \doteq  \hat M^n(0) - q, \;\; \bar{\clr}^n_2(t) \doteq  - \int_{t_n-T}^{t_n-T+t} \left(\bar L^n(a(s))- \bar L^n(s)\right) ds,
\]
and
\[
\bar{\clr}^n_3(t) \doteq \int_{t_n-T}^{t_n- T +t} \left( \sum\limits_{v\in\clv^d} v \bar\Lambda^n(v \mid s) - \beta^0_{(1)}\right)ds.
\] 

We begin by considering $\bar \clr^n_1$ and $\bar \clr^n_2$. From Lemma \ref{lem:initest} we see that
\begin{equation}\label{eq:2511aa}
 P( \| \bar \clr^n_1 \|\ge 3\veps_0 ) \le 3 \veps_1,
\end{equation}
while \eqref{eq:tanbd} and \eqref{eq:2479ab} ensure that
\begin{equation}\label{eq:2519aa}
\sup\limits_{ t \in [0,  c]} \| \bar \clr^n_2(t) \| \le  \frac{2c}{ m(t_n-T) + 2} \le \veps_0.
\end{equation}
We now consider $\bar \clr^n_3$. First, observe that
\begin{equation}\label{eq:2515aa}
 \sup_{ t \in [0,  \veps_0]} \| \bar \clr^n_3(t)\| \leq 2 \veps_0,
\end{equation}
which ensures that, for $t \in (\veps_0, c]$,
\begin{equation}\label{eq:2528aabb}
\begin{split}
\| \bar \clr^n_3(t)\| \le 2 \veps_0 +\| \bar \clr^n_4\| + \| \bar\clr^n_5(t) \| + \| \bar\clr^n_6(t)\|,
\end{split}
\end{equation}
where, for $t \in ( \veps_0, c]$,
\[
\bar \clr^n_4 \doteq \int_{t_n - T + \veps_0}^{t_{m(t_n - T + \veps_0) + 1}}\left( \sum\limits_{v\in\clv^d} v \bar\Lambda^n(v \mid s) - \beta^0_{(1)}\right)ds, \; \; \bar \clr^n_5(t) \doteq  \int_{t_{m(t_n - T + \veps_0) + 1}}^{t_{m(t_n-T+t)}}\left( \sum\limits_{v\in\clv^d} v \bar\Lambda^n(v \mid s) - \beta^0_{(1)}\right)ds,
\]
and
\[
\bar \clr^n_6(t) \doteq \int_{t_{m(t_n-T+t)}}^{t_n - T + t}\left( \sum\limits_{v\in\clv^d} v \bar\Lambda^n(v \mid s) - \beta^0_{(1)}\right)ds,
\]
Using \eqref{eq:2479ab} we see that
\begin{equation}\label{eq:2551aabbb}
\| \bar \clr^n_4 \| \le 
\veps_0, \; \; \sup\limits_{t \in ( \veps_0, c]}\| \bar \clr^n_6(t) \| \le  \veps_0.
\end{equation}
Observe that, for $t \in (\veps_0, c]$,
\begin{equation}\label{eq:2497bb}
\begin{split}
 \bar \clr^n_5(t) 
 = \sum\limits_{k= m^n_0 + p^n_0 +2}^{m^n_0 + p^n_0(t)} \frac{1}{k+1} ( \bar v^{n,k} - \beta^0_{(1)}).
\end{split}
\end{equation}
Define
\[
V^{n,r} \doteq \frac{1}{m^n_0 + p^n_0 + r + 2} \sum\limits_{k=m^n_0}^{m^n_0 + p^n_0 + r + 1} \left( \bar v^{n,k} - \beta^0_{(1)}\right), \;\; r \in \NN_0.
\]
It is easy to verify using an induction argument that 
\begin{equation}\label{eq:2508bb}
V^{n,r} = V^{n,0} + \sum\limits_{k=1}^r \frac{1}{m^n_0 + p^n_0 + k + 2} \left( \left( \bar v^{n,m^n_0 + p^n_0 + k+1} - \beta^0_{(1)}\right) - V^{n,k-1} \right).
\end{equation}
From \eqref{eq:2497bb} and \eqref{eq:2508bb}, we have that, for each $t \in (\veps_0,c]$,
\begin{equation}\label{eq:2569aabb}
\begin{split}
\bar \clr^n_5(t) &= \sum\limits_{k=1}^{p^n_0(t) - p^n_0 - 1} \frac{1}{k+ m^n_0 + p^n_0 + 2} \left( \bar v^{n, k+m^n_0 + p^n_0 +1} - \beta^0_{(1)}\right)\\
&= V^{n, p^n_0(t) - p^n_0 -1} - V^{n,0} + \sum\limits_{k=1}^{p^n_0(t) - p^n_0 - 1} \frac{1}{k + m^n_0 + p^n_0 + 2}  V^{n,k-1}.
\end{split}
\end{equation}
Using \eqref{eq:236n} and \eqref{eq:2479ab}, note that, for  $t \in (\veps_0,c]$ and $r \in \{0,\dots, p^n_0(t) - p^n_0 -1\}$, on $(B^n)^c \cap \mathcal{E}_0^n$,
\begin{equation}\label{eq:2575bbb}
\begin{split}
\| V^{n,r} \| &= \frac{p^n_0 + r + 2}{m^n_0 + p^n_0  + r  + 1} \left\| \frac{1}{p^n_0 + r + 2} \sum\limits_{k=m^n_0}^{m^n_0+p^n_0 + r+1 }  \bar v^{n,k} - \beta^0_{(1)} \right\|\\
&\le  \left\| \frac{1}{p^n_0 + r + 2} \sum\limits_{k=m^n_0}^{m^n_0+ p^n_0 + r +1  }  \bar v^{n,k} - \beta^0_{(1)} \right\|\\
&=  \left\| \frac{1}{p^n_0 + r + 2} \sum\limits_{k=0}^{ p^n_0 + r+1 }  \bdelta_{U^0_k} - \beta^0_{(1)} \right\| \le \veps_0.\\
\end{split}
\end{equation}
From \eqref{eq:2558mepsic}, \eqref{eq:2569aabb}, and \eqref{eq:2575bbb} we see that for each $t \in (\veps_0, c]$, on $(B^n)^c \cap \cle_0^n$,
\begin{equation}\label{eq:2584aabb}
\| \bar \clr^n_5(t)\| \le \veps_0 + \veps_0 +  \veps_0 \left( \sum\limits_{k=1}^{p^n_0(t) - p^n_0 + 1} \frac{1}{k + m^n_0 + p^n_0 + 2}\right)
\le \veps_0 \left(2 + \frac{m(t_n-T+c)+1}{m(t_n- T + \veps_0)}\right) 
\le (4 + c) \veps_0.
\end{equation}
Combining \eqref{eq:2438a}, \eqref{eq:2519aa}, \eqref{eq:2515aa}, \eqref{eq:2528aabb}, \eqref{eq:2551aabbb}, and applying Gr\"{o}nwall's lemma we see that
\[
\sup\limits_{t \in [0,c]} \| \hat M^n(t) - \hat M (t) \|  \le e^c \left( 5 \veps_0 +  \|\bar \clr^n_1 \| +  \sup\limits_{t \in (\veps_0,c]}\| \bar \clr^n_5(t)\|\right).
\]
From  \eqref{eq:r1choice2050},
\eqref{eq:236nn}, \eqref{eq:954b}, \eqref{eq:2511aa}, \eqref{eq:2584aabb}, the last estimate, and the result stated in Lemma \ref{lem:initest}, it follows that 
\begin{align*}
&P\left( \sup\limits_{t \in [0,c]} \| \hat M^n(t) - \hat M (t) \|  \ge d_1 \veps_0\right)\\
&\quad \le P(  \| \clr^n_1\| \ge 3\veps_0) + P \left( \sup\limits_{t \in (\veps_0,c]} \| \bar \clr^n_5(t) \| > (4 + c ) \veps_0 \right)\\
&\quad \le 3 \veps_1 + P(N_1 > r_1)  + P(\cld^n_0) + P(B^n) +  P \left((B^n)^c ,\mathcal{E}_0^n,  \sup\limits_{t \in (\veps_0,c]} \| \bar \clr^n_5(t) \| > (4 + c ) \veps_0 \right)\\
&\quad \le 3 \veps_1 + \veps_1 + \veps_1 +\veps_1+ 0  = 6\veps_1.
\end{align*}
The result follows.
\end{proof}

We now give a recursion estimate that will allow us to replace
$\sup_{0\le t \le c}$ in Lemma
\ref{lem:firint} with $\sup_{0 \le t \le T}$. Recall $n_1$ introduced above \eqref{eq:2479ab}.
\begin{lemma}
\label{lem:recstep}
Fix $n \ge n_1$.
Suppose that for some $1 \le i \le l_0$ and $a_1, a_2 >0$,
\begin{equation}
P\left(\sup_{0\le t \le ic} \|\hat M^n(t)- \hat M(t)\| \ge a_1e^c \veps_0\right) \le a_2 \veps_1.\label{eq:1243a}
\end{equation}
Then
\[
P\left(\sup_{0\le t \le (i+1)c\wedge T} \|\hat M^n(t)- \hat M(t)\| \ge (a_1+ b_1) e^c\veps_0\right) \le (2a_2 +3) \veps_1
\]
where $b_1 = 4 + c$ is as in \eqref{eq:impconst}.
\end{lemma}
\begin{proof}
We will only consider the case where $i< l_0$. The case $i=l_0$ is treated similarly.
Note that, for $t \in [ic, (i+1)c]$,
\begin{align}\label{eq:2495aaab}
\hat M^n(t) &= \hat M^n(ic)
 + \int_{t_n-T+ic}^{t_n-T+t} \sum\limits_{v 
	\in \clv^d} (v-\bar L^n(a(s))) \bar \Lambda^n(v\mid s) ds \nonumber \\
 &= \hat M(ic)  + (t-ic) \beta^i_{(1)} - \int_{ic}^{t} \hat M^n(s) ds
 + \bar \clr^n_1 + \bar \clr^n_2(t) + \bar \clr_3^n(t),
 \end{align}
where
\[
\bar\clr^n_1 \doteq  \hat M^n(ic) - \hat M(ic) , \;\; \bar{\clr}^n_2(t) \doteq  - \int_{t_n-T+ic}^{t_n-T+t} \left(\bar L^n(a(s))- \bar L^n(s)\right) ds,
\]
and
\[
\bar{\clr}^n_3(t) \doteq \int_{t_n-T+ic}^{t_n- T +t} \left( \sum\limits_{v\in\clv^d} v \bar\Lambda^n(v \mid s) - \beta^i_{(1)}\right)ds.
\] 

We begin by considering $\bar \clr^n_1$ and $\bar \clr^n_2$.  From  \eqref{eq:tanbd}, \eqref{eq:2479ab}, and the assumption stated in \eqref{eq:1243a}  we see that
\begin{equation}\label{eq:2511ab}
 P( \| \bar \clr^n_1 \|\ge a_1 e^c \veps_0 ) \le a_2\veps_1, \;\; \sup\limits_{ t \in [0,  c]} \| \bar \clr^n_2(t) \| \le  \frac{2c}{ m(t_n-T) + 2} \le \veps_0.
\end{equation}
We now consider $\bar \clr^n_3$. As in the proof of Lemma \ref{lem:firint}, we can write, for $t \in (\veps_0, c]$,
\begin{equation}\label{eq:2528aa}
\begin{split}
\| \bar \clr^n_3(t)\| \le 2 \veps_0 +\| \bar \clr^n_4 \| + \| \bar\clr^n_5(t) \| + \| \bar\clr^n_6(t)\|,
\end{split}
\end{equation}
where, 
\begin{equation}\label{eq:2551aa}
\| \bar \clr^n_4 \| \le 
\veps_0, \; \; \sup\limits_{t \in ( \veps_0, c]}\| \bar \clr^n_6(t) \| \le  \veps_0, \;\; \mbox{ and, on } (B^n)^c \cap \cle_0^n, \; \sup\limits_{t\in (\veps_0,c]} \| \bar \clr^n_5(t)\|
\le (4 + c) \veps_0.
\end{equation}

Combining \eqref{eq:2495aaab}, \eqref{eq:2511ab}, \eqref{eq:2528aa}, \eqref{eq:2551aa}, and applying Gr\"{o}nwall's lemma we see that  
\[
\sup\limits_{t \in [0,c]} \| \hat M^n(ic+t) - \hat M (ic+t) \|  \le e^c \left( 5 \veps_0 +  \|\bar \clr^n_1 \| +  \sup\limits_{t \in (\veps_0,c]}\| \bar \clr^n_5(t)\|\right).
\]
From \eqref{eq:2511ab}, the last estimate, and the assumption stated in \eqref{eq:1243a}, it follows that
\begin{align*}
&P\left( \sup\limits_{t \in [0,(i+1)c]} \| \hat M^n(t) - \hat M (t) \|  \ge (a_1+b_1)e^c \veps_0\right)\\
&\quad \le P\left( \sup\limits_{t \in [0,ic]} \| \hat M^n(t) - \hat M (t) \|  \ge (a_1+b_1)e^c \veps_0\right) + P\left( \sup\limits_{t \in [ic,(i+1)c]} \| \hat M^n(t) - \hat M (t) \|  \ge (a_1+b_1)e^c \veps_0\right)\\
&\quad \le  a_2 \veps_1+ P(  \| \clr^n_1\| \ge a_1 e^c\veps_0) + P \left( \sup\limits_{t \in (\veps_0,c]} \| \bar \clr^n_5(t) \| > (4 + c ) \veps_0 \right)\\
&\quad \le 2a_2 \veps_1 + P(N_1 > r_1) + P(\cld^n_0) + P(B^n) +  P \left((B^n)^c ,\mathcal{E}_{0}^n,  \sup\limits_{t \in (\veps_0,c]} \| \bar \clr^n_5(t) \| > (4 + c ) \veps_0 \right)\\
&\quad \le 2a_2 \veps_1 + \veps_1+\veps_1  + \veps_1 + 0  = (2a_2 + 3)\veps_1.
\end{align*}
The result follows.
\end{proof}
As an immediate consequence of the previous two lemmas we have the following corollary. Recall the constants $d_3, d_4$ defined  in  \eqref{eq:impconst}.
\begin{corollary}\label{cor:1101}
For all $n\ge n_1$
\[P\left(\sup_{0\le t \le T} \|\hat M^n(t)- \hat M(t)\| \ge d_3 \veps_0\right) \le d_4 \veps_1.\]
\end{corollary}

\subsection{Convergence of Costs of Controls}\label{subsec:conv-costs}
{The main result of this section is Lemma \ref{lem:costest} below which gives the desired inequality for the asymptotic cost.} \PZ{It roughly states that the relative entropy cost of the controls we have constructed is arbitrarily close to the relative entropy cost that was prescribed by the near-optimal controls we identified in Section \ref{sec:lowbd}}.

Recall the constants $A_1, B_1, C_1$ defined in \eqref{eq:impconst}, and recall $\hat \eta$, $\hat M$ introduced at the start of Section \ref{sec:constructcontrols}.

The following lemma estimates the cost of the constructed controls. 
\begin{lemma}
\label{lem:costest}
Let the collection $\{\bar\nu^{n,k}, \bar \mu^{n,k+1}, \bar L^{n,k+1}, \;  n\ge n_1, \; k \le n\}$ be given by Construction \ref{constr}. Then, 
\begin{multline*} \limsup_{n\to \infty} E\left(n^{-1} \sum\limits_{k=0}^{n-1} R\left(\bdelta_{\bar \nu^{n,k}}\otimes\bar \mu^{n,k+1} \|\bdelta_{\bar \nu^{n,k}}\otimes G^{\clv}(\bar L^{n,k+1})\right)\right)\\
\le 
e^{-T} \int_0^T e^u R\left( \hat \eta( u) \big\| \hat \eta_{(1)}( u) \otimes G(\hat M(u))  \right) du  + {C_1} \veps_1 + \veps + (l_0+1) (A_1 \veps_0 +{ B_1 \veps_1}).\end{multline*}
\end{lemma}
\begin{proof}
For notational simplicity, denote 
\[
R^{n,k} \doteq R\left(\bdelta_{\bar \nu^{n,k}}\otimes\bar \mu^{n,k+1} \|\bdelta_{\bar \nu^{n,k}}\otimes G^{\clv}(\bar L^{n,k+1})\right), \;\;  n \ge n_1, \; k \le n,
\]
and fix $n \ge n_1$, where $n_1$ is as introduced above \eqref{eq:2479ab}. We begin with the following observations.
\begin{itemize}
\item By construction,
\begin{equation}
\frac{1}{n} \sum\limits_{k=0}^{N_1-1} R^{n,k} = 0,\;\; 
\bm{1}_{\{r_1<N_1\}}\frac{1}{n} \sum\limits_{k=0}^{n-1} R^{n,k} = 0, \; \; n \in \NN.
\end{equation}
\item On $\{N_1 \le r_1\}$, 
$
\min\limits_{x \in \Delta^o} \bar L^{n, N_1}(x) > a^*
$,
which says that, with
$
a_1^* \doteq \frac{a^*}{k_1 + 2}$,
we have
\[
\min\limits_{x\in\Delta^o} \inf\limits_{N_1 \le k \le k_2} \bar L^{n, k+1}(x) \ge a^*_1.
\]
This in turn, from Assumption \ref{ass1}(2b) implies that 
\begin{equation}\label{eq:2896aa}
\min\limits_{(\bar v^{n,k},y) \in A_+} \inf\limits_{N_1 \le k \le k_2}G^{\clv}(\bar L^{n,k+1})(\bar \nu^{n,k}, \bdelta_y) \ge a_1^* \delta_0^A,
\end{equation}
 Together, \eqref{eq:1945aaa} and \eqref{eq:2896aa} imply that 
\begin{equation}\label{eq:2905aa}
\frac{1}{n} \sum\limits_{k=N_1}^{k_2-1} R^{n,k} 
= \frac{1}{n} \sum\limits_{k=N_1}^{k_2-1} R\left(\bdelta_{\bar \nu^{n,k}}\otimes Q^{\clv} \|\bdelta_{\bar \nu^{n,k}}\otimes G^{\clv}(\bar L^{n,k+1}) \right)
\le \left|\log\left(a_1^*\delta_0^A\right)\right|\frac{k_1}{n}.
\end{equation}
\item Recall from \eqref{eq:2443aa} that
$\cld^n_0 = \{N_1 \le r_1 \} \cap \{ \clj^n_0 = 1\}$,
and note that, as in \eqref{eq:2905aa}, if $\cld^n_0$ occurs, then from Construction \ref{constr} (iii)
\begin{equation}\label{eq:2912aa}
\frac{1}{n} \sum\limits_{k=0}^{n-1} R^{n,k} 
= \frac{1}{n} \sum\limits_{k=N_1}^{k_3-1} R^{n,k}
= \frac{1}{n} \sum\limits_{k=N_1}^{k_2-1} R^{n,k}
+ \frac{1}{n} \sum\limits_{k=k_2}^{k_3-1} R^{n,k}  
\le \left|\log\left(a_1^*\delta_0^A\right)\right|\frac{k_1}{n} +
\frac{1}{n} \sum\limits_{k=k_2}^{k_3-1} R^{n,k}.
\end{equation}
Additionally, from \eqref{eq:219}, 
$
\min\limits_{x\in \Delta^o} q_x > \delta/4$,
and, from \eqref{eq:conscho}, $\veps_0 < \delta/16$, so, on 
recalling the definition of $k_3$, we see that  
\[
\min\limits_{x\in \Delta^o}\inf\limits_{k_2 \le k \le k_3 - 1} \bar L^{n,k+1}(x) \ge \delta /8.
\]
Thus, for each $k \in \{ k_2, \dots, k_3 - 1\}$, with $\delta_1$ defined in \eqref{eq:impconst}, 
\begin{equation}\label{eq:2933aa}
R^{n,k} = R\left(\bdelta_{\bar \nu^{n,k}}\otimes Q^{\clv}  \|\bdelta_{\bar \nu^{n,k}}\otimes G^{\clv}(\bar L^{n,k+1})\right) \le |\log\delta_1 |.
\end{equation}
From \eqref{eq:2912aa} and \eqref{eq:2933aa} we obtain
\begin{align}\label{eq:2951aa}
\bm{1}_{\cld^n_0} \frac{1}{n} \sum\limits_{k=0}^{n-1} R^{n,k} \le \left|\log\left(a_1^*\delta_0^A\right)\right|\frac{k_1}{n} + |\log\delta_1|\frac{m^n_0}{n}.
\end{align}
Using Lemma \ref{lem:mtpsieasymptotics} along with \eqref{eq:954b} and \eqref{eq:2951aa} we see that  
\begin{equation}\label{eq:1201}
\limsup_{n\to \infty}  E \left(\bm{1}_{\cld_0^n} \frac{1}{n} \sum\limits_{k=0}^{n-1}
R\left(\bdelta_{\bar \nu^{n,k}}\otimes\bar \mu^{n,k+1} \|\bdelta_{\bar \nu^{n,k}}\otimes G^{\clv}(\bar L^{n,k+1})
\right)\right) \le |\log \delta_1|\veps_1.
\end{equation}
\item On $\cle^n_0$, 
a calculation similar to the one above shows that
\begin{align}\label{eq:2965aa}
\frac{1}{n} \sum\limits_{k=0}^{m^n_0-1} R^{n,k}
&= \frac{1}{n} \sum\limits_{k=0}^{N_1-1} R^{n,k}
+ \frac{1}{n} \sum\limits_{k=N_1}^{k_2-1} R^{n,k}
+ \frac{1}{n} \sum\limits_{k=k_2}^{m^n_0-1} R^{n,k}\nonumber\\
&\le \left|\log\left(a_1^*\delta_0^A\right)\right|\frac{k_1}{n} + |\log\delta_1|\frac{m^n_0}{n}.
\end{align}
\item Next, on $\cld^n_1 \doteq \cle^n_0 \cap \{ \clj^n_1 = 1\}$, 
since
$\min\limits_{x\in \Delta^o} \bar L^{n,m^n_0  -1}(x) > \delta/8$, we have
\[
\inf\limits_{x\in \Delta^o} \inf\limits_{0 \le k \le k^*}\bar L^{n,m^n_0  + k}(x)  \ge \frac{ \delta}{8(k^* + 2)}.
\]
Thus, using \eqref{eq:2965aa} we see that, since $\tau^{n,0} \ge k^*$  and $\bm{1}_{\cld^n_1}R^{n,k} = 0$ for all $k \in \{m^n_0 + \tau^{n,0}, \dots, n-1\}$, we have, on $\cld^n_1$, that 
\begin{align}
\frac{1}{n} \sum\limits_{k=0}^{n-1} R^{n,k}
&= \frac{1}{n} \sum\limits_{k=0}^{m^n_0-1} R^{n,k}
+ \frac{1}{n} \sum\limits_{k=m^n_0}^{m^n_0+k^*-1} R^{n,k}
+ \frac{1}{n} \sum\limits_{k=m^n_0+k^*}^{m^n_0+\tau^{n,0}-1} R^{n,k} 
 + \frac{1}{n} \sum\limits_{k=m^{n}_0 + \tau^{n,0}}^{n-1} R^{n,k} \nonumber\\
&\le \left|\log\left(a_1^*\delta_0^A\right)\right|\frac{k_1}{n}+ |\log\delta_1|\frac{m^n_0}{n} + 
\left|\log\left( \frac{\delta_1}{k^*+2} \right)\right|\frac{k^*}{n}
+ \frac{1}{n} \sum\limits_{k=m^n_0+k^*}^{m^n_0+\tau^{n,0}-1} R^{n,k}.
\end{align}
 Now, recall from  \eqref{eq:2013aa} that
$\inf\limits_{x\in\Delta^o}\beta^0_{(1)}(x) \ge \delta$, and from \eqref{eq:219} that $\inf\limits_{x\in\Delta^o} q(x) \ge \delta/4$,
which, from the definition of $\tau^{n,0}$ in \eqref{eq:taun0def} and the fact that $\veps_0 \le \delta/16$, says that
\begin{equation}\label{eq:848nn}
\inf\limits_{x\in\Delta^o}\inf\limits_{m^n_0 + k^* \le k \le m_0^n + \tau^{n,0}-1} \bar L^{n, k+1}(x) \ge \delta /8.
\end{equation}
It then follows that
\[
\frac{1}{n} \sum\limits_{k=m^n_0+k^*}^{m^n_0+\tau^{n,0}-1} R^{n,k} \le
|\log\delta_1|, \quad R^{n,m^n_0+\tau^{n,0}} \le \left| \log \left(\frac{\delta_1}{2}\right)\right|
\]
and consequently, on $ \cld^n_1$,
\begin{align*}
  \frac{1}{n} \sum\limits_{k=0}^{n-1} R^{n,k} 
  \le \left|\log\left(a_1^*\delta_0^A\right)\right|\frac{k_1}{n}+
|\log\delta_1|\frac{m^n_0}{n} + 
\left|\log\left(\frac{\delta_1}{k^* + 2}\right)\right|\frac{k^*}{n} + |\log\delta_1| + \frac{\left| \log\left( \frac{\delta_1}{2}\right)\right|}{n}.
\end{align*}
Thus, on using 
\eqref{eq:236n}, we see that 
\begin{equation}\label{eq:910nn}
\limsup_{n\to \infty}  E \left(\bm{1}_{\cld_1^n} \frac{1}{n} \sum\limits_{k=0}^{n-1}
R\left(\bdelta_{\bar \nu^{n,k}}\otimes\bar \mu^{n,k+1} \|\bdelta_{\bar \nu^{n,k}}\otimes G^{\clv}(\bar L^{n,k+1})
\right)\right)
\le 2|\log \delta_1| \veps_1.
\end{equation}

\item By a similar calculation, on $\cle^n_1$,
\begin{multline}
\frac{1}{n} \sum\limits_{k=0}^{m^n_0+ l^{n,0}-1} R^{n,k}
= \frac{1}{n} \sum\limits_{k=0}^{m^n_0-1} R^{n,k}
+ \frac{1}{n} \sum\limits_{k=m^n_0}^{m^n_0+k^*-1} R^{n,k}
+ \frac{1}{n} \sum\limits_{k=m^n_0+k^*}^{m^n_0+ l^{n,0}-1} R^{n,k}\\
\le \left|\log\left(a_1^*\delta_0^A\right)\right|\frac{k_1}{n}+ |\log\delta_1|\frac{m^n_0}{n} + 
\left|\log\left(\frac{\delta_1}{k^* + 2 }\right)\right|\frac{k^*}{n}
+ \frac{1}{n} \sum\limits_{k=m^n_0+k^*}^{m^n_0+ l^{n,0}-1} R^{n,k}.\label{eq:1101}
\end{multline}
Next, recalling the relationship between $\{\bar{X}^{n}_{k}, \; k \le n\}$ and $\{ \bar v^{n,k}, \; k \le n\}$, note that 
\begin{align}
\frac{1}{n} \sum\limits_{k=m^n_0+k^*}^{m^n_0+ l^{n,0}-1} R^{n,k}
&= \frac{1}{n} \sum\limits_{k=m^n_0+k^*}^{m^n_0+ l^{n,0}-1} R\left(\bdelta_{\bar{X}^{n}_{k}}\otimes \beta^0_{2|1}(\cdot \mid \bar{X}^{n}_{k}) \|\bdelta_{\bar{X}^{n}_{k}}\otimes G(\bar L^{n,k+1}) \right)\nonumber\\
&= \frac{1}{n} \sum\limits_{k=m^n_0+k^*}^{m^n_0+ l^{n,0}-1} \sum_{x \in \Delta^o} 
\bm{1}_{\{\bar{X}^{n}_{k}=x\}} h_x(\bar L^{n,k+1})\label{eq:102n}
\end{align}
where
\begin{equation}
h_x(m) \doteq R\left( \beta^0_{2|1}(\cdot \mid x) \| G(m)(x, \cdot)\right), \; \; x \in \Delta^o, m \in \clp(\Delta^o).
\end{equation}

Note that for $m,m' \in \clm_{\delta/8} \doteq \{\pi \in \clp(\Delta^o): \inf_{x \in \Delta^o}\pi_x > \delta/8\}$,
\begin{equation}\label{eq:852nn}
	|h_x(m) - h_x(m')| \le \delta_1^{-1} \|m-m'\|.\end{equation}
Additionally, on $\cle^n_1$, for each $k \in \{m^n_0 + k^*, \dots, m^n_0 + l^{n,0} - 1\}$, we have that $\bar L^{n,k+1} \in \clm_{\delta/8}$, which, together with \eqref{eq:852nn}, ensures that 
\begin{align*}h_x(\bar L^{n,k+1}) &= h_x(\bar L^{n}(t_k)) = h_x(\bar L^{n}(t_n-T + (t_k-t_n+T)))\\
&= h_x(\hat M^n(t_k - (t_n-T))) \le  h_x(\hat M(t_k - (t_n-T))) + \delta_1^{-1} \sup_{t \in [0,T]} \|\hat M^n(t)-\hat M(t)\|,
\end{align*}
for each such $k$. Thus,
\begin{equation}
\frac{1}{n} \sum\limits_{k=m^n_0+k^*}^{m^n_0+ l^{n,0} - 1} R^{n,k} 
\le \frac{1}{n} \sum\limits_{k=m^n_0+k^*}^{m^n_0+ l^{n,0} - 1} \sum_{x \in \Delta^o} 
\bm{1}_{\{\bar{X}^{n}_{k}=x\}} h_x(\hat M(t_k - (t_n-T))) +\delta_1^{-1} \sup_{t \in [0,T]} \|\hat M^n(t)-\hat M(t)\|.\label{eq:1102}
\end{equation}
Next, for $k \in \{m^n_0 + k^* ,\dots, m^n_0 + l^{n,0} - 1\}$, let
\[
H_{k,x} \doteq \frac{1}{k^*+1} \sum_{j=0}^{k^*} h_x(\hat M(t_{j+k} - (t_n-T))).
\]
Then, using \eqref{eq:1233} and \eqref{eq:2479ab}, we see that, for each $k \in \{m^n_0 + k^* ,\dots, m^n_0 + l^{n,0} - 1\}$,
\
\begin{equation}\label{eq:2930aaa}
\begin{split}
H^*_{k,n}& \doteq \sup\limits_{x\in \Delta^o} |H_{k,x}-h_x(\hat M(t_k - (t_n-T)))|\\
& \le \delta_1^{-1}
\max_{0\le j \le k^*} \|\hat M(t_{j+k} - (t_n-T)) - \hat M(t_k - (t_n-T))\|\\
&\le 2 \delta_1^{-1} | t_{m^n_0 + 2k^*} - t_{m^n_0 + k^*}|
\le  \delta_1^{-1} \veps_0.
\end{split}
\end{equation}
Using this estimate we see that 
\begin{equation}\label{eq:3083aa}
\begin{split}
&\frac{1}{n} \sum\limits_{k=m^n_0+k^* }^{m^n_0+ l^{n,0} - 1} \sum_{x \in \Delta^o} 
\bm{1}_{\{\bar{X}^{n}_{k}=x\}} h_x(\hat M(t_k - (t_n-T)))  \\
&\le \frac{1}{n} \sum\limits_{k=m^n_0+k^*}^{m^n_0+ l^{n,0} - 1} \sum_{x \in \Delta^o} 
\bm{1}_{\{\bar{X}^{n}_{k}=x\}} H_{k,x}+ \frac{1}{n}\sum\limits_{k=m^n_0+k^*}^{m^n_0+l^{n,0}-1} H^*_{k,n}\\
&\le \frac{1}{n} \sum\limits_{k=m^n_0+k^*}^{m^n_0+ l^{n,0} - 1} \sum_{x \in \Delta^o} 
\bm{1}_{\{\bar{X}^{n}_{k}=x\}} H_{k,x}+ \delta_1^{-1} \veps_0.
\end{split}
\end{equation}
Furthermore, on recalling, from \eqref{eq:219}, that
$
\sup\limits_{x \in \Delta^o}\sup\limits_{t \in [0,T]} h_x( \hat M(t)) \le | \log \delta_1|
$,
we see that
\begin{equation}\label{eq:3111aa}
\begin{aligned}	
&\frac{1}{n} \sum\limits_{k=m^n_0+k^*}^{m^n_0+ l^{n,0} - 1} \sum_{x \in \Delta^o} 
\bm{1}_{\{\bar{X}^{n}_{k}=x\}} H_{k,x}\\
&\le \frac{1}{n} \sum\limits_{r=m^n_0+2k^* }^{m^n_0+ l^{n,0} - 1} \sum_{x \in \Delta^o} 
\left( \frac{1}{k^*+1} \sum_{j=0}^{k^*} \bm{1}_{\{\bar X^{n}_{r-j}=x\}}\right) h_x(\hat M(t_{r} - (t_n-T))) + \frac{2k^*}{n} |\log \delta_1|,
\end{aligned}
\end{equation}
and, from  \eqref{eq:236nb},  we have that, for each $r \in \{m^n_0 + 2k^*, \dots, m^n_0 + l^{n,0} - 1\}$,
\begin{equation}\label{eq:3119aa}
\sum_{x \in \Delta^o} E \left|  \frac{1}{k^*+1} \sum_{j=0}^{k^*} \bm{1}_{\{\bar X^{n}_{r-j}=x\}} - \beta^0_{(1)}(x)\right| \le \veps_0.
\end{equation}
From \eqref{eq:3083aa}, \eqref{eq:3111aa}, and \eqref{eq:3119aa}, we see that
\begin{align}
&\frac{1}{n} \sum\limits_{k=m^n_0+k^*}^{m^n_0+ l^{n,0} - 1} \sum_{x \in \Delta^o} 
\bm{1}_{\{\bar{X}^{n}_{k}=x\}} h_x(\hat M(t_k - (t_n-T)))\nonumber\\
&\quad \le \frac{1}{n} \sum\limits_{r=m^n_0+2k^*}^{m^n_0+ l^{n,0} -1} \sum_{x \in \Delta^o}
 \beta^0_{(1)}(x)h_x(\hat M(t_{r} - (t_n-T))) + \clr^{n} \label{eq:859nn}
\end{align}
where
\begin{equation}
E|\clr^n| \le  \left(  \veps_0  + \frac{2 k^* }{n} \right) |\log \delta_1| +  \veps_0 \delta_1^{-1} . 
\label{eq:1103}\end{equation}
Next, letting $u^n_1 \doteq t_{m^n_0+ 2k^*}$, we have that
\begin{align}
&\frac{1}{n} \sum\limits_{r=m^n_0+2k^*}^{m^n_0+ l^{n,0}-1} \sum_{x \in \Delta^o} 
\beta^0_{(1)}(x)h_x(\hat M(t_{r} - (t_n-T))) \nonumber\\
&\quad \le \frac{1}{n} \int_{u^n_1}^{t_{m^n_1}} \psi_e(s) \sum_{x \in \Delta^o} h_x(\hat M(a_n(s)-(t_n-T)))
\beta^0_{(1)}(x) ds \nonumber\\
&\quad \le \frac{1}{n} \int_{u^n_1}^{t_{m^n_1}} \psi_e(s) \sum_{x \in \Delta^o} h_x(\hat M(s-(t_n-T)))
\beta^0_{(1)}(x) ds +  \veps_0 \delta_1^{-1} \nonumber \\
&\quad \le \frac{1}{n}\int_{t_n-T}^{t_n-T+c}\psi_e(s) \sum_{x \in \Delta^o} h_x(\hat M(s-(t_n-T)))
\beta^0_{(1)}(x) ds +  \veps_0 \delta_1^{-1}\nonumber\\
&\quad =\frac{1}{n}\int_{0}^{c}\psi_e(t_n-(T-s)) \sum_{x \in \Delta^o} h_x(\hat M(s))
\beta^0_{(1)}(x) ds +  \veps_0 \delta_1^{-1},\label{eq:1104}
\end{align}
where the second inequality uses \eqref{eq:1233}, \eqref{eq:2055aa}, and \eqref{eq:852nn}.
Also, from Lemma \ref{lem:mtpsieasymptotics} and recalling the definition of
$h_x$ and $\beta^0$, 
\begin{align}
&\lim\limits_{n\to\infty}\frac{1}{n}\int_{0}^{c}\psi_e(t_n-(T-s)) \sum_{x \in \Delta^o} h_x(\hat M(s))
\beta^0_{(1)}(x) ds \nonumber \\
&\quad = e^{-T}\int_{0}^{c}e^s \sum_{x \in \Delta^o} \beta^0_{(1)}(x) 
R\left( \beta^0_{2|1}(\cdot \mid x) \| G(\hat M(s))(x, \cdot)\right) ds\nonumber\\
&\quad = e^{-T}\int_{0}^{c}e^s R\left( \hat\eta(  s) \| \hat \eta_{(1)}( s) \otimes G(\hat M(s))\right) ds.
\label{eq:1105}
\end{align}
Finally, using the fact that  on $\cle^n_1$, from \eqref{eq:848nn}, $\bar L^{n, m^n_0 + l^{n,0}} \in \clm_{\delta/8}$, we see that 
\begin{equation}\label{eq:3183aa}
\begin{split}
R^{n, m^n_0 + l^{n,0} } \le   | \log (\delta_1/2)|.
\end{split}
\end{equation}
\item 
Combining the estimates in \eqref{eq:1101}, \eqref{eq:1102}, \eqref{eq:859nn},
\eqref{eq:1103}, \eqref{eq:1104}, and \eqref{eq:3183aa}, we see that 
\begin{align}
&\limsup_{n\to \infty} E \left(\bm{1}_{\cle^n_1} \frac{1}{n} \sum\limits_{k=0}^{m^n_0+ l^{n,0}}
R\left(\bdelta_{\bar \nu^{n,k}}\otimes\bar \mu^{n,k+1} \|\bdelta_{\bar \nu^{n,k}}\otimes G^{\clv}(\bar L^{n,k+1})\right)\right)
\nonumber\\
&\quad \le \limsup_{n\to \infty}\Bigg(
\left|\log\left(a_1^*\delta_0^A\right)\right|\frac{k_1}{n}+ |\log\delta_1|\frac{m^n_0}{n} +
\left|\log\left(\frac{\delta_1}{k^*+2}\right)\right|\left(\frac{k^*+1}{n}\right) + \left| \log\left(\frac{\delta_1}{2}\right)\right|\frac{1}{n} \nonumber \\
&\qquad + \delta_1^{-1} E\sup_{t \in [0,T]} \|\hat M^n(t)-\hat M(t)\| +  \left(\veps_0 + \frac{2k^*}{n} \right) |\log \delta_1| +   \veps_0 \delta_1^{-1}
+ \veps_0\delta_1^{-1} \nonumber\\
&\qquad + \frac{1}{n}\int_{0}^{c}\psi_e(t_n-(T-s)) \sum_{x \in \Delta^o} h_x(\hat M(s))
\beta^0_{(1)}(x) ds\Bigg)\nonumber\\
&\quad \le |\log\delta_1|e^{-T}+ \left(   |\log \delta_1|+ (2 + d_3) \delta_1^{-1} \right)\veps_0 + 2 d_4 \delta_1^{-1} \veps_1  \nonumber\\
&\qquad +  \limsup_{n\to \infty} \frac{1}{n}\int_{0}^{c}\psi_e(t_n-(T-s)) \sum_{x \in \Delta^o} h_x(\hat M(s))
\beta^0_{(1)}(x) ds\nonumber\\
&\quad \le \veps + A_1\veps_0+ B_1 \veps_1 + e^{-T}\int_{0}^{c} \exp(s) R\left( \hat\eta(  s) \| \hat \eta_{(1)}( s) \otimes G(\hat M(s))\right) ds,\label{eq:101n}
\end{align}
where the second inequality follows from Corollary \ref{cor:1101}, and  the 
third inequality follows from Lemma \ref{lem:mtpsieasymptotics}, our choice of $T$ in \eqref{eq:sizeofT}, and using \eqref{eq:impconst}, and \eqref{eq:1105}.

\item Letting, for $l \in \{1, \ldots, l_0+1\}$,
\[
\cld^n_l \doteq   \cle^n_{l-1} \cap \{ \clj_l^n = 1\},
\]
we see exactly as in the proof of \eqref{eq:1201} and \eqref{eq:910nn} , that 
\begin{align}\label{eq:3227aa}
&\limsup_{n\to \infty}  E \left(\bm{1}_{\cld_l^n} \frac{1}{n} \sum\limits_{k=0}^{n-1}
R\left(\bdelta_{\bar \nu^{n,k}}\otimes\bar \mu^{n,k+1} \|\bdelta_{\bar \nu^{n,k}}\otimes G^{\clv}(\bar L^{n,k+1})
\right)\right)\nonumber\\
 &\quad \le |\log \delta_1|(l+2) \veps_1
\le |\log \delta_1|(l_0+3) \veps_1.
\end{align}

\item 
Next, we show that, for each  $l \in \{0,\dots, l_0\}$, 
\begin{align}
&\limsup_{n\to \infty} E \left(\bm{1}_{\cle^n_{l+1}} \frac{1}{n} \sum\limits_{k=0}^{m^n_l+ l^{n,l}}
R\left(\bdelta_{\bar \nu^{n,k}}\otimes\bar \mu^{n,k+1} \|\bdelta_{\bar \nu^{n,k}}\otimes G^{\clv}(\bar L^{n,k+1})\right)\right)\nonumber\\
&\le \veps + (l+1) (A_1 \veps_0  + {B_1 \veps_1}) + e^{-T} \int_0^{(l+1)c \wedge T} e^s R\left( \hat\eta (s) \| \hat \eta_{(1)}( s) \otimes G(\hat M(s)) \right) ds.\nonumber\\
 \label{eq:105n}
\end{align}
Note that by  \eqref{eq:101n}, the statement in \eqref{eq:105n} holds for $l = 0$. Now, suppose, for some $r < l_0$, that the statement in \eqref{eq:105n} holds for all $l \in \{0, 1, \ldots r\}$.
We argue that it also holds for $l = r+1$.  We only give the argument for $r< l_0-1$, as the case when $r = l_0-1$ is treated similarly.
Note that, by our inductive hypothesis, 
\begin{multline}\label{eq:3245aa}
\limsup_{n\to \infty} E \left(\bm{1}_{\cle_{r+1}^n} \frac{1}{n} \sum\limits_{k=0}^{m^n_{r+1}+ l^{n,r+1}} R^{n,k}\right)\\
\le \limsup_{n\to \infty} E \left(\bm{1}_{\cle_{r}^n} \frac{1}{n} \sum\limits_{k=0}^{m^n_{r}+ l^{n,r}} R^{n,k} 
 + \bm{1}_{\cle_{r+1}^n} \frac{1}{n} \sum\limits_{k=m^n_{r+1}}^{m^n_{r+1}+ l^{n,r+1}} R^{n,k}\right)  \\
\le \veps + (r+1) (A_1 \veps_0  +  B_1 \veps_1) + e^{-T} \int_0^{(r+1)c } e^s R\left( \hat\eta(  s) \| \hat \eta_{(1)}( s) \otimes G(\hat M(s))\right) ds  \\
 + \limsup_{n\to \infty} E \left(\bm{1}_{\cle_{r+1}^n} \frac{1}{n} \sum\limits_{k=m^n_{r+1}}^{m^n_{r+1}+ l^{n,r+1}} R^{n,k}\right).
\end{multline}
Now, an argument along the lines of the one used for \eqref{eq:102n} -- \eqref{eq:101n} shows that, with 
\[
h^{r+1}_x(m) \doteq  R\left( \beta^{r+1}_{2|1}(\cdot \mid x) \| G(m)(x, \cdot)\right), \;\; x \in \Delta^o, m \in \clp(\Delta^o),
\]
we have 
\begin{align}
&\limsup_{n\to \infty} E \left(\bm{1}_{\cle_{r+1}^n} \frac{1}{n} \sum\limits_{k=m^n_{r+1}}^{m^n_{r+1}+ l^{n,r+1}} R^{n,k}\right)   \nonumber \\
&\quad \le A_1 \veps_0 + B_1\veps_1 +  \limsup_{n\to \infty} \frac{1}{n}\int_{t_n-T+(r+1)c}^{t_n-T+(r+2)c}\psi_e(s) \sum_{x \in \Delta^o} h^{r+1}_x(\hat M(s-(t_n-T)))\beta^{r+1}_{(1)}(x) ds \nonumber \\
&\quad \le A_1 \veps_0 + B_1\veps_1 + \limsup_{n\to \infty} \frac{1}{n} \int_{(r+1)c}^{(r+2)c}\psi_e(t_n-(T-s)) \sum_{x \in \Delta^o} h^{r+1}_x(\hat M(s))
\beta^{r+1}_{(1)}(x) ds \nonumber \\
&\quad = A_1 \veps_0 + B_1\veps_1 + e^{-T} \int_{(r+1)c}^{(r+2)c} e^s \sum_{x \in \Delta^o} h^{r+1}_x(\hat M(s)) \beta^{r+1}_{(1)}(x) ds  \nonumber\\
&\quad = A_1 \veps_0 + B_1 \veps_1 +  e^{-T} \int_{(r+1)c}^{(r+2)c} e^sR\left( \hat\eta(  s) \| \hat \eta_{(1)}( s) \otimes  G(\hat M(s))\right) ds. \label{eq:3257aa}
\end{align}
Combining the estimates in \eqref{eq:3245aa} and\eqref{eq:3257aa} we have the inequality in \eqref{eq:105n} for $l =r+1$, which proves the statement in \eqref{eq:105n} with $l=r+1$.
\item
Finally, combining  \eqref{eq:1201}, \eqref{eq:3227aa} and \eqref{eq:105n}  we see that 
\begin{align*}
&\limsup_{n\to \infty} E\left(n^{-1} \sum\limits_{k=0}^{n-1} R\left(\bdelta_{\bar \nu^{n,k}}\otimes\bar \mu^{n,k+1} \|\bdelta_{\bar \nu^{n,k}}\otimes G^{\clv}(\bar L^{n,k+1}) \right)\right)\\
&\quad \le \sum_{j=0}^{l_0+1}\limsup_{n\to \infty}  E \left(\bm{1}_{\cld_j^n} \frac{1}{n} \sum\limits_{k=0}^{n-1}
R\left(\bdelta_{\bar \nu^{n,k}}\otimes\bar \mu^{n,k+1} \|\bdelta_{\bar \nu^{n,k}}\otimes G^{\clv}(\bar L^{n,k+1})
\right)\right)\\
& \qquad + \limsup_{n\to \infty} E \left(\bm{1}_{\cle_{l_0+1}^n} \frac{1}{n} \sum\limits_{k=0}^{m^n_{l_0}+ l^{n,l_0}}
R\left(\bdelta_{\bar \nu^{n,k}}\otimes\bar \mu^{n,k+1} \|\bdelta_{\bar \nu^{n,k}}\otimes G^{\clv}(\bar L^{n,k+1})\right)\right)\\
&\quad \le C_1 \veps_1 + \veps + (l_0+1) (A_1 \veps_0 + B_1 \veps_1)\\
& \qquad + e^{-T} \int_0^{ T} \exp(s) R\left( \hat\eta( s) \| \hat \eta_{(1)}( s) \otimes G(\hat M(s))\right) ds.
\end{align*}
\end{itemize}
The result follows.
\end{proof}

\subsection{Proof of Laplace Lower Bound}\label{sec:pflowbd}
We now complete the proof of the Laplace lower bound in Theorem \ref{thm:lowbd}. Recall the Lipschitz function $F: \clp(\Delta^o) \to \RR$ and  $\veps \in (0,1)$ fixed in Section \ref{sec:prelimest}. Also recall the constant $T \in (0, \infty)$ from \eqref{eq:sizeofT} with $M^1$ chosen as in Section \ref{subsec:non-degen}. 
Let $\hat M^3 \doteq \hat M$ and $\hat \eta^3 \doteq \hat \eta $ be as constructed in 
\eqref{eq:601fin} -- \eqref{eq:219}. Also, recall the constants  $c\doteq \kappa_3$ and $l_0 \doteq \lfloor Tc^{-1} \rfloor$ associated with $\hat M$ and $\hat \eta$ defined in Section \ref{sec:constructcontrols}.
Fix $\veps_0, \veps_1$ as in \eqref{eq:conscho}.
Let the collection $\{\bar\nu^{n,k}, \bar \mu^{n,k+1}, \bar L^{n,k+1}, \; n\ge n_0, \; k \le n\}$ be given by Construction \ref{constr}.
Then, using \eqref{eq:varrep}, 
\begin{multline}\label{eq:varrepfin}
	-n^{-1}\log E \exp[-n F(L^{n+1})]
	\\
        \le E \left[F(\bar L^n(t_n) ) + n^{-1} \sum\limits_{k=0}^{n-1} R\left(\bdelta_{\bar \nu^{n,k}}\otimes\bar \mu^{n,k+1} \|\bdelta_{\bar \nu^{n,k}}\otimes G^{\clv}(\bar L^{n,k+1})\right)\right].
\end{multline}
From Corollary \ref{cor:1101},
\begin{align*}
E(F(\bar L^n(t_n) ) )= E(F(\hat M^n(T)))
\le E(F(\hat M(T)) + F_{\mbox{\tiny{lip}}} (d_3 \veps_0 + 2d_4 \veps_1)
\le F(m^3) + \veps,
\end{align*}
where the second  inequality follows on recalling that $\hat M(T) = \hat M^3(T) = m^3$
and on using \eqref{eq:conscho}.
Also, from Lemma \ref{lem:costest}, 
\begin{multline*}
\limsup_{n\to \infty}
E\left[ n^{-1} \sum\limits_{k=0}^{n-1} R\left(\bdelta_{\bar \nu^{n,k}}\otimes\bar \mu^{n,k+1} \|\bdelta_{\bar \nu^{n,k}}\otimes G^{\clv}(\bar L^{n,k+1})\right)\right]\\
\le e^{-T} \int_0^T \exp(u) R\left( \hat \eta( u) \big\| \hat \eta_{(1)}( u)\otimes G(\hat M(u)) \right) du + C_1\veps_1 + \veps + (l_0+1)( A_1 \veps_0 + B_1 \veps_1)\\
\le e^{-T} \int_0^T \exp(u) R\left( \hat \eta( u) \big\| \hat \eta_{(1)}( u)\otimes G(\hat M(u)) \right) du + 2\veps,
\end{multline*}
where the last line follows from \eqref{eq:conscho}.
Combining the last two estimates 
\begin{multline*}
\limsup_{n\to \infty}
-n^{-1}\log E \exp[-n F(L^{n+1})]\\
\le F(m^3) + e^{-T} \int_0^T e^u R\left( \hat \eta( u) \big\| \hat \eta_{(1)}( u)\otimes G(\hat M(u) )\right) du + 3\veps\\
\le \inf_{m \in \clp(\Delta^o)} [F(m) + I(m)] + (9+2L_G) \veps, 
\end{multline*}
where the last line is from \eqref{eq:601fin}. Since $\veps>0$ is arbitrary, the proof of Theorem \ref{thm:lowbd} is complete.
  \hfill \qed

 \section{Compactness of Level Sets}\label{sec:levelset}
 In this section we show that the function $I_A$ defined in \eqref{def:ratefunction1}, for each fixed $A \in \cla$,  is a rate function.
  
\begin{proposition}\label{lem:compactlevelsets}
{For each $A \in \cla$, the  function $I_A$ defined in \eqref{def:ratefunction1} is a rate function. Namely, for each $k \in (0,\infty)$, the set
 $S_k = \{m \in \clp(\Delta^o): I_A(m)\le k\}$ is compact in $\clp(\Delta^o)$.}
 \end{proposition}
 \begin{proof}
    {Fix $A \in \cla$. Since $A$ is fixed, we write $I$ in place of $I_A$.} Since Let $\{m_n, \; n \in \NN\}$ be a sequence in $S_k$. Since $\clp(\Delta^o)$ is compact, $\{m_n, \; n \in \NN\}$ converges along a subsequence to some limit point $m \in \clp(\Delta^o)$. It suffices to show that $m \in S_k$. Since $m_n \in S_k$, for each $n \in \NN$ we can find $\eta^n \in \clu(m_n)$
 such that 
 \begin{equation}\label{eq:935}
 \int_0^{\infty} \exp(-s) R\left(\eta^n( s) \| \eta^n_{(1)} ( s) \otimes G(M^n(s)) \right) ds
 \le I(m_n) +n^{-1}  \le k +n^{-1},
 \end{equation}
 where $M^n$ solves $\clu(m^n, \eta^n)$. For each $n \in \NN$, define  $\hat \lambda^n, \hat \rho^n \in \clp( \Delta^o \times \Delta^o \times \RR_+)$ as, for $t \in \RR_+$ and $x,y \in \Delta^o$,
 \begin{align*}
 \hat \lambda^n(\{x\}\times \{y\}\times [0,t]) &= \int_0^t \exp(-s) \eta^n(x, y \mid s) ds\\
 \hat \rho^n(\{x\}\times \{y\}\times [0,t]) &= \int_0^t \exp(-s) \eta^n_{(1)} (x \mid s)G(M^n(s))(x, y) ds.
 \end{align*}
 Since $\Delta^o$ is compact and $\hat \lambda^n_{(3)}(ds) = \hat \rho^n_{(3)}(ds) = \exp(-s) ds$ for each $n \in \NN$, it follows that the sequences $\{\hat \lambda^n, \; n\in \NN\}$, $\{\hat \rho^n, \; n \in \NN\}$ are tight in 
 $\clp(\Delta^o  \times \Delta^o \times \RR_+)$. Consider a further subsequence (of the subsequence along which $m^n$ converges) along which $\hat \lambda^n$
 and $\hat \rho^n$ converge to $\hat \lambda$ and $\hat \rho$, respectively, and relabel this subsequence once more as $\{n\}$.
 Note that, for each $n \in \NN$, since $M^n$ solves $\clu(m^n, \eta^n)$, we have, for $t \in \RR_+$,
 $$M^n(t) = m^n - \int_0^t \eta^n_{(1)}(s) ds + \int_0^t M^n(s) ds.$$
 A straightforward calculation shows that, for each $n \in \NN$, 
 $\|M^n(t) - M^n(s)\| \le 2 (t-s)$ for all $0 \le s \le t <\infty$,
 from which it follows that $\{M^n, \; n \in \NN \}$ is relatively compact in
 $C(\RR_+:\clp(\Delta^o))$. Assume without loss of generality (by selecting a further subsequence if needed) that $M^n \to M$ in 
 $C(\RR_+:\clp(\Delta^o))$ as $n\to \infty$.  Note that we can write, for  $t \in \RR_+$ and $x \in \Delta^o$,
 $$M^n(t)(x) = m^n(x) - \int_0^t \exp(s) \hat \lambda^n_{(1,3)}(\{x\} \times  ds) + \int_0^t M^n(s)(x) ds.$$
 Sending $n \to \infty$ in the previous display, we get, for each $t \in \RR_+$,
 \begin{equation}\label{eq:923}
 M(t)(x) = m(x) - \int_0^t \exp(s) \hat \lambda_{(1,3)}(\{x\} \times  ds) + \int_0^t M(s)(x) ds.
 \end{equation}
 Furthermore, since $\hat \lambda_{(3)}(ds) = \exp(-s) ds$, we can disintegrate $\hat \lambda$ as 
 \begin{equation}\label{eq:disinhatgam}
 \hat \lambda(\cdot \times ds) = \hat\eta(\cdot\mid s) \exp(-s) ds,
 \end{equation}
 where
 $s \mapsto \hat\eta(s) \doteq \hat \eta(\cdot \mid s)$ is a measurable map from $\RR_+$ to $\clp(\Delta^o\times \Delta^o)$. 
 Also, for $x,y \in \Delta^o$ and $s \in \RR_+$, we can disintegrate $\hat \eta(s)(x,y)$ as $\hat \eta_{(1)}(x\mid s) \hat \eta_{2|1}(x,y \mid s)$.
 With this observation and \eqref{eq:923}, we have, for $t \in \RR_+$,
 $$
 M(t) = m - \int_0^t  \hat \eta_{(1)}(s) + \int_0^t M(s) ds.
 $$
Since $\eta^n \in \clu(m^n)$, we have that, with 
$$\hat \beta^n(\{x\} \times \{y\}\times [0,t]) \doteq \int_0^t  \eta^n(x,y \mid s) ds, \; x,y \in \Delta^o, t \in \RR_+,$$
 \eqref{P1} holds with $\beta$ replaced with $\hat \beta^n$ for all $n \in \NN$. Letting
 $$\hat \beta(\{x\} \times \{y\}\times [0,t]) \doteq \int_0^t \hat \eta(x,y \mid s) ds, \; x,y \in \Delta^o, t \in \RR_+,$$
 we have on sending $n \to \infty$, and recalling the convergence $\hat \lambda^n \to \hat \lambda$, that \eqref{P1} holds with $\beta$ replaced with $\hat \beta$. Consequently, Property \ref{prop:z1}(a) holds.

 Next, since $\hat \eta^n \in \clu(m^n)$, we have that, for each $n \in \NN$, \eqref{P2} holds with $\eta$
 replaced with $\eta^n$. This says that, for each $n \in \NN$,
 $$\hat \lambda^n(\{x\} \times \Delta^o \times [0,t]) =
 \hat \lambda^n(\Delta^o \times \{x\} \times  [0,t]), \;\; t \in \RR_+, x \in \Delta^o.$$
 Sending $n\to \infty$, recalling the convergence $\hat \lambda^n \to \hat \lambda$, and the definition of $\hat \eta$, we now see that \eqref{P2} holds with $\eta$
 replaced with $\hat \eta$ as well, thereby ensuring that Property \ref{prop:z1}(b)  holds.
 
 Next note that, since $\eta^n \in \clu(m^n)$, for each $n \in \NN$, there is some $\MP^n \in C([0, \infty): \clp(\Delta^o\times \Delta^o))$ such that Property \ref{prop:z1}(c) holds  with $(\MP , M)$ replaced with $(\MP^n, M^n)$. Note that this in particular says that $\{\MP^n, \; n \in \NN\}$ is tight in $C(\RR_+: \clp(\Delta^o\times \Delta^o))$. Thus, by considering a further subsequence if needed, we can assume without loss of generality that $\MP^n$ converges to $\MP$
 as $n \to \infty$ in $C(\RR_+: \clp(\Delta^o\times \Delta^o))$.
 It is easily checked that Property \ref{prop:z1}(c)  holds for $(\MP, M)$.

 Together, the above observations say that $\hat \eta$ satisfies Property \ref{prop:z1}, showing that
 \begin{equation}\label{eq:940}
 \hat \eta \in \clu(m) \mbox{ and } M \mbox{ solves } \clu(m, \hat \eta).
 \end{equation}
 
 Next, from the definition of $\{ (\hat \rho^n, \hat \lambda^n), \; n \in \NN\}$ and \eqref{eq:935}, we see that $R(\hat \lambda^n \| \hat \rho^n) \le k+1/n$ for each $n \in \NN$. Using the fact that  
 $(\hat \lambda^n, \hat \rho^n) \to (\hat \lambda, \hat \rho)$ as $n \to \infty$ and the lower semicontinuity of relative entropy, we have on sending $n\to \infty$ that
 $R(\hat \lambda \| \hat \rho) \le k$.
 We now claim that
 \begin{equation}\label{eq:33193319}
 \hat \rho(\{x\}\times \{y\}\times [0,t]) = \int_0^t \exp(-s) \hat\eta_{(1)} (x \mid s)G(M(s))(x, y) ds, \; t \in \RR_+,\; x,y \in \Delta^o.
 \end{equation}
 Fix $t \in \RR_+$ and $x,y \in \Delta^o$. Since $\hat \lambda^n_{(1,3)} \to \hat \lambda_{(1,3)}$  as $n \to \infty$, we have on using the continuity of $G$ and $M$ that, as $n \to \infty$,
 $$
 \int_0^t \exp(-s) \hat\eta^n_{(1)} (x \mid s)G(M(s))(x, y) ds
 \to \int_0^t \exp(-s) \hat\eta_{(1)} (x \mid s)G(M(s))(x, y) ds.
 $$
 Also,
 $$
 \int_0^t \exp(-s) \hat\eta^n_{(1)} (x \mid s)|G(M^n(s))(x, y) -G(M(s))(x, y)|  ds
 \le L_G \sup_{0\le s \le t} \|M^n(s) - M(s)\| \to 0, \; \mbox{ as } n \to \infty.
 $$
 Combining the last two observations we have that, as $n \to \infty$,
 \begin{align*}
 \hat \rho^n(\{x\}\times \{y\}\times [0,t]) &= \int_0^t \exp(-s) \eta^n_{(1)} (x \mid s)G(M^n(s))(x, y) ds\\
 &\to 
 \int_0^t \exp(-s) \eta_{(1)} (x \mid s)G(M(s))(x, y) ds.
 \end{align*}
 Combining this with the fact that $\hat \rho^n \to \hat \rho$ as $n \to \infty$, we now have \eqref{eq:33193319}.
 Finally, on combining \eqref{eq:disinhatgam} with \eqref{eq:33193319} and using the chain rule for relative entropies, we have
 $$
 k \ge R(\hat \lambda \| \hat \rho) = \int_0^{\infty} \exp(-s) 
 R\left(\hat\eta(\cdot \mid s) \| \hat\eta_{(1)} (\cdot \mid s) \otimes G(M(s))(\cdot, \cdot)\right) ds.
 $$
 Combining this with \eqref{eq:940} we now see that $I(m) \le k$. The result follows. 
 \end{proof}

 \section{Examples}\label{sec:examp}
 In Section \ref{sec:statres} (see Example \ref{ex:operatorexamples}) we presented one important setting where the conditions of Theorem \ref{thm:ldp1} are met.
 In this section we provide several other examples for which Theorem \ref{thm:ldp1} holds.

	 \begin{enumerate}
	\item \label{ex:sec8example1} Suppose that $P^o \in \clk(\Delta^o)$ is irreducible and $G(m) = P^o$ for all $m \in \Delta^o$. Clearly this $G$ satisfies Assumption \ref{ass1} with $A$ defined as  {$A_{x,y} = \bm{1}_{\{P^o(x,y)>0\}}$}.
 Theorem \ref{thm:ldp1} in this case is exactly the large deviation principle for empirical measures of irreducible finite state Markov chains (cf. \cite{donvar1,  donvar3}). To see that the rate function $I$ given in \eqref{def:ratefunction1} coincides  with the well known formula \eqref{eq:donvarrf},
	we note the following. 
	The inequality $I(m) \le \tilde I(m)$ was argued in Remark \ref{rem:ivstili}.
 Conversely, suppose that $m \in \clp(\Delta^o)$ is such that $I(m) < \infty$ and that we are given a $\eta \in \clu(m)$.
	Define $\lambda \in \clp(\Delta^o\times \Delta^o)$ as $\lambda \doteq \int_0^{\infty} \exp(-s) \eta(s) ds$ and observe that
	 $\lambda_{(1)} = \lambda_{(2)}$. Also, if $M$ solves $\clu(m, \eta)$, then it is easily checked 
	by multiplying both sides of \eqref{P3}  by $\exp(-t)$ and integrating over $[0, \infty)$ that
	$m = \int_0^{\infty} \exp(-s) \eta_{(1)}(s) ds = \lambda_{(1)}$, namely $\lambda \in \cli(m)$, where $\cli(m)$ is defined below \eqref{eq:donvarrf}. Finally, from the convexity of relative entropy  
	\begin{align*}&\int_0^{\infty} \exp(-s) R (\eta( s)\| \eta_{(1)}( s) \otimes P^o) ds\\
	&\quad \ge R\left (\int_0^{\infty} \exp(-s) \eta( s) ds\|  \int_0^{\infty} \exp(-s) \eta_{(1)}(  s) \otimes P^o  ds\right)\\
	&\quad = R(\lambda \| m \otimes P^o)
\end{align*}
	which shows that $\tilde I(m) \le I(m)$. This proves that $I=\tilde I$. Note that when $P^o$ is replaced with $G(\cdot)$ (with a general $G$), one cannot carry out a similar convexity argument.
	 \item \label{ex:oldmodel} 
	  Let $A$ be an irreducible  adjacency matrix. Then we have $\sum_{j=1}^d A^j>0$.
	 For each $z \in \Delta^o$, let $M^z \in \clk(\Delta^o)$ be such that $M^z(x,y)>0$  if and only if $(x,y) \in A_+$. Define $G: \clp(\Delta^o) 
	 \to \clk(\Delta^o)$ as
	 $$G(m)(x,y) \doteq \sum_{z \in \Delta^o} m(z) M^z(x,y), \;\; x,y \in \Delta^o.$$
	 Clearly, Assumption \ref{ass1} part 1 and part 2(a) are satisfied. Assumption \ref{ass1} part 2(b) is also satisfied with 
	 $\delta_0^A \doteq \min_{(x,y) \in A_+} \min_{z \in \Delta^o} M^z(x,y)$.
	 Also, since $G(m)$ is irreducible for every $m \in \clp(\Delta^o)$, from Remark \ref{rem:conditions1}(3) we see that Assumption \ref{ass1} part 3 holds. Finally, 
  since $\sum_{j=1}^d A^j>0$ it follows that
  for all $m_1, \ldots m_{d} \in \clp(\Delta^o)$, $\sum_{j=1}^d G(m_1) \cdots G(m_j)>0$. Using  Remark \ref{rem:conditions1}(3) again, we see that Assumption \ref{ass1} part 4 holds as well. Thus, this family of models satisfies all the conditions of Theorem \ref{thm:ldp1}. This model can be viewed as a { generalized P\'{o}lya urn} in the following manner. Consider an urn that contains balls of $d$ different colors. Initially there is a single ball in the urn which is of color $x_0$. At each time instant a ball is selected from the urn, and then that ball, together with a new ball (of possibly different color), is added back to the urn according to the following probabilistic rule.
	 Given that the  ball drawn at time instant $n$ is of color $z$ and the new ball added at time instant $n-1$ was of color $x$, we return the drawn ball to the urn (namely the ball with color $z$) and add a new ball to the urn of color $y$ with probability $M^z(x,y)$.
  
\item Let $M \in \clk(\Delta^o)$ be such that $M$ is irreducible. Define
$G: \clp(\Delta^o) \to \clk(\Delta^o)$ as
$$G(m)(x,y) = \sum_{z \in \Delta^o} m(z) M(z,y) = (mM)(y), \; x,y \in \Delta^o.$$
Under the condition $M(x,y)>0$ for all $x,y\in \Delta^o$, a large deviation principle of the form in Theorem \ref{thm:ldp1} was recently established in \cite{BudWat3}. The current work shows that the above strict positivity condition can be relaxed to simply the requirement that $M$ is irreducible. To see this, it suffices to verify Assumption \ref{ass1}. Clearly, part 1 of the assumption holds. Also take $A$ to be the $d\times d$ matrix with all entries $1$. Then, part 2(a) of the assumption holds (vacuously). Also, part 2(b) holds with $\delta_0^A \doteq \min_{(x,y)\in M_+} M(x,y)$, where $M_+ = \{(x,y)\in \Delta^o \times \Delta^o: M(x,y)>0\}$.
The fixed point equation in part 3 in this case reduces to the equation $\pi^*M = \pi^*$, which, since $M$ is irreducible, has a unique solution in $\clp_+(\Delta^o)$.
Finally for part 4, note that from the irreducibility of $M$, $\inf_{x,y \in \Delta^o} \sum_{k=1}^d M^k(x,y) \doteq \alpha >0$.
Now for $k \in \NN$, by a straightforward conditioning argument it follows that, for every $x \in \Delta^o$
$$P(L^{k(d+1)}(x) =0) = P(L^{k(d+1)}(x) =0, L^{kd}(x) =0) \le (1-\alpha) P(L^{kd}(x) =0).$$
So by Borel Cantelli lemma
$P(L^{k(d+1)}(x) = 0 \mbox{ for infinitely many } k) =0$, which verifies part 4 of the assumption.

\item For each $x \in \Delta^o$ 
	  let $M^x \in \clk(\Delta^o)$ be irreducible, and let $P$ and $P^o$ be as in  Example \ref{ex:operatorexamples}. Define $G: \clp(\Delta^o) \to \clk(\Delta^o)$  as
     \begin{equation}\label{eq:137nn}
     G(m)_{x,y} = P_{x,y} + P_{x,0} (mM^x)_y, \;\;  x,y \in \Delta^o,  m \in \clp(\Delta^o).
     \end{equation}
Let $A$ be as introduced in Example \ref{ex:operatorexamples}. 
 Assumption \ref{ass1} part 1 and part 2(a) are clearly satisfied. Also, Assumption \ref{ass1}  part 2(b)
 holds with $$\delta_0^A \doteq \left(\inf_{(x,y) \in A_+} (P_{x,y}+ P_{x,0}) \right) \inf_{(x,y) \in \Delta^o \times \Delta^o} \sum_{z\in \Delta^o} M^x(z,y)$$
 which is clearly positive from the definition of $A_+$ and the irreducibility assumption on each $M^x$. From the irreducibility of $P^o$ it follows that $\sum_{k=1}^d (P^o)^k$ is strictly positive. This shows that the condition \eqref{eq:1059nn} in Remark \ref{rem:conditions1}(3) is satisfied, which, in 
 view of the discussion in the same remark, shows that Assumption \ref{ass1} parts 3 and 4 hold as well.
 Thus, Theorem \ref{thm:ldp1} holds with $G$ defined as above under the assumed conditions on $P$ and $\{M^x, \; x \in \Delta^o\}$.
 
 One family of models that fits the above setting is a variant of the Personalized PageRank (PPR) algorithm, see e.g., \cite{BOURCHTEIN2016149,XieBinDemGeh} and the references therein. 
 Consider an individual performing a random walk on the graph of webpages. Denote by $\Delta^o$ the set of webpages and, for each $x,y \in \Delta^o$, let $\clv(x,y)$ denote the number of links from webpage $x$ to webpage $y$. 
 Let, for each $x \in \Delta^o$,
 $\clv(x) \doteq \{ y \in \Delta^o : \clv(x,y) > 0 \}$
 denote the set of webpages that are linked to by webpage $x$, and  assume that  $\clv(x)$ is nonempty for each $x \in \Delta^o$. For each $x \in \Delta^o$, let $D_x^+ \doteq \sum\nolimits_{y \in \clv(x)} \clv(x,y)$
 denote the out-degree of webpage $x$.
 Consider the transition kernel $Q$ on $\Delta^o$ defined as 
 $Q_{x,y} \doteq \frac{ \clv(x,y)}{D^+_x} $, $x,y \in \Delta^o$.
 For $x \in \Delta^o$, fix a \emph{damping factor} $\alpha_x \in (0,1)$   and define $G: \clp(\Delta^o) \to \clk(\Delta^o)$ as
 \[
 G(m)(x,y) \doteq (1- \alpha_x) Q_{x,y} + \alpha_x L(m)(x,y), \quad x,y \in \Delta^o,
 \]
 where $L: \clp(\Delta^o) \to \clk(\Delta^o)$ is defined as
 $$L(m) = \theta q + (1-\theta) \sum_{z \in \Delta^o} m_z M^x(z,y)$$
 for some $\theta \in (0,1]$, $q \in \clp_+(\Delta^o)$, and $M^z \in \clk(\Delta^o)$ for $z \in \Delta^o$.
 The self-interacting chain defined using the map $G$ as above,  in the special case where
 $\alpha_x = \alpha \in (0,1)$,  $\theta = 1$, and $q_y = \frac{1}{|\Delta^o|}$ for all $y \in \Delta^o$, is the well-known PageRank (PR) Markov chain.
 A limitation of classical PR is that it does not take into consideration the user's preferences. For that reason,   variants of the PR algorithm have been proposed that account for personal preferences, see, e.g. \cite{XieBinDemGeh}.
 Such variants can be captured by a $G$ of the above form that reflects an individual's browsing history in determining transition probabilities. 
 It is easy to verify that the above $G$ can be expressed in the form \eqref{eq:137nn} 
 with $P_{x,y} = (1-\alpha_x) Q_{xy} + \theta \alpha_x q_y$ for $x,y \in \Delta^o$ and $P_{x,0} = \alpha_x(1-\theta)$,
 and that,
 under the assumption that $M^x$ is irreducible for every $x \in \Delta^o$, Assumption \ref{ass1} holds.
 \item \label{ex:edgereinforced} As noted in Example \eqref{ex:operatorexamples}, our assumptions cover certain types of vertex reinforced random walks.
 We now give an example that shows that certain variants of edge reinforced random walks are also covered by our assumptions.
 Suppose that $\clg$ is a connected undirected graph on the vertex set $\clv = \{1, \ldots, \ell\}$, with the edge set denoted as $\cle$. For $x \in \clv$, we  denote by $d(x)$ the degree of vertex $x$. Let $\tilde A$ be the incidence matrix of the graph, namely it is the $\ell\times \ell$ matrix with entries $0$ or $1$ such that
 {$\tilde A_{u,v} = \tilde A_{v,u} = 1$} if and only if $\{u, v\} \in \cle$. For simplicity of presentation, we assume that the graph has no self-loops, namely the diagonal entries of $\tilde A$ are $0$. Let $\Delta^o = \{(x,y) \in \clv \times \clv: \{x,y\} \in \cle\}$.
  For $z \in \Delta^o$, $z_i$, $i=1,2$, will denote the $i$-th coordinate of $z$.
 Fix $\{x_0, y_0\}\in \cle$ so that $\tilde A_{x_0, y_0} = 1$, and let $\delta \in (0,1)$. The latter parameter will control the strength of the reinforcement mechanism.
 
 We now define a sequence $\{X_n, \; n\in \NN_0\}$ of $\clv$-valued random variables, recursively, as follows. Let $X_0=x_0$ and $X_1=y_0$, and set $Z_0 = (X_0, X_1)$.
 Having defined $\{X_i, \; 0\le i \le n\}$ and $\{Z_i \doteq (X_i, X_{i+1}), \; 0 \le i \le n-1\}$, we now define $X_{n+1}$ 
 according to the following conditional law:
 \begin{multline}
 P(X_{n+1}=y \mid X_0, \ldots X_n) \\
 \doteq {\tilde A_{X_n,y}} \left[ \delta \hat L^{n-1}[(X_n,y)] + \frac{1}{d(X_n)} \left(1-\delta \sum_{\bar z \in \Delta^o} \hat L^{n-1}(\bar z) {\tilde A_{X_n, \bar z_2}}\bm{1}_{\{X_n = \bar z_1\}} \label{eq:958nn}
 \right)\right ],	
 \end{multline}
 where, denoting by $L^{n-1} \doteq \frac{1}{n} \sum_{i=0}^{n-1} \bdelta_{Z_i}$, and for
 $z = (z_1, z_2) \in \Delta^o$, $z^r = (z_2, z_1)$,
 $$\hat L^{n-1}(z) = \frac{1}{2n} \sum_{i=0}^{n-1} [\bdelta_{Z_i} (z) + \bdelta_{Z_i} (z^r)] = \frac{1}{2} (L^{n-1}(z) + L^{n-1}(z^r)), \; z \in \Delta^o.$$
  Now set $Z_n = (X_n, X_{n+1})$. 
 The above conditional law can be interpreted as follows. At each time instant $n\ge2$, for each neighboring site $y$, the walker jumps to site $y$ with probability $\frac{\delta}{2}$ times the fraction of time the edge connecting with that site has been traversed  (in either direction) by the walker by time $n-1$; and with the remaining probability it  selects one of the neighboring sites (including $y$) at random. Thus, the first term on the right side of \eqref{eq:958nn} captures the edge-reinforcement  mechanism.
 It is convenient to  directly describe the evolution of the sequence $\{Z_n, \; n \in \NN_0\}$. With $d\doteq |\Delta^o|$, define the $d\times d$ dimensional incidence matrix $A$ as {$A_{z, \tilde z}=1$} if and only if $z_2 = \tilde z_1$ and {$\tilde A_{\tilde z_1, \tilde z_2}=1$}. Since the graph is connected, $A$ is irreducible.
 Then, in terms of $A$, the conditional law of $Z_{n}$ can be written as
 $$P(Z_{n}=\tilde z \mid Z_0, \ldots , Z_{n-1}=z) = G(L^n)(z, \tilde z), \; z, \tilde z \in \Delta^o,$$
 where for $m \in \clp(\Delta^o)$,
 \begin{equation}
G(m)(z, \tilde z) \doteq {A_{z, \tilde z}}\left[ \delta \hat m(\tilde{z})
 + \frac{1}{d(z_2)}   \left(1-\delta \sum_{\bar z \in \Delta^o} \hat m(\bar z) {A_{z, \bar z}}\right)\right]
\end{equation}
and $\hat m(z) = \frac{1}{2}(m(z) + m(z^r))$.We now verify that  Assumption \ref{ass1} holds. Part 1 and Part 2(a) of the assumption clearly hold with the above definition of $A$. Also, since $d(z_2) \le \ell$ and $\sum_{\bar z \in \Delta^o} \hat m(\bar z) {A_{z, \bar z}} \le 1$, Part 2(b) holds with $\delta_0^A = (1-\delta)/\ell$. This observation, together with the fact that 
$A$ is irreducible,  also shows that $G(m)$ is irreducible for every $m \in \clp(\Delta^o)$ and in fact, for every $j \in \NN$
 and all  $m_1, \ldots , m_j \in \clp(\Delta^o)$,
     $G(m_1)G(m_2)\cdots G(m_j) \ge (\delta_0^A)^j A^j$, coordinate wise. These observations, in view of 
Remark \ref{rem:conditions1} (3) show that parts 3 and 4 of the assumption are satisfied as well.
Thus, Theorem \ref{thm:ldp1} holds with $G$ defined as above for the sequence $\{Z_n , \; n \in \NN_0\}$.
Note that the empirical measure $L^{n,X} \doteq \frac{1}{n+1} \sum_{i=0}^n \bdelta_{X_i}$ can be obtained from $L^{n}$ using the relation $L^{n,X}(x) = \sum_{y \in \Delta^o} L^n(x, y)$, $x \in \Delta^o$, and so, by using the contraction principle, one  also obtains a large deviation principle for $\{L^{n,X}, \; n \in \NN\}$.

 \end{enumerate}
 
 \appendix

  \section{Some Auxiliary Results}
  \begin{lemma}
  \label{lem:posit}
  Let $\{L^{k,Z}, \; k \in \NN\}$ be the sequence introduced in Section \ref{sec:constructcontrols} and $\veps_1$ be as fixed in 
  \eqref{eq:conscho}. Then, under Assumption \ref{ass1}, there is an $a^*>0$ and $r_1 \in \NN$ such that $P(N_1 > r_1) \le \veps_1$, where $N_1$ is as defined in \eqref{eq:255nn}.
   \end{lemma}
 \begin{proof}
 From Assumption \ref{ass1}(4) it follows that, with
 $M(\om) \doteq \inf\{n \in \NN: \inf_{x\in \Delta^o} L^{n,Z}(\om)(x) >0\}$, we have $P(\om: M(\om)<\infty)=1$.
 For $a \in \RR_+$, let $M^a(\om) \doteq \inf\{n \in \NN: \inf_{x\in \Delta^o} L^{n,Z}(\om)(x) >a\}$.
 Note that  $\{M<\infty\} = \cup_{k=1}^{\infty} \{M^{1/k}<\infty\}$. 
 Thus, there exists an $a^*>0$ such that $P(M^{a^*}<\infty) > 1 - \veps_1/2$.
 Since $\cup_{m=1}^{\infty} \{M^{a^*}<m\} =\{M^{a^*}<\infty\}$, we can find an $r_1 \in \NN$ such that $P(M^{a^*} \le r_1) > 1-\veps_1$.
 The result follows on noting that $N_1 = M^{a^*}$.
 \end{proof}
  
   \begin{lemma}
   \label{lem:irred}
   Suppose that $G$ satisfies Assumption \ref{assu:lip} and for  some $K \in \NN$ and all $m_1, \ldots , m_K \in \clp(\Delta^o)$, and $x,y \in \Delta^o$,
    $\sum_{j=1}^K[G(m_1)G(m_2)\cdots G(m_j)]_{x,y}>0$. Then Assumption \ref{ass1}(4) is satisfied.
    \end{lemma}
    
     \begin{proof}
     By continuity of $G$ and compactness of $\clp(\Delta^o)$
     $$\inf_{x,y \in \Delta^o} \inf_{m_1, \ldots m_K \in \clp(\Delta^o)} \sum_{j=1}^K [G(m_1)G(m_2)\cdots G(m_j)]_{x,y} \doteq \tilde{\veps} >0.$$
     Then by a straightforward conditioning argument it follows that, for any $x \in \Delta^o$, and $n>1$,
     \begin{align*}
     P(L^{nK, X}(x) = 0) = P(L^{(n-1)K, X}(x) = 0, L^{nK, X}(x) = 0) \le (1-\tilde{\veps}) P(L^{(n-1)K, X}(x) = 0).
     \end{align*}
     Thus, $P(L^{nK, X}(x) = 0) \le (1-\tilde{\veps})^{n-1}$ and so 
     the result follows from the Borel-Cantelli lemma.

      \end{proof}
{
The following chain rule for relative entropies is well known (cf. \cite[Corollary 2.7]{buddupbook}).
      \begin{theorem}[{\bf Chain rule for relative entropies}]\label{thm:chainrule}
         Let   $\mathcal{X}$ and $\mathcal{Y}$ be Polish spaces, and let $\sigma( dy \mid x)$ and  $\tau(dy \mid x)$ be transition kernels on $\mathcal{Y}$ given $\mathcal{X}$. Then, for each probability measure $\theta$ on $\mathcal{X}$, the function mapping $x \in \mathcal{X} \mapsto R(\sigma(\cdot \mid x) \| \tau(\cdot \mid x))$ is measurable, and 
         \begin{align*}
             \int_{\mathcal{X}} R( \sigma(\cdot \mid x) \| \tau(\cdot \mid x))\theta(dx)  = R(\theta \otimes \sigma \| \theta \otimes \tau).
         \end{align*}
      \end{theorem}
 }
\section{Commonly Used Notation}\label{sec:notationtable}

\PZ{
\subsection{Notation Used Primarily in Section \ref{sec:controlproc} and \ref{sec:lub}}\label{sec:notationtablepart1}

\renewcommand{\arraystretch}{1.5}
\begin{tabular}{|c|p{0.9\textwidth}|}
\hline
\textbf{Symbol} & \textbf{Description} \\
\hline
$X_n$      & self-interacting Markov chain, \eqref{eq:rmcdynamic1} \\
\hline
$L^n $ & empirical measure, \eqref{eq:occmzr} \\
\hline
$\nu^k$ & iid $\clv^d$-valued random field, above \eqref{eq:defL}  \\
\hline
$\bar{\mu}^{n,k}$ &  control measure, beginning of Section \ref{sec:contrepx} \\
\hline
$\bar{\nu}^{n,k}$ & control  with conditional law $\bar{\mu}^{n,k}$,  \eqref{eq:condlawnu}\\
\hline 
$\bar{X}^n_k$ &  controlled analogue of $X^n_k$, below \eqref{eq:condlawnu}\\
\hline 
$\bar{L}^{n,k}$ & controlled analogue of $L^n$, \eqref{eq:barlnk} \\
\hline
$\bar{L}^{n,k}_{\text{pair}}$ & two-step controlled analogue of $L^n$, \eqref{eq:bartnk2} \\
\hline
$\bar\Lambda^n$ & $\bar \Lambda^n(\cdot \mid t) = \bdelta_{\bar \nu^{n,k+1}}(\cdot)$,  \eqref{eq:576new}\\
\hline
$\bar{\xi}^n$ & $\bar \xi^n(\cdot \mid t) = \bdelta_{\bar \nu^{n,k}} \otimes  \bar \mu^{n,k+1}(\cdot)$,  \eqref{eq:576new} \\
\hline
$\bar{\Xi}^n$ & $\bar \Xi^n(\cdot \mid t) = \bdelta_{\bar \nu^{n,k+1}}  \otimes \bdelta_{\bar \nu^{n,k+2}}(\cdot)$,  \eqref{eq:576new} \\
\hline
$\bar\zeta^n$ &
	$\bar \zeta^n\left(\cdot\mid t,m\right) = \bdelta_{\bar \nu^{n,k}} \otimes G^{\clv}(m)(\cdot)$, \eqref{eq:zetadef3} \\
\hline 
$\lambda^n$ & $\lambda^n(\{\bdelta_x\} \times [0,t]) = n^{-1} \int_{0}^{t_n \wedge t} \psi_e(t_n-s) \bar \Lambda^n(\bdelta_x \mid t_n-s) ds$, \eqref{eq:525}\\
\hline
$\beta^n$ & $\beta^n(\{\bdelta_x\} \times \{\bdelta_y\}\times [0,t]) =  n^{-1} \int_{0}^{t_n \wedge t} \psi_e(t_n-s)\bar \xi^n((\bdelta_x, \bdelta_y)  \mid t_n-s) ds$, \eqref{eq:525} \\
\hline
$\rho^n$ & $\rho^n(\{\bdelta_x\} \times \{\bdelta_y\}\times [0,t])=  n^{-1} \int_{0}^{t_n \wedge t}  \psi_e(t_n-s) \bar \zeta^n
\left((\bdelta_x, \bdelta_y)  \mid t_n-s, \bar L^n(a(t_n-s))\right) ds$, \eqref{eq:525}\\
\hline
$\check\bfL^n$ &  time-reversal of $\bar{L}^{n}$, \eqref{eq:526}\\
\hline
$\check \bfL^n_{\text{pair}}$ &  time-reversal of $\bar{L}^{n}_{\text{pair}}$, \eqref{eq:526} \\
\hline
$\check \Lambda^n$ & time-reversal of $\bar\Lambda^n$, \eqref{eq:527}  \\
\hline
$\check \Xi^n$ & time-reversal of $\bar{\Xi}^n$, below \eqref{eq:527} \\
\hline
\end{tabular}
}

\PZ{
    \subsection{Notation Used Primarily in Section \ref{sec:llb}}\label{sec:notationtablepart2}

\renewcommand{\arraystretch}{1.5}
\begin{tabular}{|c|p{0.9\textwidth}|}
\hline
\textbf{Symbol} & \textbf{Description} \\
\hline
$T$  & sufficiently large length of time,  \eqref{eq:sizeofT}\\
\hline
$\hat \eta$ &  piecewise constant   control, \eqref{eq:mhat3defn} and  \eqref{eq:etahatmhatcq}\\
\hline
$\hat M$ & trajectory associated with $\hat \eta$, \eqref{eq:mhat3defn} and  \eqref{eq:etahatmhatcq}  \\
\hline
$q$ & $q = M^1(T)$,  Section \ref{subsec:non-degen} and \eqref{eq:etahatmhatcq} \\
\hline
$Q$ & irreducible  transition kernel  with stationary distribution $q$, \eqref{eq:Qdef1}\\
\hline
$\beta^j$ & transition kernel given by $\beta^j(x,y) = \hat \eta(x,y \mid cj)$, \eqref{eq:betjdef4343}\\
\hline
\end{tabular}
}

  \section*{Acknowledgments}
 {We thank the two referees for a careful review of this work which led to a substantial improvement in the presentation of the results.} AB was supported in part by the NSF (DMS-2152577, DMS-2134107).  PZ was supported in part by a dissertation completion fellowship from UNC's graduate school. Later work of PZ was funded by the Deutsche Forschungsgemeinschaft (DFG, German Research Foundation) under Germany's Excellence Strategy EXC 2044 –390685587, Mathematics Münster: Dynamics–Geometry–Structure.

 \bibliographystyle{plain}
\bibliography{cas-refs}

\vspace{\baselineskip}

\scriptsize{\textsc{\noindent A. Budhiraja \newline
Department of Statistics and Operations Research\newline
University of North Carolina\newline
Chapel Hill, NC 27599, USA\newline
email:  budhiraj@email.unc.edu
\vspace{\baselineskip} }

\textsc{\noindent A. Waterbury\newline
Department of Mathematics,\newline
 Denison University\newline
 Granville,  OH 43023, USA\newline
email: waterburya@denison.edu \vspace{\baselineskip}  }

\scriptsize{\textsc{\noindent P. Zoubouloglou\newline
Institute for Mathematical Stochastics \newline
University of Münster\newline
Münster, 48149 Germany \newline
email:  p.zoubouloglou@uni-muenster.de }
}
  
\end{document}